\pdfoutput=1
\documentclass[aip,
jmp,
 amsmath,amssymb,
reprint, onecolumn 
]{revtex4-1}

\usepackage{url}
\usepackage{breakurl}
\usepackage[breaklinks]{hyperref}

\usepackage{graphicx}
\usepackage{dcolumn}
\usepackage{bm}

\usepackage[utf8]{inputenc}
\usepackage[T1]{fontenc}
\usepackage{mathptmx}
\usepackage{etoolbox}

\usepackage{subcaption}

\usepackage{amsthm}
\newtheorem{prop}{Proposition}
\newtheorem{theorem}{Theorem}
\newtheorem{lem}{Lemma}
\newtheorem{defn}{Definition}

\usepackage{mathrsfs}
\DeclareMathOperator{\vvec}{\operatorname{vec}}
\newcommand{\ba}{\mathbf{a}}

\newcommand{\bx}{\mathbf{x}}
\newcommand{\br}{\mathbf{r}}
\usepackage{theoremref}

\DeclareMathOperator{\rank}{rank}
\DeclareMathOperator{\trace}{trace}

\newcommand{\htwo}{\ensuremath{\text{H}_2}}
\newcommand{\heh}{\ensuremath{\text{HeH}^+}}
\usepackage[dvipsnames]{xcolor}

\makeatletter
\def\@email#1#2{%
 \endgroup
 \patchcmd{\titleblock@produce}
  {\frontmatter@RRAPformat}
  {\frontmatter@RRAPformat{\produce@RRAP{*#1\href{mailto:#2}{#2}}}\frontmatter@RRAPformat}
  {}{}
}%
\makeatother
\begin{document}

\preprint{AIP/123-QED}

\title[Propagating 1-Electron Reduced Density Matrices]{Incorporating Memory into Propagation of 1-Electron Reduced Density Matrices}

\author{Harish~S. Bhat}
 \email{hbhat@ucmerced.edu}
 \affiliation{Department of Applied Mathematics, University of California, Merced, CA 95343, USA}
 
\author{Hardeep Bassi}
\affiliation{Department of Applied Mathematics, University of California, Merced, CA 95343, USA}

\author{Karnamohit Ranka}
\affiliation{Department of Chemistry \& Biochemistry, University of California, Merced, CA 95343, USA}

\author{Christine~M. Isborn}
\affiliation{Department of Chemistry \& Biochemistry, University of California, Merced, CA 95343, USA}

\begin{abstract}
For any linear system with unreduced dynamics governed by invertible propagators, we derive a closed, time-delayed, linear system for a reduced-dimensional quantity of interest.  This method does not target dimensionality reduction: rather, this method helps shed light on the memory-dependence of $1$-electron reduced density matrices in time-dependent configuration interaction (TDCI), a scheme to solve for the correlated dynamics of electrons in molecules.   Though time-dependent density functional theory has established that the $1$-electron reduced density possesses memory-dependence, the precise nature of this memory-dependence has not been understood.  We derive a symmetry/constraint-preserving method to propagate reduced TDCI electron density matrices.   In numerical tests on two model systems ($\htwo$ and $\heh$), we show that with sufficiently large time-delay (or memory-dependence), our method propagates reduced TDCI density matrices with high quantitative accuracy.  We study the dependence of our results on time step and basis set.  To implement our method, we derive the $4$-index tensor that relates reduced and full TDCI density matrices.  Our derivation applies to any TDCI system, regardless of basis set, number of electrons, or choice of Slater determinants in the wave function.
\end{abstract}


\maketitle

\section{Introduction}
\label{sec:intro}

For an $N$-electron system, the time-dependent Schr\"odinger equation (TDSE) is a partial differential equation for the system's electronic wave function $\Psi$. Ignoring spin and assuming fixed nuclear coordinates within the Born-Oppenheimer approximation, $\Psi$ depends on $3N$ spatial variables and time.  This electronic wave function governs the electron dynamics of the system, important in ultrafast charge transfer in the initial steps of photosynthesis and laser-induced processes, among others \cite{Jortner1990,Kling2008,Provorse2016,Li2020}. Unless $N$ is very small, we cannot expect to solve for $\Psi$.  Within a finite-dimensional basis set, a numerically exact solution of the time-independent problem is possible with configuration interaction (CI) 
\cite{Shavitt1977,mcweeny1989methods,szabo2012modern,Shavitt1998}.  CI expands $\Psi$ in a basis of Slater determinants; TDCI uses the same expansion but allows the coefficients to depend on time, enabling the solution of genuinely time-dependent problems---see Sections \ref{sect:background} and \ref{sect:methods} for details.  When all possible determinants are used, as in the present paper, we refer to the resulting $\Psi$ as a \emph{full CI} wave function.  

The full CI wave function was recently calculated for the C$_3$H$_8$ molecule ($N=26$), which used distributed computing to determine the energy of the system with 1.31 trillion determinants, the largest known full CI calculation to date \cite{Gao2024}. The full CI wave function can be approximated with various levels of truncation of the expansion of determinants, as is commonly done for molecular systems.  In the limit where we use only a single Slater determinant to approximate $\Psi$, we have Hartree-Fock theory, which neglects electron correlation \cite{lowdin1995historical}.

Note that $\Psi = \Psi(\br_1, \br_2, \ldots, \br_N, t)$ where each $\br_j$ is three-dimensional.  Let us now define the $1$-electron reduced density $\rho(\br,t)$ as $N$ times the integral of $|\Psi(\br, \br_2, \ldots, \br_N, t)|^2$ with respect to $\br_2, \ldots, \br_N$. Clearly, $\rho$ is a much lower-dimensional object than $\Psi$.  Many observables of practical interest can be computed easily and explicitly from $\rho$; additionally, given any physical observable, there exists \cite{Ullrich2012} a mapping expressing it as a function of $\rho$.

Time-dependent density functional theory (TDDFT), which can be derived\cite{Ullrich2012} from the TDSE, gives us an alternative route to compute $1$-electron reduced densities $\rho$ \emph{without} computing the full wave function $\Psi$.  For the same $N$-electron system, 
TDDFT yields computationally tractable evolution equations in three spatial variables plus time; solving these equations enables us to compute the $1$-electron reduced density $\rho(\br,t)$.  The Runge-Gross theorem establishes the existence of exchange-correlation potentials $V_\text{xc}$ such that the $1$-electron densities computed via TDDFT exactly match those computed from solutions of the TDSE\cite{runge1984density}.  Thus, in principle, TDDFT provides an accurate, scalable path to computing electron dynamics. Note that $V_\text{xc}$ is shorthand for $V_\text{xc}[\rho](\br,t)$---this potential depends \emph{functionally} on the $\rho$ trajectory $\{ \rho(\mathbf{w},s) \mid \mathbf{w} \in \mathbb{R}^3, s \leq t \}$.  Because exact, time-dependent exchange-correlation functionals are unavailable, TDDFT practitioners use approximations\cite{Gross2012} in which $V_\text{xc}[\rho](\br,t)$ depends only on $\{ \rho(\mathbf{w},t) \mid \mathbf{w} \in \mathbb{R}^3\}$, ignoring memory-dependence on $\rho(\cdot,s)$ for $s < t$.  This ubiquitous approximation, known as the adiabatic approximation within the TDDFT community, leads to errors that have been quantified in prior work \cite{elliott2012universal,Rappoport2012, Habenicht2014, Provorse2015, lacombe2020developing, Ranka2023}.

One way to devise more accurate $V_\text{xc}$ models that incorporate memory-dependence is to use machine learning.  For a spatially one-dimensional, two-electron model problem, the approach of learning memory-dependent $V_\text{xc}$ models from data has been pursued recently \cite{Suzuki2020,bhat2021dynamic}.  To scale these machine learning methods to larger, three-dimensional molecular systems, it will be essential to incorporate prior  knowledge regarding the memory-dependence of $V_\text{xc}$.  This will enable targeted searches over neural network architectures that match physically relevant dynamics.  In TDDFT, memory-dependence of $V_\text{xc}$ results in time-delayed equations of motion for $1$-electron Kohn-Sham orbitals \cite{maitra2002memory}.  The sum of the squared magnitudes of all such $1$-electron Kohn-Sham orbitals equals the $1$-electron reduced density $\rho$.  Memory-dependence of $V_\text{xc}$ gives rise to memory-dependence of $\rho$.

Let $P(t)$ and $Q(t)$ denote, respectively, the full density matrix and $1$-electron reduced density matrix (1RDM) that can be computed from the TDCI solution $\Psi$. 
 In this paper, we use all possible Slater determinants in our CI expansion for $\Psi$.  Then the resulting $\Psi$, $P$, and $Q$ all incorporate the effects of electron correlation and exchange.  Our TDCI approach circumvents the approximations of methods such as time-dependent Hartree-Fock and TDDFT.  Therefore, we can logically aim to \emph{improve} a theory such as TDDFT by studying the dynamics of TDCI $1$-electron reduced density matrices $Q(t)$.  The missing piece has been an equation of motion for $Q(t)$.

In this paper, we derive and study a closed, time-delay system that enables propagation of TDCI $1$-electron reduced density matrices $Q(t)$. 
 As this system requires the full many-body propagator, it does not achieve computational savings.  Instead, we use the system to generate insight into the explicit memory-dependence of $Q(t)$ on past values $Q(s)$ for $s < t$.  As we explain in Section \ref{sect:methods}, the TDCI $1$-electron reduced density matrices $Q(t)$ are closely related to to the  $1$-electron reduced densities $\rho$ in TDDFT.  Therefore, insight into the memory-dependence of the TDCI reduced electron density matrices $Q(t)$, enabled by the methods of this paper, will assist us in improving future models of $V_\text{xc}$ for TDDFT.

\subsection{Related Work}
\label{sect:related_work}
 The idea of trading high-dimensionality of one dynamical system (\ref{eqn:motivlinsys}) for time-delayed dynamics in a reduced-dimensional space (\ref{eqn:motivlinsys2}) is connected to notions from Mori-Zwanzig theory \cite{Zwanzig1961,Mori1965,chorin2000optimal,Darve2009,Chorin2014}, Takens embedding theory \cite{takens2006detecting,Stark1999,BroerHenk2010DSaC}, and Koopman operator theory \cite{MauroyAlexandre2020TKOi,OttoRowley2021,Brunton2022}.  We now review these connections.
 
Given a dynamical system in a high-dimensional space, Mori-Zwanzig theory produces closed systems for the time-evolution of low-dimensional observables \cite{Zwanzig1961,Mori1965}.  The resulting equations trade (i) non-time-delayed dynamics in the original, high-dimensional space for (ii) time-delayed dynamics in a low-dimensional space \cite{Chorin2014}.  The time delays are expressed in terms of a memory kernel.  Recent work has connected Mori-Zwanzig theory to Koopman operator methods, showing how to learn the memory kernel from data \cite{Lin2021}.

Our method shares with Mori-Zwanzig theory the trade-off between (i) and (ii).  A key difference is that Mori-Zwanzig methods often assume that the unobserved or irrelevant variables can be treated as noise, leading to generalized Langevin equations \cite{chorin2000optimal,Darve2009}.  If we view the transformation from full electron density matrices $P(t)$ to $1$-electron reduced density matrices $Q(t)$ as projection onto a subspace $S$, then the unobserved/irrelevant variables are the projection of $P(t)$ onto the orthogonal complement $S^\perp$.  We make no assumptions regarding the dynamics of these variables; they do not appear at all in our self-contained scheme for propagating $Q(t)$.  This is a key difference between our work and Mori-Zwanzig approaches.  In Section \ref{sect:pols}, we explain this difference further through a concrete example.

Takens justified trading (i) for (ii) in the sense of establishing a sufficient criterion for how much memory is required (i.e., the number of time-delayed scalar observations) to accurately represent dynamics in the original phase space \cite{takens2006detecting}.  Takens' theory establishes the existence of time-delayed embeddings, but does not give methods to find them.   
Time-delayed embeddings are linked to Koopman operator methods \cite{Mezic2004}, which are themselves closely connected to the dynamic mode decomposition (DMD) \cite{kutz2016dynamic,Kamb2020,Brunton2022}.  In Section \ref{sect:yprop}, we show concretely that our method results in the same type of model as time-delayed DMD, but without any need to learn from data.

 Our method shares structural features with self-energy methods where one solves for a many-body system's Green's function $G(t,t')$ via Dyson's equation \cite{stefanucci2013nonequilibrium, kadanoff2018quantum}
 \begin{equation*}
    (i\partial_t\ - H_0)G(t,t') = \delta(t-t') + \int_0^t\Sigma(t-t'')G(t'',t')dt''.
\end{equation*}
As in Mori-Zwanzig theory, forward evolution of $G$ depends on past states via a memory integral, here represented by convolution with the self-energy $\Sigma$, which itself depends on $G$.  Here $H_0$ is the non-interacting Hamiltonian;  $\Sigma$ incorporates interactions that go beyond mean-field theories\cite{negele1982mean}.  There are close analogies between the terms that comprise $\Sigma$ and the Coulomb and exchange-correlation ($V_{\text{xc}}$) potentials that one finds in TDDFT\cite{RevModPhys.74.601}.  As with $V_{\text{xc}}$, in general, the analytical form of $\Sigma$ is unknown, necessitating approximations such as GW\cite{aryasetiawan1998gw, reining2018gw, hedin1999correlation}.  Recent work has sought to use machine learning to improve these approximations\cite{Bassi2023, zhu2024predicting}.  The present paper's methods may yield insight into how to model $\Sigma$, especially its temporal decay within the memory integral.

A key feature of our system is that the propagator is unitary, implying the reversibility of the dynamics.  The Liouville-von Neumann equation (\ref{eqn:lvn0}) that governs the dynamics of $P(t)$ is not dissipative.  Hence we cannot directly apply balanced truncation methods that have been developed for open quantum systems \cite{Schmidt2011,Benner2020}.  Variants of these methods, when applied to the TDCI system, may yield substantial computational savings.  This is outside the scope of the present work.

\subsection{Liouville-von Neumann}
In several contexts, electron density matrices $P(t)$ satisfy an equation of motion known as the Liouville-von Neumann equation:
\begin{equation}
\label{eqn:lvn0}
i \frac{dP}{dt} = [H, P].
\end{equation}
Here $[H,P]=HP-PH$ denotes the commutator. In fact, we show in Section \ref{sect:background} that the TDCI \emph{full} electron density matrices $P(t)$ satisfy (\ref{eqn:lvn0}). The 1RDM $\widetilde{P}(t)$ from time-dependent Hartree-Fock theory satisfies (\ref{eqn:lvn0}) with a density-dependent Hamiltonian $H(\widetilde{P})$ \cite{mclachlan1964time,mcweeny1989methods}.  Thus it is natural to ask whether the TDCI $1$-electron \emph{reduced} density matrices $Q(t)$ evolve via (\ref{eqn:lvn0}).  If they do, one can show that they must satisfy $Q(t) = U(t)^\dagger Q(0) U(t)$, where $U(t)$ is unitary and $\dagger$ denotes conjugate transpose; consequently, the eigenvalues of $Q(t)$ must equal those of $Q(0)$ for all $t \geq 0$.

\begin{figure}[tbh]
    \centering
     \includegraphics[width=0.9\linewidth]{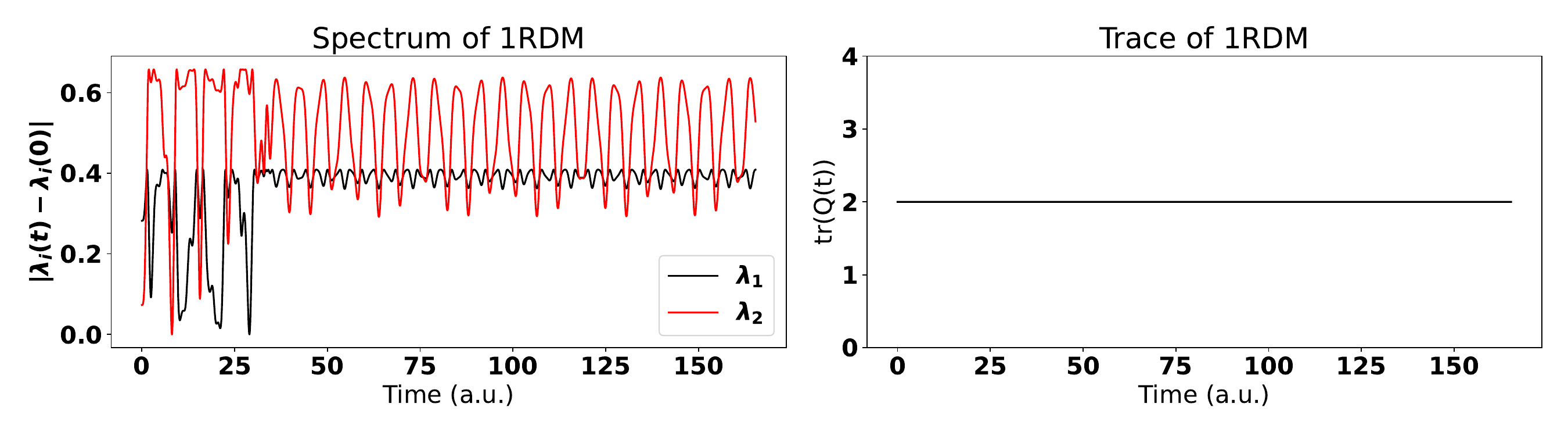}
    \caption{For the molecule $\heh$ in the STO-3G atomic orbital basis set, we apply the methods described in Section \ref{sect:procedure} to compute $1$-electron reduced density matrices $Q(t)$.  From the left plot, we conclude that the eigenvalues of $Q(t)$ do not stay constant in time, implying that $Q(t)$ cannot satisfy the Liouville-von Neumann equation for \emph{any} choice of Hamiltonian $H(t)$.  The right plot shows that $\trace(Q(t)) = 2$.  Since $\heh$ has $N=2$ electrons, this agrees with $\trace(Q(t)) = N$, which we prove in Appendix \ref{appendix:trace}.}
    \label{fig:1rdmtrace}
\end{figure}

This gives us a way to test if it is possible that TDCI $1$-electron reduced density matrices $Q(t)$ satisfy the Liouville-von Neumann equation.  For four two-electron molecular systems ($\heh$ and $\htwo$ in both STO-3G and 6-31G atomic orbital basis sets, for which we can compute the full CI solution with all possible determinants), we have applied TDCI methods---see Sections \ref{sect:background} and \ref{sect:procedure}---to compute full density matrices $P(t)$, from which we have computed $1$-electron reduced density matrices $Q(t)$.  Let us focus attention on $\heh$ in the STO-3G atomic orbital basis set; here $Q(t)$ is $2 \times 2$ and thus has two eigenvalues, $\lambda_1(t)$ and $\lambda_2(t)$. In Figure \ref{fig:1rdmtrace}, we plot the absolute deviation $| \lambda_j(t) - \lambda_j(0)|$, and we see that the eigenvalues of the $Q(t)$ matrices do not stay constant in time.  We omit plots for the other three molecular systems, all of which would show substantial fluctuations in the eigenvalues of $Q(t)$.

\emph{Because the eigenvalues of the TDCI $1$-electron reduced density matrices $Q(t)$ do not stay constant in time, there does not exist a Hamiltonian $H$ such that these $Q(t)$ matrices satisfy the Liouville-von Neumann equation.}  Thus we cannot apply methods from our prior work \cite{bhat2020machine,gupta2022statistical} to learn a Hamiltonian $H$ such that these $Q(t)$ satisfy (\ref{eqn:lvn0}).

\subsection{Reduced Density Matrices (RDMs)}
A subset of the many-body quantum physics and chemistry literature is concerned with deriving governing equations for RDMs, for both time-independent \cite{ColemanA.J.1963SoFD,PhysRevA.57.4219,Mentel2015} and time-dependent \cite{PhysRevA.75.012506,Giesbertz2010,10.1063/1.3499601,AkbariA.2012Citt} problems.  Consider the operator obtained by integrating $| \Psi \rangle \langle \Psi |$ with respect to all spins and all but $k$ of the electrons' positions; the matrix representation of this operator is the $k$-electron reduced density matrix or $k$RDM.  Here we have used bra-ket notation---see Definition \ref{def:fulldenop}.
Starting from the Schr\"odinger equation, one can derive a governing equation for the 1RDM $Q(t)$, but it will necessarily involve the 2RDM \cite{FerreNicolas2016RDMF}.  While there have been multiple attempts to close the 1RDM equation via approximate reconstruction of the 2RDM from the 1RDM \cite{PhysRevA.57.4219,HerbertJohnM.2002Cotd,PhysRevLett.101.033004,Giesbertz2010,Jeffcoat2014}, we have not seen any prior attempts to use \emph{memory} (present and past 1RDMs) to derive a closed system.  More generally, if one seeks an equation of motion for the 2RDM, it will either involve reconstruction of higher-order matrices such as the 3RDM or 4RDM \cite{PhysRevA.57.4219,PhysRevA.91.023412,PhysRevA.95.033414}, or a reexpression of constraints on higher-order RDMs (the $N$-representability conditions) \cite{Fosso-TandeJacob2016LVTR,PhysRevLett.130.153001}. Even within this literature on higher-order RDM systems, the use of memory appears to be unexplored.

\subsection{Summary of Contributions}
In Section \ref{sect:generalderivs}, we give a general method to derive a self-contained linear delay equation for the propagation of a reduced quantity of interest.  We provide this section to highlight the key steps of the derivation, which already constitutes one novel contribution of this paper, prior to introducing any notions from quantum mechanics.  The general setting also allows us to contrast our method with both time-delayed DMD and a Mori-Zwanzig approach, which we  do in Section \ref{sect:generalderivs}.

Section \ref{sect:background} reviews essential background material, and is intended to increase the accessibility of the paper to applied mathematicians.  Here we  define a subset of concepts from quantum/computational chemistry such as orbitals, Slater determinants, and basis functions.  We review TDCI and its relation to the TDSE.  We show that the full density matrix $P(t)$ satisfies the Liouville-von Neumann equation.   Theorem \ref{constanttrace} in Appendix \ref{appendix:trace} shows that $1$-electron reduced density matrices $Q(t)$ have constant trace equal to $N$, as exemplified by the right plot in 
Figure \ref{fig:1rdmtrace}. 

In Section \ref{sect:methods}, we specialize the derivation from Section \ref{sect:generalderivs} to the TDCI setting.  We begin by deriving several  useful properties of $1$-electron reduced density matrices that stem from TDCI.  The derivations here are general with respect to the number of electrons ($N$), the number of atomic/molecular orbitals ($K$), and the number of configurations (or Slater determinants) included in the wave function ($N_C$).  Proposition \ref{Bten} enables concrete evaluation of the $4$-index tensor $B$ that relates full configuration interaction (CI) electron density matrices to $1$-electron reduced density matrices.   Theorem \ref{Qpropsymm} establishes a self-contained delay equation for $Q(t)$ that preserves Hermitian symmetry, constant trace, and identically zero elements.


We present numerical results in Section \ref{sec:results}, applying the time-delayed scheme from Theorem \ref{Qpropsymm} to propagate reduced density matrices $Q(t)$ for four different molecular systems.  First, we establish the accuracy of the $Q(t)$ matrices computed using the scheme from Theorem \ref{Qpropsymm}.  For each molecular system, accuracy is controlled by the total memory length, which is defined to be the physical time corresponding to the maximum time delay in the evolution equation for $Q(t)$.  With sufficiently large total memory length, the resulting $Q(t)$ matrices computed using our method match those computed directly from full TDCI simulations.  We also quantify how the accuracy of $Q(t)$ depends on our choice of time step, basis set, and stride.  A stride of $k$ means that, in the equation to propagate $Q(t)$, we use every $k$-th previous reduced density matrix, up to a preset limit.

\section{Core Ideas}
\label{sect:generalderivs}
Here we present derivations that illustrate the core ideas of our approach without the complexity of the quantum system that we present and analyze in Sections \ref{sect:background} and \ref{sect:methods}.  The derivation in Section \ref{sect:yprop} clarifies the key steps of our time-delay scheme to propagate $\mathbf{y}(t)$.  In Section \ref{sect:qcderiv}, we revisit and refine this derivation for the quantum chemistry context.  The derivation in Section \ref{sect:pols} frames our problem as a partially observed linear system, enabling a comparison with Mori-Zwanzig theory.

We begin with the discrete-time linear system
\begin{equation}
\label{eqn:motivlinsys}
\mathbf{z}(t+1) = A(t) \mathbf{z}(t)
\end{equation}
with initial condition $\mathbf{z}(0)$.  Here $t$ is a nonnegative integer, $A(t)$ is an $n \times n$ complex, unitary matrix, and $\mathbf{z}(t)$ is a complex $n \times 1$ vector.  We take $A(t)$ to be unitary as this corresponds to quantum systems of interest.  More generally, the derivations in this section hold when $A(t)$ is invertible, a condition always satisfied when (\ref{eqn:motivlinsys}) arises by discretizing a linear ordinary differential equation with a sufficiently small time step.  When $A$ is unitary, $A^\dagger = A^{-1}$; in case $A(t)$ is invertible but not unitary, one can replace $A(t)^\dagger$ with $A(t)^{-1}$.

Suppose we are given a complex $m \times n$ matrix $R$ with $m \ll n$ that enables one to compute, at each time $t$, a \emph{reduced quantity of interest}:
\begin{equation}
\label{eqn:reduced}
\mathbf{y}(t) = R \mathbf{z}(t).
\end{equation}
Here $\mathbf{y}(t)$ is a complex $m \times 1$ vector.  Starting with $\mathbf{y}(t+1)$ and using (\ref{eqn:motivlinsys}-\ref{eqn:reduced}), we have
\begin{equation}
\label{eqn:motivlinsys2}
\mathbf{y}(t+1) = R \mathbf{z}(t+1) = R A(t) \mathbf{z}(t).
\end{equation}
To close this equation, and thereby enable propagation of $\mathbf{y(t)}$, we must express $\mathbf{z}(t)$ in terms of $\{ \mathbf{y}(t-\ell), \ldots, \mathbf{y}(t-1), \mathbf{y}(t) \}$ (for some $\ell \leq t$).

\subsection{Self-Contained Propagation of Reduced Quantities} 
\label{sect:yprop}
From (\ref{eqn:motivlinsys}), we can derive
\begin{equation}
\label{eqn:backprop}
\mathbf{z}(t-\ell) =  \prod_{j=0}^{\ell-1} A(t-\ell+j)^\dagger \mathbf{z}(t).
\end{equation}
Multiplying both sides by $R$ and using (\ref{eqn:reduced}) at time $t - \ell$ yields
\begin{equation}
\label{eqn:usefulfact}
\mathbf{y}(t-\ell) = R  \prod_{j=0}^{\ell-1} A(t-\ell+j)^\dagger  \mathbf{z}(t).
\end{equation}
Stacking column vectors vertically and applying both (\ref{eqn:reduced}) and (\ref{eqn:usefulfact}), we obtain
\begin{equation}
\label{eqn:Mdef}
\mathbf{Y}_{\ell}(t) = \begin{bmatrix}
\mathbf{y}(t)\\
\mathbf{y}(t-1)\\
\mathbf{y}(t-2)\\
\vdots\\
\mathbf{y}(t-\ell)
\end{bmatrix} = 
\begin{bmatrix}
R \mathbf{z}(t)\\
R \mathbf{z}(t-1)\\
R \mathbf{z}(t-2)\\
\vdots\\
R \mathbf{z}(t-\ell)
\end{bmatrix}
= \underbrace{\begin{bmatrix}
R \\
R A(t-1)^\dagger \\
R A(t-2)^\dagger A(t-1)^\dagger \\
\vdots\\
R \prod_{j=0}^{\ell-1} A(t-\ell+j)^\dagger 
\end{bmatrix}}_{M(t)} \mathbf{z}(t).
\end{equation}
Here $M(t)$ has size $(\ell+1) m \times n$, i.e., $\ell+1$ block rows each of size $m \times n$.  Assume that $R$, the first block of $M(t)$, has rank $m$ (full rank because $m < n$).  Each subsequent block of $M(t)$ has rank $m$, since each consists of $R$ multiplied by a unitary matrix.  Suppose that these blocks, when concatenated, yield a matrix of full rank.  Then $M(t)$ will have rank $n$ as long as
$\ell \geq \lfloor n/m \rfloor - 1$.
\emph{Under this condition, the matrix $M(t)$ will have a pseudoinverse $M(t)^{+}$ that is a one-sided left inverse.}  Applying the pseudoinverse, we obtain
\begin{equation}
\label{eqn:propy}
\mathbf{y}(t+1) = R \mathbf{z}(t+1) = R A(t) \mathbf{z}(t) = \underbrace{R A(t) M(t)^{+}}_{S(t)} \mathbf{Y}_{\ell}(t) = S(t) \begin{bmatrix}
\mathbf{y}(t)\\
\mathbf{y}(t-1)\\
\mathbf{y}(t-2)\\
\vdots\\
\mathbf{y}(t-\ell)
\end{bmatrix}
\end{equation}
\emph{This self-contained, linear, time-delay equation for the reduced variable $\mathbf{y}(t)$ underpins the present study.} Note that $S(t)$ has size $m \times (\ell+1) m$.  In practice, we compute an approximate pseudoinverse $M(t)^{+}$ by ignoring sufficiently small singular values of $M(t)$.  This enables use of (\ref{eqn:propy}) even when $M(t)$ is ill-conditioned and/or rank-deficient.

Time-delayed DMD aims to learn $S$ that minimizes the squared error between the left- and right-hand sides of
\begin{equation}
\label{eqn:motivlinsys3}
\mathbf{y}(t+1) = \sum_{j=0}^{\ell} S_{j} \mathbf{y}(t - j) = \underbrace{\begin{bmatrix}S_0 & S_1 & \cdots & S_{\ell}\end{bmatrix}}_{S} \begin{bmatrix} \mathbf{y}(t) \\ \mathbf{y}(t-1) \\ \vdots \\ \mathbf{y}(t-\ell) \end{bmatrix}
\end{equation}
In time-delayed DMD, one learns $S$ from observations of $\mathbf{y}(t)$. We see that (\ref{eqn:propy}) is structurally identical to (\ref{eqn:motivlinsys3}) except that $S$ is replaced by $S(t)$.  \emph{The key difference between our method and time-delayed DMD is that rather than learn $S(t)$ from data, we derive $S(t)$ from first principles,  leveraging the structure of (\ref{eqn:motivlinsys}-\ref{eqn:reduced}).}  In Section \ref{sect:methods}, we will see that no observations of $\mathbf{y}(t)$ are needed or used to derive the $S(t)$ that appears in (\ref{eqn:propy}). To compute $M(t)^+$ from $M(t)$, we use (i) the unitary propagator $A(t)$, (ii) the reduction matrix $R$, and (iii) one singular value decomposition (SVD) per time step.

Reexamining the derivation above, we find that there is no reason why we must step back by precisely one time step at a time.  That is, we can formulate our method using a \emph{stride of $k$ steps}.  With a general stride of $k \geq 1$, we  replace the $\mathbf{Y}_{\ell}(t)$ vector in (\ref{eqn:Mdef}) with
\begin{equation}
\label{eqn:Mdef2}
\mathbf{Y}_{\ell}(t) = \begin{bmatrix}
\mathbf{y}(t)\\
\mathbf{y}(t- k)\\
\mathbf{y}(t-2 k)\\
\vdots\\
\mathbf{y}(t-\ell k)
\end{bmatrix} = 
\begin{bmatrix}
R \mathbf{z}(t)\\
R \mathbf{z}(t-k)\\
R \mathbf{z}(t-2 k)\\
\vdots\\
R \mathbf{z}(t-\ell k)
\end{bmatrix}
= \underbrace{\begin{bmatrix}
R \\
R \prod_{j=0}^{k-1} A(t-k+j)^\dagger \\
R \prod_{j=0}^{2 k-1} A(t-2 k+j)^\dagger \\
\vdots\\
R \prod_{j=0}^{\ell k-1} A(t-\ell k+j)^\dagger 
\end{bmatrix}}_{M(t)} \mathbf{z}(t).
\end{equation}
The stride is important for the following reason. The intrinsic time scales of the system (\ref{eqn:motivlinsys}) will determine how close each $A(t)$ is to the identity.  If $A(t-1)$ is too close to the identity (i.e., if the system is slowly varying at time $t-1$), then  in the $M(t)$ matrix that appears in (\ref{eqn:Mdef}), the block $R A(t-1)^\dagger$ will be close to the $R$ block.  This may make it difficult for $M(t)$ to achieve full rank.  By using a sufficiently large stride $k$ in (\ref{eqn:Mdef2}), we can correct for this.  Our motivating principle is that with $k$ sufficiently large, each successive $\prod$ term in the $M(t)$ matrix in (\ref{eqn:Mdef2}) will \emph{differ enough from the identity} so that $M(t)$ will have high (if not full) rank.

In this work, we focus on the discrete-time scheme (\ref{eqn:Mdef2}).  Our results in Section \ref{sect:results}---see Section \ref{sect:timesteprefinement} in particular---hint that the dynamics may converge as we refine our time step $\Delta t$ while keeping the physical memory length $k \ell \Delta t$ constant.  In Appendix \ref{sect:continuoustime}, we give a continuous-time generalization of the above derivation, but we reserve further discussion of this for future work.

\subsection{Connection with Partially Observed Linear Systems}
\label{sect:pols}
To provide further mathematical motivation for the present paper, we examine the consequences of viewing (\ref{eqn:motivlinsys}-\ref{eqn:reduced}) through the lens of partially observed linear systems\cite[Section 1]{pan2020structure}.  In this subsection only, we assume $A(t) = A$ is a fixed (not time-dependent) unitary matrix.  As mentioned above, we assume that the matrix $R$ from (\ref{eqn:reduced}) has rank $m$.  Starting with $R$ and completing the basis, one can find $\widetilde{R}$ of size $(n-m) \times n$ such that the square $n \times n$ matrix $\mathbf{R} = \displaystyle \begin{bmatrix} R \\ \widetilde{R} \end{bmatrix}$ has rank $n$.  Then
\[
\begin{bmatrix} \mathbf{y}(t) \\ \widetilde{\mathbf{y}}(t) \end{bmatrix} = \mathbf{R} \mathbf{z}(t) \ \Longrightarrow \ \mathbf{z}(t) =  \mathbf{R}^{-1} \begin{bmatrix} \mathbf{y}(t) \\ \widetilde{\mathbf{y}}(t) \end{bmatrix}.
\]
With this, we have
\[
\begin{bmatrix} \mathbf{y}(t+1) \\ \widetilde{\mathbf{y}}(t+1) \end{bmatrix} = \mathbf{R} \mathbf{z}(t+1) = \mathbf{R} A \mathbf{z}(t) = \underbrace{ \mathbf{R} A \mathbf{R}^{-1} }_{B} \begin{bmatrix} \mathbf{y}(t) \\ \widetilde{\mathbf{y}}(t) \end{bmatrix}
\]
Now we partition the $n \times n$ matrix $B$ into blocks that match the sizes of $\mathbf{y}$ and $\widetilde{\mathbf{y}}$: $B_{11}$ is $m \times m$, $B_{12}$ is $m \times (n-m)$, $B_{21}$ is $(n-m) \times m$, and $B_{22}$ is $m \times m$.  Then the previous equation becomes
\[
\begin{bmatrix} \mathbf{y}(t+1) \\ \widetilde{\mathbf{y}}(t+1) \end{bmatrix} = \begin{bmatrix} B_{11} & B_{12} \\ B_{21} & B_{22} \end{bmatrix} \begin{bmatrix} \mathbf{y}(t) \\ \widetilde{\mathbf{y}}(t) \end{bmatrix}.
\]
From this, one can derive the following time-delay evolution equation\cite{pan2020structure}:
\begin{equation}
\label{eqn:ourMZ}
\mathbf{y}(t+1) = B_{11} \mathbf{y}(t) + \sum_{s=0}^{t-1} B_{12} B_{22}^s B_{21} \mathbf{y}((t-1)-s) + B_{12} B_{22}^t \widetilde{\mathbf{y}}(0).
\end{equation}
This is essentially a Mori-Zwanzig equation: the three terms on the right-hand side can be viewed as (i) a non-time-delayed or Markovian term $B_{11} \mathbf{y}(t)$, (ii) a time-delayed or non-Markovian term that is a linear combination of  $\{\mathbf{y}(0), \ldots, \mathbf{y}(t-1)\}$, and (iii) a closure term that involves $\widetilde{\mathbf{y}}$.  To handle this closure term, one assumes\cite{pan2020structure} that $B_{12} B_{22}^t \widetilde{\mathbf{y}}(0) \to 0$.  More generally, $\widetilde{\mathbf{y}}(t)$ is often viewed as noise \cite{Chorin2014}.

Here we note that as $A$ is unitary and $B = \mathbf{R} A \mathbf{R}^{-1}$, the eigenvalues of $B$ all have modulus $1$.  Hence we cannot expect that $B_{12} B_{22}^t \widetilde{\mathbf{y}}(0) \to 0$ for generic unitary choices of $A$ and $\mathbf{R}$.  More generally, we cannot expect to use arguments based on retaining (or projecting onto) eigenvectors of $B$ associated with the top $k$ (in modulus) eigenvalues of $B$.  It is likely that we will not be able to truncate the summation in the memory term.

In preliminary work, we implemented (\ref{eqn:ourMZ}) for a system that corresponds to (\ref{eqn:motivlinsys}-\ref{eqn:reduced}) with unitary $A$.  Given $R$, it is up to us to form $\widetilde{R}$ such that $\mathbf{R}$ is full rank.  We found that when $A$ is dense and unitary, for many choices of $\widetilde{R}$, the resulting system (\ref{eqn:ourMZ}) is numerically unstable for long-time propagation.  When $A$ is diagonal and unitary, long-time propagation succeeds.  In this regime, we found that we cannot ignore the $\widetilde{\mathbf{y}}(0)$ term; it does not decay as $t \to \infty$.  Similarly, regarding the second term in (\ref{eqn:ourMZ}), unless we sum over \emph{all} past states $\mathbf{y}(s)$ for $s < t$, the resulting $\mathbf{y}(t)$ propagation is highly inaccurate.

The above provides additional mathematical/computational motivation to pursue our propagation scheme (\ref{eqn:propy}).  In (\ref{eqn:propy}), there is no requirement to solve for $\widetilde{R}$, the time delay is fixed, and we do not encounter any closure problems.

\section{Background}
\label{sect:background}
We focus our attention on the dynamics of electrons in molecular systems under the Born-Oppenheimer approximation, in which we treat the nuclei as fixed.  This yields the problem of solving for the dynamics of $N$ electrons in the field of $N'$ nuclei.  Then the equation of motion is the time-dependent Schr\"odinger equation (TDSE)
\begin{equation}
\label{eqn:tdse}
i\frac{\partial}{\partial t} \Psi(\bx_1, \bx_2, \ldots, \bx_N, t) = \hat{H}(t) \Psi(\bx_1, \bx_2, \ldots, \bx_N, t).
\end{equation}
Here each $\mathbf{x}_k = (\br_k, \sigma_k)$ with spatial coordinate $\br_k \in \mathbb{R}^3$ and spin coordinate $\sigma_k$ that does not appear in the Hamiltonian $\hat{H}(t)$.  We describe our treatment of the spin coordinates below in Definition \ref{defn:spinorbitals}.  The Hamiltonian $\hat{H}(t)$ is a self-adjoint operator that can be decomposed as $\hat{H}(t) = \hat{H}_e + \hat{V}_{\text{ext}}(t)$, the sum of the electronic Hamiltonian $\hat{H}_e$ with an external potential operator $\hat{V}_{\text{ext}}(t)$ that models, for instance, an applied electric field.  In atomic units, the electronic Hamiltonian is
\begin{equation}
\label{eqn:molham}
\hat{H}_e = - \sum_{k=1}^{N} \frac{1}{2} \nabla_k^2 - \sum_{k=1}^{N} \sum_{A=1}^{N'} \frac{Z_A}{\| \br_k - \mathbf{R}_A \|} + \sum_{k=1}^{N} \sum_{j < k} \frac{1}{\| \br_j - \br_k \|}.
\end{equation}
Here $\nabla_k^2$ is the Laplacian over coordinates $\br_k \in \mathbb{R}^3$. The first term encodes the kinetic energy operators of all $N$ electrons; the remaining terms are potential energy operators.   The second term encodes all electron-nuclear Coulomb attractions, with nucleus $A$ having charge $Z_A$ and fixed position $\mathbf{R}_A$.  The third term encodes all electron-electron Coulomb repulsions. For now, we have ignored spin coordinates of $\Psi$, as they do not appear in the Hamiltonian.  Further details can be found in standard references\cite{szabo2012modern,mcweeny1989methods}.

As $\Psi$ is a function of $3N+1$ variables, it should not be surprising that exact solutions are unknown except when $N=1$.  For $N \geq 2$, solving (\ref{eqn:tdse}-\ref{eqn:molham}) numerically presents serious challenges. Classical numerical methods such as finite differences and finite elements do not address the key issue, which is how to represent/store $\Psi$.  Computational chemistry provides several methods to do this, at varying levels of approximation.  One such method is configuration interaction, which we now describe.


\subsection{Time-Dependent Configuration Interaction}
\label{sect:TDCI}
We begin with a choice of \emph{basis set}. As they are designed to approximate eigenfunctions of an atomic Hamiltonian, the functions in the basis set are often called \emph{atomic orbitals}.  Two atom-centered basis sets commonly used for theoretical studies, which we refer to later in this paper, are STO-3G \cite[\S 3.6.2]{szabo2012modern} and 6-31G \cite{Hehre1972}.  STO-3G is a minimal basis set that uses a linear combination of three Gaussians to describe a Slater type orbital.  For hydrogen and helium, STO-3G provides one s-orbital for each atom.  6-31G is a double valence basis set that uses a larger number of Gaussian functions to better describe the nuclear cusp, and for hydrogen and helium provides s-orbitals of different size on each atom to deliver more flexibility for the electronic distribution.   Let $K$ denote the number of functions (atomic orbitals) in the basis set.  As $K$ increases, we approach the complete basis set limit, yielding more accurate calculations at added computational cost.

\begin{defn}[Coordinates, Spin, and Orbitals]\label{defn:spinorbitals}
Let $\mathbf{x} = (\br, \sigma)$ where $\br$ is a spatial coordinate and $\sigma$ is a spin coordinate.  Let $\{f_k(\br)\}_{k=1}^K$ denote a set of atomic orbitals. 
 Let $\alpha(\sigma)$ and $\beta(\sigma)$ denote two functions of $\sigma$ that correspond to spin up and spin down, the two possible spin states of an electron.  We take $\alpha$ and $\beta$ to be orthonormal, i.e., $\langle \alpha, \alpha \rangle = \langle \beta, \beta \rangle = 1$ while $\langle \alpha, \beta \rangle = 0$.  Let $\{\phi_j(\br)\}_{j=1}^{K}$ be an orthonormal set of spatial molecular orbitals; each $\phi_j$ is a linear combination of the $K$ atomic orbitals, i.e.,
\begin{equation}
\label{eqn:MO}
\phi_j(\br) = \sum_{k=1}^K Y_{jk} f_k(\br).
\end{equation}
Define the orthonormal set of spin-orbitals $\{ \chi_k(\mathbf{x}) \}_{k=1}^{2K}$ via
\begin{equation}
\label{eqn:spinorbital}
\chi_{2j - 1}(\mathbf{x}) = \phi_j(\br)\alpha(\sigma) \  \text{ and } \ 
\chi_{2j}(\mathbf{x}) = \phi_j(\br)\beta(\sigma), \quad  \text{ for } j = 1, \ldots, K.
\end{equation}
\end{defn}
Using the $1$-electron molecular spin-orbitals, we form \emph{Slater determinants} \cite[\S 2.2.3]{szabo2012modern}. Their use is motivated by the Pauli exclusion principle, which requires that electronic wave functions be antisymmetric.  Exchanging $\mathbf{x}_j$ and $\mathbf{x}_k$ on the left-hand side of (\ref{eqn:slater}) causes an exchange of rows $j$ and $k$ in the determinant, yielding a factor of $-1$.

\begin{defn}[Slater Determinants] With a combination $\mathbf{i} = \{i_1, i_2, \ldots, i_N\}$ of $N$ distinct elements from $\{1, 2, \ldots, 2 K \}$, we form an $N$-electron Slater determinant
\begin{equation}
\label{eqn:slater}
\psi^{\text{SL}}_\mathbf{i}(\mathbf{x}_1,\mathbf{x}_2, \ldots, \mathbf{x}_N) = \frac{1}{\sqrt{N!}}\begin{vmatrix}
\chi_{i_1}(\br_1,\sigma_1) & \chi_{i_2}(\br_1,\sigma_1) & \cdots & \chi_{i_N}(\br_1,\sigma_1) \\
\chi_{i_1}(\br_2,\sigma_2) & \chi_{i_2}(\br_2,\sigma_2) & \cdots & \chi_{i_N}(\br_2,\sigma_2) \\
\vdots & \vdots & \ddots & \vdots \\
\chi_{i_1}(\br_N,\sigma_N) & \chi_{i_2}(\br_N,\sigma_N) & \cdots & \chi_{i_N}(\br_N,\sigma_N)
\end{vmatrix}.
\end{equation}
\end{defn}
There are $\binom{2K}{N}$ possible $N$-electron Slater determinants; we use a subset of these to form CI (Configuration Interaction) basis functions.  Let $N_C$ denote the size of this subset.
\begin{defn}[CI Basis]
For $1 \leq q \leq N_C\leq \binom{2K}{N}$, let $\mathbf{i}(q)$ enumerate the combinations in a fixed subset of distinct Slater determinants.  Then the CI basis functions are
\begin{equation}
\label{eqn:CIbasis}
\Psi^\mathrm{CI}_a = \sum_{q=1}^{N_C} C_{aq} \psi^{\text{SL}}_{\mathbf{i}(q)} \qquad \text{ for } a = 1, \ldots, N_C.
\end{equation}
\end{defn}
Given an atomic orbital basis set $\{f_k\}_{k=1}^K$, in order to proceed with concrete CI basis functions (\ref{eqn:CIbasis}), we must solve numerically for the matrices $Y$ and $C$ in (\ref{eqn:MO}) and (\ref{eqn:CIbasis}).  That is, we must pin down the coefficients that express the molecular orbitals in terms of the atomic orbitals, and the coefficients that express the CI basis functions in terms of the $N$-electron Slater determinants.  In CI approaches,  one first solves the time-independent Hartree-Fock equations, a nonlinear eigenvalue problem, to determine $Y$.  The resulting $\{\phi_j\}_{j=1}^{K}$ are used to construct Slater determinants (\ref{eqn:slater}).  One then solves for $C$ such that the CI basis $\{\Psi^\mathrm{CI}_a\}_{a=1}^{N_C}$ is orthonormal and diagonalizes $\hat{H}_e$.  In contrast, there exist other approaches, such as CASSCF (complete active space self-consistent field) \cite{Roos1980}, where one solves jointly for $Y$ and $C$.  Either way, the CI basis functions (\ref{eqn:CIbasis}) are orthonormal and diagonalize the $\hat{H}_e$ operator.

In TDCI and TDCASSCF \cite{PhysRevA.88.023402,olsen1988determinant,peng2018simulating}, the time-dependent versions of CI and CASSCF, the wave function takes the following form:
\begin{equation}
\label{eqn:CIrep}
\Psi(\mathbf{x}_1, \ldots, \mathbf{x}_N, t) = \sum_{n=1}^{N_C} a_n(t) \Psi_n^{\mathrm{CI}} (\mathbf{x}_1, \ldots, \mathbf{x}_N).
\end{equation}
Even with an incomplete (i.e., finite) set of atomic orbitals, if we use all possible Slater determinants in (\ref{eqn:CIrep}), and if we then solve for the coefficients $a_n(t)$ such that $\Psi$ satsifies (\ref{eqn:tdse}), what we obtain is the exact solution of the TDSE within the finite atomic orbital basis set we chose.  In particular, this solution accounts for electron correlation, an effect that we must model if we are to use our results to improve models of the exchange-correlation potential $V_{\text{xc}}$ in TDDFT.  Repeating this procedure for increasing values of basis set size ($K$), $\Psi$ will converge to the solution of the TDSE (\ref{eqn:tdse}).

Substituting (\ref{eqn:CIrep}) into the TDSE (\ref{eqn:tdse}) and expressing the Hamiltonian operator $\hat{H}(t)$ in the CI basis $\{ \Psi_n^{\mathrm{CI}} \}_{n=1}^{N_C}$ yields an ordinary differential equation for the time-evolution of the coefficients $\ba(t) = [a_1(t), \ldots, a_{N_C}(t)]$:
\begin{equation}
\label{eqn:matrixschro}
i \frac{d}{dt} \ba(t) = (H_0 + V_{\text{ext}}(t)) \ba(t) = H(t) \ba(t).
\end{equation}
Here $H_0$ and $V_{\text{ext}}(t)$ are the matrix expressions of the corresponding operators $\hat{H}_e$ and $\hat{V}_{\text{ext}}(t)$ in the CI basis.  In this paper, we typically apply to the molecule an electric field in the $z$-direction; the corresponding potential can be written within the dipole approximation as $V_{\text{ext}}(t) = f(t) M_{\text{dip}}$, where $f(t)$ is a time-dependent field strength and $M_{\text{dip}}$ is the dipole moment matrix (in the $z$ direction).  Both $H_0$ and $M_{\text{dip}}$ will depend on the particular molecular system that we are studying.  Hence
\begin{equation}
\label{eqn:matrixham}
H(t) = H_0 + V_{\text{ext}}(t) = H_0 + f(t) M_{\text{dip}}.
\end{equation}
\subsection{Full Density Matrix}
The full density matrix associated with (\ref{eqn:CIrep}) is
\begin{equation}
\label{eqn:fulltdciden}
P(t) = \ba(t) \ba(t)^\dagger.
\end{equation}
Taking the time-derivative of both sides and using (\ref{eqn:matrixschro}), we obtain (\ref{eqn:lvn0}) with $H=H(t)$, the time-dependent Hamiltonian (\ref{eqn:matrixham}).  This shows that the $N_C \times N_C$ full TDCI density matrix $P(t)$ satisfies the Liouville-von Neumann equation with the same Hamiltonian matrix $H(t)$ used in (\ref{eqn:matrixschro}). By (\ref{eqn:fulltdciden}), $P(t)$ is Hermitian for all $t$.  As $H(t)$ does not depend on $P$, the Liouville-von Neumann equation is linear in the \emph{entries} of $P$.  For a matrix $A$ of size $M \times M$, let $\vvec(A)$ denote its vectorized or flattened representation as an $M^2 \times 1$ vector.  Then (\ref{eqn:lvn0}) can be written as
\begin{equation}
\label{eqn:superoperator}
\frac{d}{dt} \vvec(P(t)) = -i \mathcal{H}(t) \vvec(P(t)),
\end{equation}
where, using $\otimes$ to denote Kronecker product, $\mathcal{H}(t) = I \otimes H(t) - H(t)^T \otimes I$. This is the superoperator algebra formulation of the Liouville-von Neumann equation\cite{zwanzig1964identity,ohtsuki1989bath,May2023}.  Suppose we fix $\Delta t > 0$ and discretize (\ref{eqn:superoperator}) in time using the first-order scheme
\begin{equation}
\label{eqn:discsuper}
\vvec(P(t + \Delta t)) = \exp(-i \mathcal{H}(t) \Delta t) \vvec(P(t)).
\end{equation}
With $\oplus$ denoting Kronecker sum, we have
\begin{equation}
\label{eqn:kronsum}
\exp(-i \mathcal{H}(t) \Delta t) = \exp((i H(t)^T \Delta t) \oplus (-i H(t) \Delta t)) = \exp(i H(t)^T \Delta t) \otimes \exp(-i H(t) \Delta t),
\end{equation}
a unitary matrix of size $N_C^2 \times N_C^2$.  Hence we can identify $\exp(-i \mathcal{H}(t))$ with $A(t)$ in (\ref{eqn:motivlinsys}), and further identify $\vvec(P(t + \Delta t))$ and $\vvec(P(t))$ with $\mathbf{z}(t+1)$ and $\mathbf{z}(t)$, respectively.

Normalization of $\Psi$ implies $\mathbf{a}(t)^\dagger \mathbf{a}(t) = 1$ for all $t$, which then implies $\trace(P(t))=1$ for all $t$.



\section{Methodological Results}
\label{sect:methods}
\subsection{Reduced Density}
\label{sect:rdmdef}
To define the $1$-electron reduced density matrix, we must first define density operators.  Let $\mathbf{X} = (\mathbf{x}_1, \ldots, \mathbf{x}_N)$.  We use $\Psi_t(\mathbf{X})$ as equivalent notation for the wave function $\Psi(\mathbf{X},t)$ defined by (\ref{eqn:CIrep}).  Then, for each $t$, $\Psi_t$ is an element of a complex Hilbert space $\mathcal{S}$ with inner product $\langle F, G \rangle = \int \overline{F(\mathbf{X})} G(\mathbf{X}) \, d \mathbf{X}$.  For the purposes of the present study, precise identification of $\mathcal{S}$ is unimportant---all we need is the inner product.  
\begin{defn}[Full Density Operator]
\label{def:fulldenop}
For $F \in \mathcal{S}$, let $F^\ast$ denote the linear functional in the dual space $\mathcal{S}^\ast$ that acts on a function $G \in \mathcal{S}$ via $F^\ast(G) = \langle F, G \rangle$.  This is equivalent to bra-ket notation; we identify $F \in \mathcal{S}$ with $| F \rangle$, and we identify $F^\ast \in \mathcal{S}^\ast$ with $\langle F |$.  Then the full density operator $P_t : \mathcal{S} \to \mathcal{S}$ associated with the wave function (\ref{eqn:CIrep}) is
\begin{equation}
\label{eqn:fulldenop}
\mathrm{P}_t = | \Psi_t \rangle \langle \Psi_t | = \Psi_t  \Psi_t^\ast = \sum_{k,\ell} a_k(t) \overline{a_{\ell}(t)} \Psi_k^{\mathrm{CI}} (\Psi_{\ell}^{\mathrm{CI}})^\ast.
\end{equation}
If we express this operator in the orthonormal CI basis $\{ \Psi_n^{\mathrm{CI}} \}_{n=1}^{N_C}$, we obtain the full density matrix $P(t)$ defined in (\ref{eqn:fulltdciden}).
\end{defn}
To obtain the $1$-electron reduced density operator, we take the sum of $N$ integrals: for the $j$-th such integral, we integrate the full density operator over all spins and all spatial coordinates other than $\br_j$.  All of the integrals are identical due to the indistinguishability of electrons.  Thus we carry out this procedure for $j=1$ and scale the result by $N$:
\begin{defn}[Marginalization and Reduced Density Operator]
Let $\mathcal{R}$ be a complex Hilbert space of functions that depend only on $\br \in \mathbb{R}^3$.  
For the full density operator $\mathrm{P}_t$, marginalization yields a reduced density operator $\left( \mathrm{P}_t \right)_1 : \mathcal{R} \to \mathcal{R}$ defined by its action on a test function $f \in \mathcal{R}$:
\begin{equation}
\label{eqn:sub1}
\left( \mathrm{P}_t \right)_1 (f) = N \int \Psi_t(\br,\sigma_1,\mathbf{x}_2,\ldots,\mathbf{x}_n) \int_{\br'} \overline{\Psi_t(\br',\sigma_1,\mathbf{x}_2,\ldots,\mathbf{x}_n)} f(\br') \, d \br' \, d \sigma_1 d \mathbf{x}_2 \cdots d \mathbf{x}_N.
\end{equation}
Equivalently, we can first define the $1$-electron reduced  pair density function
\begin{equation}
\label{eqn:pairdensity}
\rho(\br,\br',t) = N \int \Psi_t(\br,\sigma_1,\mathbf{x}_2,\ldots,\mathbf{x}_n) \overline{\Psi_t(\br',\sigma_1,\mathbf{x}_2,\ldots,\mathbf{x}_n)}  \, d \sigma_1 d \mathbf{x}_2 \cdots d \mathbf{x}_N,
\end{equation}
in which case we can more easily express the $1$-electron reduced density operator as
\begin{equation}
\label{eqn:pairdenop}
\mathrm{Q}_t = \left( \mathrm{P}_t \right)_1 (f) = \int_{\br'} \rho(\br, \br',t) f(\br') \, d \br'.
\end{equation}
\end{defn}
  Note that the $1$-electron reduced density $\rho(\br,t)$ defined in Section \ref{sec:intro} is obtained by evaluating the pair density (\ref{eqn:pairdensity}) at $\br' = \br$.  The matrix representation of the operator (\ref{eqn:pairdenop}) induced by the pair density (\ref{eqn:pairdensity}) is the 1RDM ($1$-electron reduced density matrix) $Q(t)$. The details are collected here:
\begin{defn}[Reduced Quantities]
\label{redquan}
Consider the reduced density operator: 
\begin{equation}
\label{eqn:reduceddenop}
\mathrm{Q}_t = \left( P_t \right)_1 =  N \sum_{k,\ell} a_k(t) \overline{a_{\ell}(t)} \left( \Psi_k^\mathrm{CI} (\Psi_{\ell}^\mathrm{CI})^\ast \right)_1. 
\end{equation}
The 1RDM $Q(t)$ is the matrix expression of the operator $\mathrm{Q}_t$ in the orthonormal basis of molecular orbitals $\{\phi_j(\br)\}_{j=1}^{K}$.  Using (\ref{eqn:fulltdciden}), we obtain both the $K \times K$ matrix $Q(t)$ and the $N_c \times N_c \times K \times K$ tensor $B$:
 \begin{gather}\label{eqn:rdmdef}
     Q_{b,c}(t) = \langle \phi_b, \mathrm{Q}_t \phi_c  \rangle = \sum_{k,\ell=1}^{N_C} P_{k,\ell}(t) B_{k,\ell,b,c}, \\
     \label{eqn:Btendef}B_{k,\ell,b,c} = N \int \overline{\phi_b(\br)} \left( \Psi_k^\mathrm{CI} (\Psi_{\ell}^\mathrm{CI})^\ast \right)_1 (\br) \phi_c(\br) \, d \br,
 \end{gather}
We reshape the tensor $B$ into a matrix of size $N_C^2 \times K^2$, which we denote as $\widetilde{B}^T$.
\end{defn}
The full density operator $\mathrm{P}_t$ yields the joint probability density of finding each electron with a particular spin at a particular spatial location at time $t$.  In contrast, the reduced density operator $\mathrm{Q}_t$ yields the probability density of finding \emph{an} electron (with either spin) at a particular spatial location at time $t$. 
Knowledge of the 1RDM $Q(t)$ is often sufficient to compute quantities that can be compared against experiments.

\subsection{Computing the $B$ Tensor}
\label{sect:Btendef}
To compute $B$, we must carry out the integral in (\ref{eqn:Btendef}); in full generality, this boils down to marginalizing outer products of Slater determinants.
\begin{prop}
\label{Bten}
For any fixed choice of distinct Slater determinants, 
suppose the CI basis functions are defined by (\ref{eqn:CIbasis}).  Then for $1 \leq k, \ell \leq N_C$, we can compute the core of $B$ defined by (\ref{eqn:Btendef}).
\begin{equation}
\label{eqn:radred10}
\left( \Psi^\mathrm{CI}_k (\Psi^\mathrm{CI}_{\ell})^\ast \right)_1 = \sum_{q=1}^{N_C} \sum_{q'=1}^{N_C} C_{k q} \overline{C_{\ell q'}} \bigl( \psi^{\text{SL}}_{\mathbf{i}(q)} (\psi^{\text{SL}}_{\mathbf{i}(q')})^\ast \bigr)_1.
\end{equation}
To compute the right-hand side, there are two cases.  When $q = q'$, 
\begin{equation}
\label{eqn:radred3}
\bigl( \psi^\text{SL}_{\mathbf{i}(q)} (\psi^\text{SL}_{\mathbf{i}(q)})^\ast \bigr)_1 = \frac{1}{N} \sum_{k=1}^{N} \phi_{\lceil i_k(q)/2 \rceil} \phi_{\lceil i_k(q)/2 \rceil}^\ast 
\end{equation}
When $q \neq q'$, there exist integers $a$, $a'$, and $Z$ such that $a \neq a'$ and
\begin{equation}
\label{eqn:radred4}
\bigl( \psi^\text{SL}_{\mathbf{i}(q)} (\psi^\text{SL}_{\mathbf{i}(q')})^\ast \bigr)_1 = \frac{(-1)^Z}{N}  \phi_{\lceil a/2 \rceil} \phi_{\lceil a'/2 \rceil}^\ast.
\end{equation}
The results (\ref{eqn:radred3}-\ref{eqn:radred4}), together with (\ref{eqn:radred10}), enable a complete evaluation of $B$.
\end{prop}
For the proof, see Appendix \ref{sect:proofBten}.  Using Proposition \ref{Bten}, we have developed Python code that computes the $B$ tensor.  This code takes as input the results of a static (time-independent) CI calculation: the coefficients $Y$ in (\ref{eqn:MO}), the coefficients $C$ in (\ref{eqn:CIbasis}), and the enumeration of Slater determinants $\mathbf{i}(q)$ in (\ref{eqn:CIbasis}).

\subsection{Propagating 1RDMs}
\label{sect:qcderiv}
Equipped with the $B$ tensor, we can propagate $1$-electron reduced densities forward in time, as summarized in the next result.  As in (\ref{eqn:Mdef2}), we formulate this for the case of a general stride $k \geq 1$.
\begin{prop}
\label{Qprop}
Let $k \geq 1$ be the stride.  With the Hamiltonian $H(t)$ from (\ref{eqn:matrixham}) and $\widetilde{B}^T$ from the end of Definition \ref{redquan}, set
\begin{subequations}
\label{eqn:CjAj}
\begin{align}
\label{eqn:Bsubj}
\mathcal{D}_j &= \widetilde{B}  \left( C_{j k}(t)^T \otimes A_{j k}(t) \right) \\
\label{eqn:cj}
    C_m(t) &=   \exp(-i H(t - \Delta t) \Delta t) \exp(-i H(t - 2 \Delta t) \Delta t) \cdots \exp(-i H(t - m \Delta t) \Delta t) \\
\label{eqn:aj}
    A_m(t) &= \exp(i H(t - m \Delta t) \Delta t) \cdots  \exp(i H(t - 2 \Delta t) \Delta t) \exp(i H(t - \Delta t) \Delta t) \\
\label{eqn:Mtql}
M(t) &= \begin{bmatrix}
\widetilde{B} \\
\mathcal{D}_1 \\
\vdots \\
\mathcal{D}_{\ell}
\end{bmatrix} \qquad \text{ and } \qquad
\mathbf{q}_{\ell}(t) = \begin{bmatrix}
\vvec(Q(t)) \\
\vvec(Q(t - k \Delta t)) \\
\vdots \\
\vvec(Q(t - \ell k \Delta t))
\end{bmatrix}.
\end{align}
\end{subequations}
Suppose that $P(t)$ evolves forward in time according to (\ref{eqn:discsuper}) and (\ref{eqn:kronsum}).  If $M(t)$ achieves full column rank, then the 1RDM $Q(t)$ satisfies the delay equation
\begin{equation}
\label{eqn:Qprop}
\vvec(Q(t+1)) =  \widetilde{B} \left( \exp(i H(t)^T \Delta t) \otimes \exp(-i H(t) \Delta t) \right) M(t)^{+} \mathbf{q}_{\ell}(t).
\end{equation}
Here $M(t)^{+}$ denotes the pseudoinverse of $M(t)$.
\end{prop}
For the proof, see Appendix \ref{sect:Qpropproof}.  The basic idea behind (\ref{eqn:Qprop}) is to use $M(t)^{+} \mathbf{q}_{\ell}(t)$ to reconstruct/approximate the full density $\vvec{P(t)}$.  We propagate this forward in time using the Kronecker product expression, and then project back down to a reduced density matrix via $\widetilde{B}$.

In practice, we use an approximate pseudoinverse computed via the SVD $M(t) = \mathcal{U} \Sigma \mathcal{V}^\dagger$.  All matrices on the right-hand side depend on $t$, but we omit the time-dependence for simplicity.  Here $\Sigma$ is an $(\ell+1)K^2 \times N_C^2$ real matrix whose $N_C^2$ diagonal entries are the singular values of $M(t)$.  We take the singular values (all of which are nonnegative) to be listed in descending order, so that $\Sigma_{11}$ is the largest.  From this, we form an $N_C^2 \times (\ell+1)K^2$ matrix $\Sigma^{+}$ such that for all $1 \leq j \leq N_C^2$,
 \begin{equation}
 \label{eqn:sigreltol}
 \Sigma^{+}_{jj} = \begin{cases} 1/\Sigma_{jj} & \Sigma_{jj} > r_{\text{tol}} \Sigma_{11} \\
  0 & \text{otherwise}. \end{cases}
  \end{equation}
Here $r_{\text{tol}}$ is a user-defined relative tolerance, which we set to $10^{-12}$ in this paper.
All off-diagonal entries of $\Sigma$ and $\Sigma^{+}$ are zero.  We then set $M(t)^{+} = \mathcal{V} \Sigma^{+} \mathcal{U}^\dagger$.
 
We expect that memory ($\ell > 0$) is needed to capture the effect of the electron-electron terms in the Hamiltonian (\ref{eqn:molham}).  Suppose we switch off such terms and use a Hamiltonian that only contains one-electron terms such as the kinetic and electron-nuclear operators from (\ref{eqn:molham}) together with dipole moment operators.  Then we expect that the memoryless ($\ell=0$) case of (\ref{eqn:Qprop}) will propagate 1RDMs exactly.  In Appendix \ref{sect:oneelectron}, for the $\htwo$ and $\heh$ molecular systems studied in this paper, we carry out detailed calculations that prove this.

Returning to the setting where we retain all terms in the Hamiltonian (\ref{eqn:molham}), we find that our results improve considerably if we take additional steps to properly account for underlying physics in the form of known symmetries and constraints.  

\subsection{Incorporating Symmetries and Constraints}
\label{sect:symm}
The crux of the method presented above is the solution of the linear system (\ref{eqn:bigsys}) which aims to reconstruct $P(t)$ from present and past 1RDMs.  Though we have mentioned that $P(t)$ is always Hermitian and satisfies $\trace(P(t)) = 1$, we have not yet made use of these symmetries/constraints.  

In Figure \ref{fig:heh+_sto-3g_td_coeffs},  the second time-dependent coefficient (labeled $a_1(t)$ in the plot) is identically zero for \heh in STO-3G.  By (\ref{eqn:fulltdciden}), this implies that the second row and second column of $P(t)$ will be identically zero.  We can view this as an additional set of constraints on $P(t)$.

In Propositions \ref{Hermitian}, \ref{traceconstr}, and \ref{identzero}, we make use of, respectively, the Hermitian symmetry of $P(t)$, constant trace of $P(t)$, and identically zero elements in $P(t)$.  Each allows us to reduce the dimensionality of the linear system (\ref{eqn:bigsys}) that must be solved at each time step of (\ref{eqn:Qprop}).  Combining these results, we obtain the reduced-dimensional system (\ref{eqn:Qpropsymm}), which automatically preserves symmetries/constraints.

Recall (\ref{eqn:fulltdciden}), which implies that $P(t)$ is Hermitian for all $t$.  Thus for $i \neq j$, it is redundant to solve separately for $P_{ij}$ and $P_{ji}$, as we do in (\ref{eqn:approxinversion}).  This motivates replacing $\vvec(P(t))$ in (\ref{eqn:approxinversion}) with a representation involving only the upper-triangular (including diagonal) entries.
\begin{defn}[Basis for Space of Hermitian Matrices]
\label{Hbasis}
Let $\mathbb{H}$ be the vector space of $N_C \times N_C$ Hermitian matrices. For $j=1$ to $j=N_C$, let the matrix $S^j$ be purely diagonal with $S^j_{kk} = \delta_{jk}$.  Now let $\boldsymbol{\tau}(j) = (\tau_1(j),\tau_2(j))$ be the $j$-th off-diagonal, strictly upper-triangular indicial pair, i.e., the $j$-th element of
$(1,2), \ldots, (1,N_C), (2,3), \ldots, (2,N_C), \ldots, (N_C-1,N_C)$.
For $j=1$ to $j=N_C(N_C-1)/2$, let $S^{N_C+j}$ be one at indices $(\tau_1(j),\tau_2(j))$ and $(\tau_2(j),\tau_1(j))$; zero otherwise; also, let $S^{N_C(N_C+1)/2 + j}$ be $i$ and $-i$ at indices $(\tau_1(j),\tau_2(j))$ and $(\tau_2(j),\tau_1(j))$, respectively; zero otherwise.
\end{defn}
For example, when $N_C=2$, we obtain the following basis of $\mathbb{H}$:
\[
S = \{S^1, S^2, S^3, S^4\} = \left\{ \begin{bmatrix} 1 & 0 \\ 0 & 0 \end{bmatrix},
\begin{bmatrix} 0 & 0 \\ 0 & 1 \end{bmatrix},
\begin{bmatrix} 0 & 1 \\ 1 & 0 \end{bmatrix},
\begin{bmatrix} 0 & i \\ -i & 0 \end{bmatrix} \right\}.
\]
We now present a sequence of results that aim to modify the raw scheme (\ref{eqn:Qprop}) to account for known symmetries and constraints.  All proofs in this subsection are deferred to Appendix \ref{sect:SymmConstrProofs}.
\begin{lem}
\label{HermLem}
$S = \{S^j\}_{j=1}^{N_C^2}$ is a basis for $\mathbb{H}$.
\end{lem}
\begin{prop}
\label{Hermitian}
There exists an $N_C^2 \times N_C^2$ matrix $\widetilde{S}$ such that (\ref{eqn:bigsys}) is equivalent to
\begin{equation}
\label{eqn:bigsys2}
M(t) \widetilde{S} \mathbf{x}(t) = \mathbf{q}_{\ell}(t).
\end{equation}
Hence we need only solve for $\mathbf{x}(t) \in \mathbb{R}^{N_C^2}$, rather than solving (\ref{eqn:bigsys}) for $\vvec(P(t)) \in \mathbb{C}^{N_C^2}$.
\end{prop}
\begin{prop}[Preserving Constant Trace]
\label{traceconstr}
Accounting for $\trace(P(t))=1$, the linear system (\ref{eqn:bigsys2}) can be expressed equivalently as
\begin{equation}
\label{eqn:bigsys3}
M'(t) \mathbf{x}_{-N_C}(t) = \mathbf{q}_{\ell}(t) - [M(t) \widetilde{S}]_{:, N_C},
\end{equation}
where $M'(t)$ is of size $(\ell+1) K^2 \times (N_C^2 - 1)$ and $\mathbf{x}_{-N_C}(t) \in \mathbb{R}^{N_C^2 - 1}$ denotes the result of deleting the $N_C$-th entry from $\mathbf{x}(t)$.
\end{prop}
\begin{prop}[Eliminating Zero Elements of the Full Density Matrix]
\label{identzero}
For a given molecular system, if we know \emph{a priori} (e.g., based on symmetries of the molecule) that $D$ entries of $P(t)$ will be zero, then (\ref{eqn:bigsys3}) is equivalent to
\begin{equation}
\label{eqn:bigsys4}
M''(t) \mathbf{x}_{-}(t) = \mathbf{q}_{\ell}(t) - [M(t) \widetilde{S}]_{:, N_C} =: \mathbf{b}_{\ell}(t),
\end{equation}
where $M''(t)$ has shape $(\ell+1)K^2 \times (N_C^2 - 1 - D)$ and $\mathbf{x}_{-}(t) \in \mathbb{R}^{N_C^2 - 1 - D}$.
\end{prop}
We can now use Propositions \ref{Hermitian}, \ref{traceconstr}, and \ref{identzero} to give an improved version of (\ref{eqn:Qprop}), the delay equation to propagate the 1RDMs $Q(t)$.
\begin{prop}[Constraint-Preserving Propagation of 1RDMs]
\label{Qpropsymm}
Suppose the hypotheses of Propositions \ref{Qprop}, \ref{traceconstr}, and \ref{identzero} are satisfied.  Suppose that $M''(t)$ as constructed in Proposition \ref{identzero} has full column rank.  Recall $\mathbf{b}_{\ell}(t)$ was defined by the right-hand side of (\ref{eqn:bigsys4}).
Then there exists an affine map $\mathcal{A} : \mathbb{R}^{N_C^2 - 1 - D} \to \mathbb{R}^{N_C^2}$ such that the 1RDM $Q(t)$ satisfies the delay equation
\begin{equation}
\label{eqn:Qpropsymm}
\vvec(Q(t+1)) =  \widetilde{B} \left( \exp(i H(t)^T \Delta t) \otimes \exp(-i H(t) \Delta t) \right) \underbrace{ \widetilde{S} \mathcal{A} M''(t)^{+} \mathbf{b}_{\ell}(t) }_{\vvec(P(t))}.
\end{equation}
\end{prop}
At time $t$, the basic idea of (\ref{eqn:Qpropsymm}) is to solve the linear system (\ref{eqn:bigsys4}) for $\mathbf{x}_{-}(t) = M''(t)^{+} \mathbf{b}_{\ell}(t)$.  Given $\mathbf{x}_{-}(t)$, we use our knowledge of \emph{a priori} zero entries (if such entries exist) and constant trace to compute the full vector $\mathbf{x}(t)$---we formalize this operation as $\mathbf{x}(t) = \mathcal{A} \mathbf{x}_{-}(t)$ where $\mathcal{A}$ is affine.  Equipped with the real $N_C^2$-dimensional vector $\mathbf{x}(t)$, we compute $\widetilde{S} \mathbf{x}(t)$, our approximation/reconstruction of the vectorized complex density $\vvec(P(t))$.  We propagate this forward by one time step using the Kronecker product expression, and then compute the 1RDM using $\widetilde{B}$.

In practice, we compute an approximate pseudoinverse $M''(t)^{+}$ using the procedure described after (\ref{eqn:Qprop}); we threshold singular values via (\ref{eqn:sigreltol}) with the relative tolerance $r_{\text{tol}} = 10^{-12}$.  While this enables us to use (\ref{eqn:Qpropsymm}) even when $M''(t)$ is rank-deficient and/or ill-conditioned, the method is no longer exact, motivating the numerical experiments in Section \ref{sec:results}.

By building the trace constraint into the solution process, we guarantee that each $P(t)$ on the right-hand side of (\ref{eqn:Qpropsymm}) has trace equal to one.  By properties established in Section \ref{sect:Btendef}, this implies that (\ref{eqn:Qpropsymm}) preserves $\trace(Q(t)) = N$ for all $t$, as in (\ref{eqn:traceQ}).

Since $M''(t)$ has $D+1$ fewer columns than $M(t)$, we expect $M''(t)$ to have a better condition number than $M(t)$, improving our ability to propagate $Q(t)$ accurately.

For the above reasons, (\ref{eqn:Qpropsymm}) should be seen as a replacement of (\ref{eqn:Qprop}); it is the method actually used in the numerical results described below.  



\subsection{Characterizing the Rank of $M(t)$}
\label{sect:rankM}
Let us reexamine $M(t)$ from (\ref{eqn:bigsys}). The purpose of including additional blocks $(\mathcal{D}_1, \ldots, \mathcal{D}_{\ell})$ in $M(t)$ is to increase the rank of the resulting matrix.  Having found numerically that  for sufficiently large $\ell$,  $M(t)$ achieves rank equal to its number of columns, we briefly examine the question of whether (and under what hypotheses) we can prove this.

As $H(t)$ is always Hermitian, $C_m(t)^T$ and $A_m(t)$ defined by (\ref{eqn:CjAj}) will always be unitary.   By (\ref{eqn:Bsubj}),  $\mathcal{D}_j$ consists of $\widetilde{B}$ multiplied by a unitary matrix.  Recall that $\widetilde{B}$ is of size $K^2 \times N_C^2$ with $K < N_C$.  
If $\rank(\widetilde{B}) = K^2$, then $\rank(\mathcal{D}_j) = K^2$ as well.

A simple observation is that if $(C_m(t)^T \otimes A_m(t))$ happens to be the identity matrix, then $M(t)$ will consist of $\widetilde{B}$ repeated $(\ell+1)$ times; in this scenario, $\rank(M(t)) = K^2$ for any $\ell$.  Thus the assumption that adding blocks \emph{increases} the rank of $M(t)$ stems from our intuition that as $t$ increases, the system's dynamics will move  $C_m(t)^T \otimes A_m(t)$ \emph{away} from the identity.


A quantum system's dynamics are determined by its Hamiltonian $H(t)$.  If we examine (\ref{eqn:matrixham}), we see three objects about which we have made no assumptions thus far: the static Hamiltonian $H_0$, the external field strength $f(t)$, and the dipole moment matrix $M_{\text{dip}}$.  We have yet to find hypotheses on $H_0$, $f(t)$ and $M_{\text{dip}}$ that guarantee that, for sufficiently large $\ell$, the matrix $M(t)$ has rank equal to $N_C^2$ for all $t \geq 1$.  

 To motivate future work, we examine the $\ell=1$ case in more detail.  For a matrix $Z$ and a set of positive integers $G$, let $Z^G$ (respectively, $Z^{-G}$) denote the matrix formed by selecting only those columns of $Z$ with indices in (respectively, not in) $G$.  Then we can give a condition under which $M(t)$ has full rank.
\begin{prop}
\label{Schur}
Let $\ell = 1$ and consider $M(t)$ from (\ref{eqn:bigsys}).  Suppose $2K^2 \leq N_C^2$ and $\rank(\widetilde{B}) = K^2$. Equipped with column indices $G \subset \{1, 2, \ldots, N_C^2\}$, let us form the block matrix $\displaystyle \widehat{M} = \begin{bmatrix}
\widetilde{B}^G & \widetilde{B}^{-G} \\
\mathcal{D}_1^G & \mathcal{D}_1^{-G}
\end{bmatrix}$. If there exists $G$ of size $|G| = K^2$ such that (i) $\widetilde{B}^G$ has rank $K^2$ and (ii) the Schur complement of $\widetilde{B}^G$ of $\widehat{M}$ has rank $K^2$, then $M(t)$ has rank $2 K^2$.
\end{prop}
For the proof, see Appendix \ref{sect:Schurproof}.  More generally, all we need is for $\rank(M(t))$ to increase each time we increase $\ell$ by $1$; an increase of precisely $K^2$ is unnecessary.  If we can reorder the columns of $M(t)$ to form the block matrix $\widehat{M}$ with the desired properties, then we can apply the Guttman rank additivity formula\cite{Zhang2005} to compute the rank of $M(t)$.  Both the above lemma and the Guttman formula use the Schur complement.

\begin{figure}[t]
    \centering
     \includegraphics[width=0.75\linewidth,clip,trim=0 25 0 10]{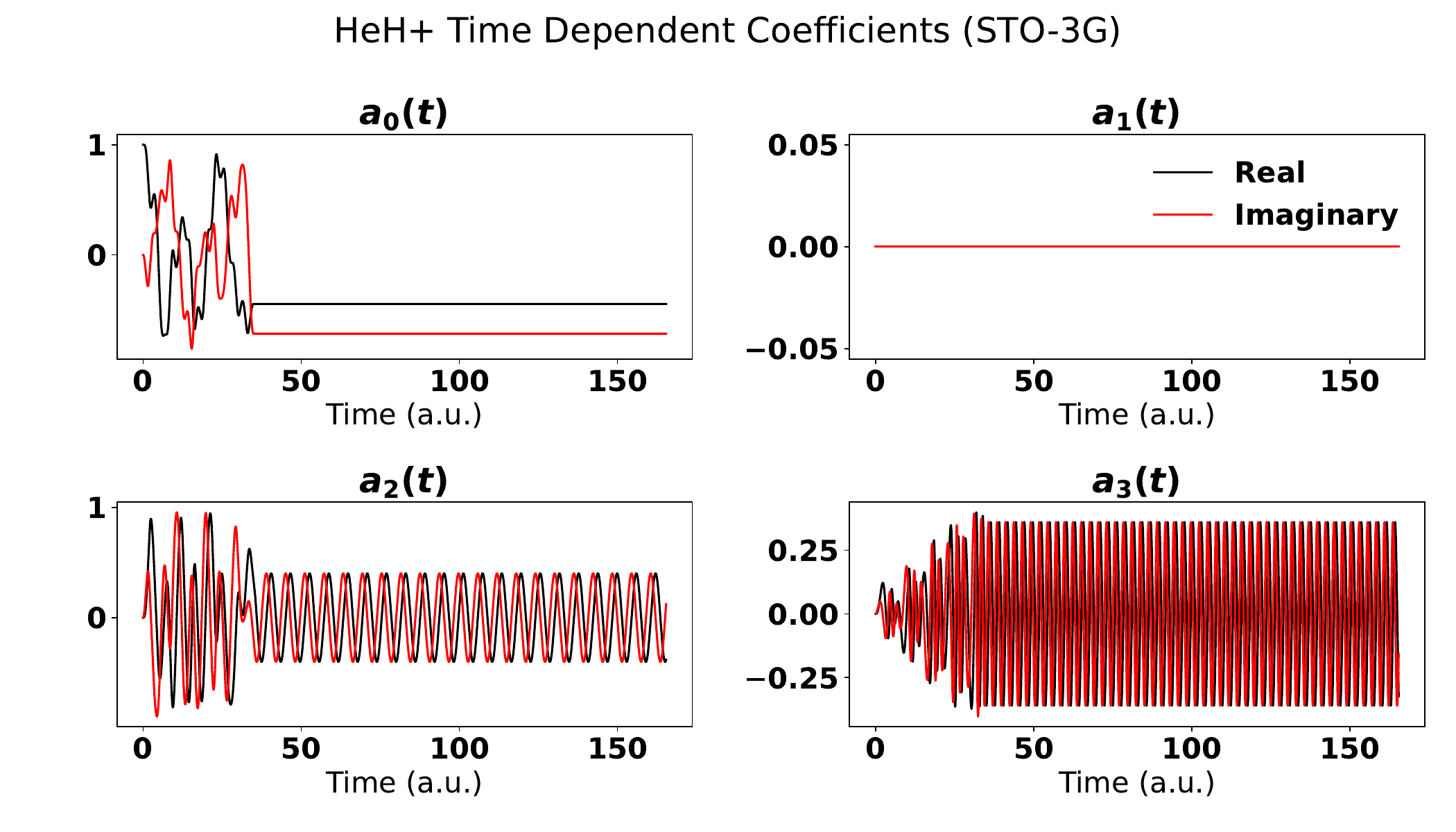}
    \caption{For the molecule \heh in the  STO-3G basis with $\Delta t = 0.008268$ a.u., we apply the TDCI procedure from Section \ref{sect:procedure} to compute the time-dependent coefficients $\mathbf{a}(t)$ for 20,000 time steps. Because $a_1(t) \equiv 0$, in the associated full TDCI density matrix $P(t) = \mathbf{a}(t) \mathbf{a}(t)^\dagger$, the second column and second row vanish identically.  The consequences of this are explained in Section \ref{sect:symm}.}
    \label{fig:heh+_sto-3g_td_coeffs}
\end{figure}


\section{Numerical Results}
\label{sec:results}
\emph{How much memory $k \ell \Delta t$ do we need in order to compute accurate 1RDMs for various molecular systems?} To answer this question, we propagate 1RDMs using (\ref{eqn:Qpropsymm}) and compare the results against ground truth 1RDMs.  We now detail how we conduct these tests numerically.   Here $\ell$ is the raw number of memory steps (and $\ell+1$ is the number of blocks) that we find in the vector $\mathbf{q}_{\ell}(t)$ in (\ref{eqn:Mtql}), $k$ is the stride, and $\Delta t$ is the time step.  We see that $k \ell \Delta t$, which we refer to below as \emph{memory used} or \emph{total memory}, appears in (\ref{eqn:Mtql}) and represents the maximum time delay used by our propagation scheme, measured in atomic units (a.u.) rather than by number of time steps.

\subsection{Procedure}
\label{sect:procedure}
Because computing full CI solutions is expensive, and as this is the first time we are studying (\ref{eqn:Qpropsymm}), we focus attention on $\heh$ and $\htwo$, both with $N = 2$ electrons, in each of two basis sets: STO-3G and 6-31G.   Recall that full density matrices $P(t)$ are $N_C \times N_C$ while 1RDMs $Q(t)$ are $K \times K$. For STO-3G, $N_C = 4$ and $ K = 2$; for 6-31G, $N_C = 16$ and $K = 4$.  

Computing ground truth 1RDMs $Q_{\text{true}}(t)$ requires some work.  For a given molecular system (choice of molecule plus basis set), we apply the CASSCF method \cite{Roos1980} with two electrons and all orbitals included in the active space to compute what is effectively a full CI solution.  That is, the solution we obtain is equivalent to using all possible Slater determinants in (\ref{eqn:CIbasis}).  We obtain from this procedure the data needed to perform a TDCI calculation, including sets of molecular orbitals and CI basis functions, the diagonal Hamiltonian matrix $H_0$ and the dipole moment matrix $M_{\text{dip}}$---see Section \ref{sect:TDCI} for definitions of these objects.


For each molecular system, we apply a sinusoidal external electric field.  Namely, the Hamiltonian takes the form $H(t) = H_0 + f(t) M_{\text{dip}}$, for $f(t) = A\sin \omega t$. For the results reported here, for both molecules, we fix $A = 0.5$.
For $\heh$, we set $\omega = 0.9 \text{(a.u.)}^{-1}$ and for $\htwo$, we set $\omega = 1.5 \text{(a.u.)}^{-1}$.  We apply the electric field for five cycles (approximately 35 a.u. and 21 a.u. of total time for $\heh$ and $\htwo$ respectively), beginning at time $t = 0$.  We chose these field parameters because, in our initial explorations with the STO-3G basis set, they produced electron dynamics that were the most difficult to fit.

Let $\mathbf{a}(0)$ be the unit vector such that $a_1(0) = 1$ and $a_j = 0$ for all $j \geq 2$.  Then for a fixed time step $\Delta t > 0$ and a prescribed external field strength $f(t)$, we solve (\ref{eqn:matrixschro}) via the scheme $\mathbf{a}(t + \Delta t) = \exp(-i H(t) \Delta t) \mathbf{a}(t)$.  From this, we compute full density matrices $P_{\text{true}}(t)$ via (\ref{eqn:fulltdciden}).  We use the methods from Sections \ref{sect:rdmdef} and \ref{sect:Btendef} to compute the tensor $B$.  With $P_{\text{true}}(t)$ and $B$, we compute $Q_{\text{true}}(t)$ via (\ref{eqn:rdmdef}).  The only parameters controlling the accuracy of $Q_{\text{true}}(t)$ are the choice of basis set and the time step $\Delta t$.


Computing 1RDMs with the methods developed in this paper requires setting two additional parameters: the delay length $\ell$ and the stride $k$.  Recall from (\ref{eqn:Bsubj}), (\ref{eqn:backwardQt}), and (\ref{eqn:bigsys}) that with delay length $\ell$ and stride $k$, $Q(t+\Delta t)$ depends on $Q(t)$ and a total of $k\ell$ previous 1RDMs: $Q(t - j k \Delta t)$ for $j = 1, 2, \ldots, \ell$.  Unless stated otherwise, we set $k=1$.  Once we have set $\ell$ and $k$, we use (\ref{eqn:Qpropsymm}) to propagate 1RDMs forward in time.  As with any delay evolution equation, we must prescribe initial values on a segment.  For $t < k \ell \Delta t$, we set $Q(t) = Q_{\text{true}}(t)$.  We then use (\ref{eqn:Qpropsymm}) to propagate forward in time for $t \geq k \ell \Delta t$.  In what follows, we use $Q_{\text{model}}(t)$ to denote the resulting solution.

We compute $Q_{\text{true}}(t)$ and $Q_{\text{model}}(t)$ at times $t=j \Delta t$ up to $j = n_{\text{steps}}$.  The final time $T = n_{\text{steps}} \Delta t$ is $165.36$ a.u.  This corresponds to either $n_{\text{steps}} = 2000$ with $\Delta t = 0.082680$ a.u., or $n_{\text{steps}} = 20000$ with $\Delta t = 0.008268$ a.u.

To measure the mismatch between ground truth and model 1RDMs, we use the root mean-squared error (RMSE) over time steps $\ell+1, \ldots, n_{\text{steps}}$.  With the Frobenius norm $\| \cdot \|_F$,
\begin{equation}
\label{eqn:RMSE}
\text{RMSE} = \sqrt{\frac{1}{K^2} \frac{1}{n_{\text{steps}-\ell}} \sum_{j=\ell+1}^{n_{\text{steps}}} \left\| Q_{\text{model}}(j \Delta t) - Q_{\text{true}}(j \Delta t) \right\|_F^2}.
\end{equation}
We also use the mean absolute error (MAE) at each time $t$:
\begin{equation}
\label{eqn:MAE}
\text{MAE}(t) = \frac{1}{K^2} \sum_{i,j=1}^{K} \left| Q^{\text{model}}_{ij}(t) - Q^{\text{true}}_{ij}(t) \right|.
\end{equation}
\subsection{Computational Resources}
All computations were carried out on login and GPU nodes of the NERSC Perlmutter cluster \footnote{\url{https://docs.nersc.gov/systems/perlmutter/architecture/}}.  All source code required to reproduce our results is available in an online repository \footnote{\url{https://github.com/hbassi/1rdm_memory_model}}. Propagation scheme (\ref{eqn:Qpropsymm}) was implemented in Python using JAX (for $k=1$) and CuPy (for $k \geq 2$) \cite{jax2018github,cupy_learningsys2017}. For calculations in the STO-3G basis set, the wall-clock run time of the propagation scheme ranged from $45$ seconds to $20$ minutes. In contrast, for calculations in the 6-31G basis set, the wall-clock run time of the propagation scheme took between $15$ and $45$ minutes. Run times depend primarily on the parameter $\ell$; as $\ell$ increases, we need more Kronecker products to propagate $Q(t)$ via (\ref{eqn:Qpropsymm}).

\subsection{Results and Discussion}
\label{sect:results}
Our main numerical results, summarized in Figures \ref{fig:MAE_total_results_with_striding}-\ref{fig:residual_error}, quantify how propagation of 1RDMs depends on total memory $k \ell \Delta t$, time step $\Delta t$, and stride $k$.  A key goal is to gain insight on how much total memory is required for each of our four molecular systems ($\heh$ and $\htwo$ in STO-3G and 6-31G) to achieve numerically acceptable levels of accuracy.


\begin{figure}
    \centering
    \includegraphics[width=\textwidth,clip,trim=0 25 0 10]{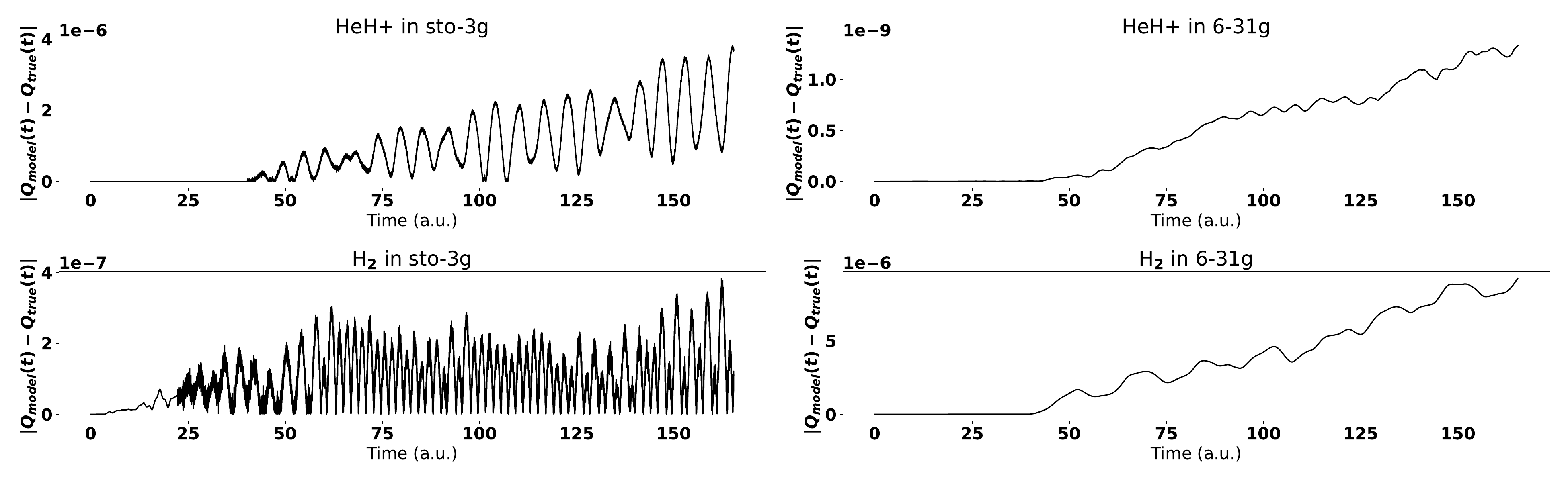}
    \caption{For each of four molecular systems we tested, with sufficiently large total memory ($k \ell \Delta t$), the 1RDMs produced by our propagation scheme (\ref{eqn:Qpropsymm}) agree closely with ground truth 1RDMs at all times $t$.  Each plot shows $\text{MAE}(t)$ defined by (\ref{eqn:MAE}) versus physical time in atomic units (a.u.)  From left to right, top to bottom, the maximum MAE values are bounded by $4 \times 10^{-6}$, $1.5 \times 10^{-9}$, $4 \times 10^{-7}$, and $10^{-5}$; the corresponding values for total memory are $5.3$, $6.6$, $0.6$, and $13$ a.u., respectively.  Note that strides of $k \geq 2$ were used for all systems except $\htwo$ in STO-3G.  For details of all other parameters, consult the text in Section \ref{sect:results}.}
    \label{fig:MAE_total_results_with_striding}
\end{figure}

\begin{figure}
    \centering
    \includegraphics[width=1.0\textwidth,clip,trim=0 25 0 10]{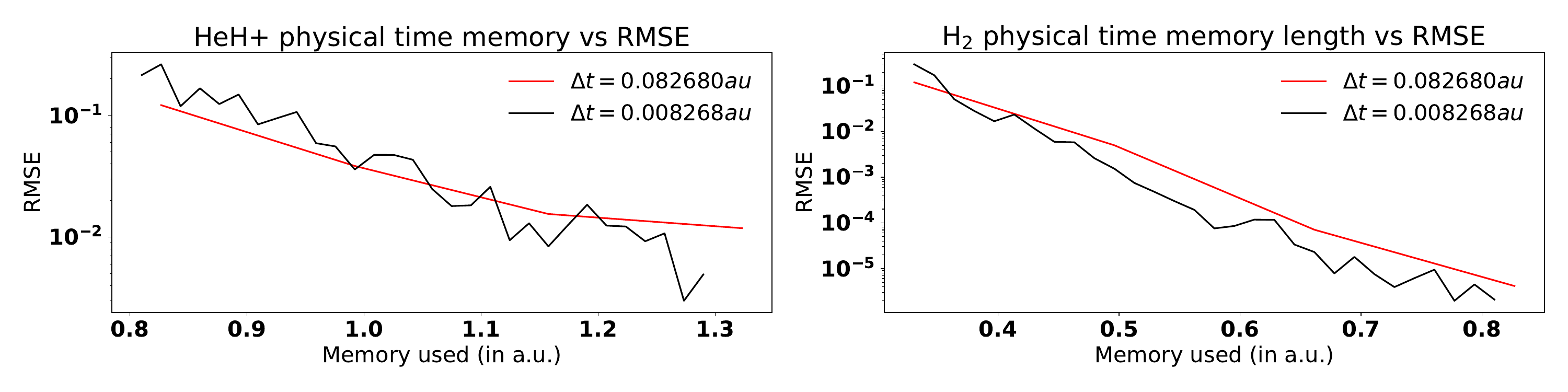}
    \caption{For each molecule in the STO-3G basis set, we fix the stride $k=1$.  For each of two choices of $\Delta t$, we plot the RMSE (\ref{eqn:RMSE}) as a function of total memory $k \ell \Delta t$.  Note that for both molecules, the RMSE approaches zero at similar rates, regardless of $\Delta t$.  This indicates that our scheme (\ref{eqn:Qpropsymm}) identifies a physical time scale of 1RDM memory-dependence.}
    \label{fig:timestep_effect}
\end{figure}

\subsubsection{With sufficient memory, our time-delay propagation scheme yields highly accurate 1RDMs} Figure \ref{fig:MAE_total_results_with_striding} shows that, with sufficiently large total memory ($k \ell \Delta t$), the 1RDMs produced by (\ref{eqn:Qpropsymm}) agree closely with ground truth 1RDMs at all times $t$.  For the results in Figure \ref{fig:MAE_total_results_with_striding}, $\Delta t = 0.008268$ a.u. and $n_{\text{steps}} = 20000$.  Here our goal was to choose parameters such that $\text{MAE}(t)$---as defined by (\ref{eqn:MAE})---drops below $O(10^{-6})$ for each molecular system.  This level of accuracy demonstrates that the scheme (\ref{eqn:Qpropsymm}) is indeed capable of computing 1RDMs that are in close quantitative agreement with ground truth 1RDMs.  For $\heh$ in STO-3G and 6-31G, we used parameters of $\ell=160$, $k=4$ and $\ell=160$, $k=5$, respectively. For $\htwo$ in STO-3G and 6-31G, we used parameters of  $\ell = 72$, $k = 1$ and $\ell = 220$, $k = 7$, respectively.  These values lead to total memory ($k \ell \Delta t$) of $5.3$, $6.6$, $0.6$, and $13$ a.u. for the four molecular systems considered.

\subsubsection{Results are approximately invariant under time step refinement}\label{sect:timesteprefinement} For each molecular system, we seek a physical time scale that governs memory-dependence of 1RDMs.  In order to rule out that the computed time scales are numerical artifacts, we check whether our results depend on $\Delta t$.  For the two molecules in STO-3G, we fix the stride to be $k=1$ and repeatedly run simulations---using (\ref{eqn:Qpropsymm})---at increasing values of $\ell$.  We do this for both values of $\Delta t$ and $n_{\text{steps}}$ mentioned above, such that $n_{\text{steps}} \Delta t = T = 165.36$ a.u.  For each choice of $\ell$ and $\Delta t$, we compute the RMSE (\ref{eqn:RMSE}).  We plot in Figure \ref{fig:timestep_effect} the RMSE versus total memory ($k \ell \Delta t$) for both choices of $\Delta t$.  For both molecules, and for both choices of $\Delta t$, the RMSE approaches zero as total memory increases.  Moreover, the rate at which RMSE approaches zero is approximately independent of $\Delta t$.  This indicates that our scheme  (\ref{eqn:Qpropsymm}) identifies a physical time scale of memory-dependence for each molecular system.


\begin{figure}[tp]
    \begin{subfigure}{\textwidth}
    \includegraphics[width=1.0\textwidth,clip,trim=0 25 0 10]{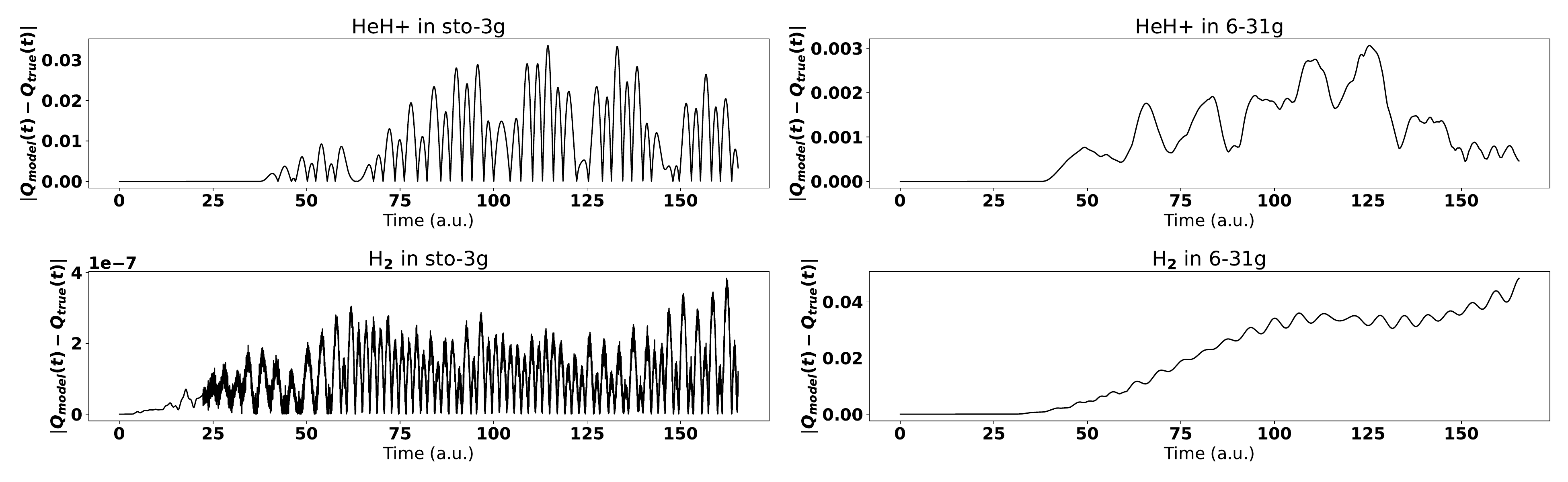}
    \caption{We plot $\text{MAE}(t)$ (black) defined by (\ref{eqn:MAE}) for each of four molecular systems.}
    \label{fig:5a}
    \end{subfigure}
    \begin{subfigure}{\textwidth}
    \includegraphics[width=1.0\textwidth,clip,trim=0 25 0 10]{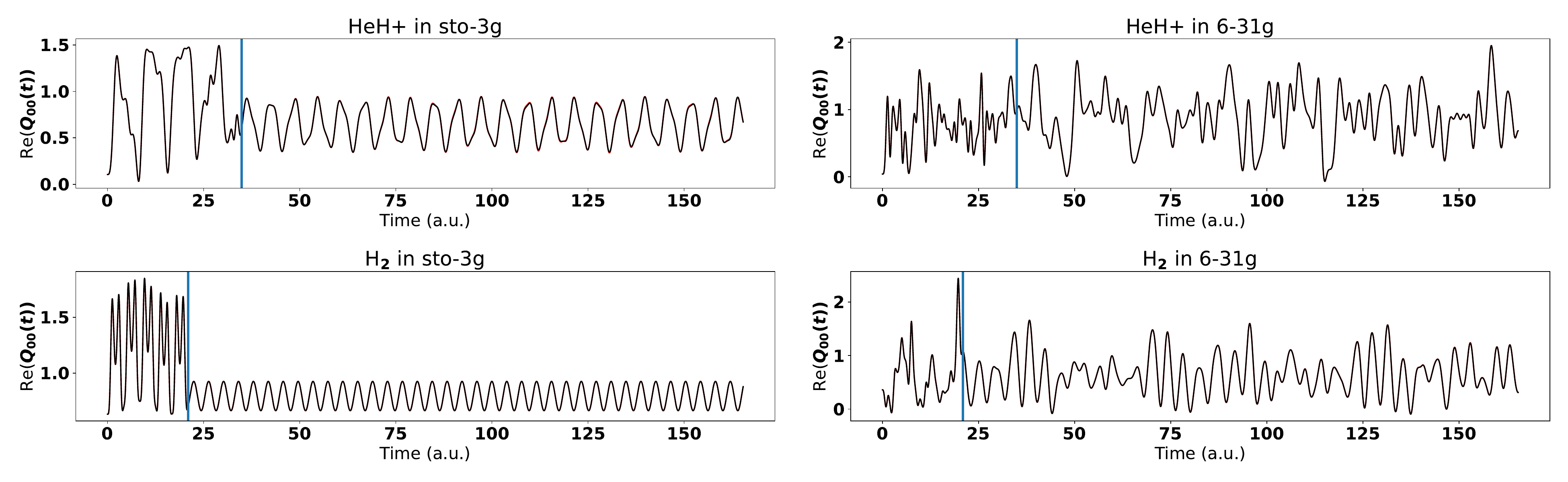}
    \caption{We plot the upper-left components of the matrices $Q_{\text{true}}(t)$ (black) and $Q_{\text{model}}(t)$ (red) for each of four molecular systems.}
    \label{fig:5b}
    \end{subfigure}
    \caption{For each of the four molecular systems we tested, we use less total memory than in Figure \ref{fig:MAE_total_results_with_striding}, and still find good agreement between modeled and true 1RDMs.  In particular, note that the red (model) and black (true) curves in Figure \ref{fig:5b} are nearly indistinguishable.  In Figure \ref{fig:5b}, the vertical blue line shows the time at which the electric field is turned off. Plots for the same molecular system were produced using the same parameter $\ell$; for these values and other details, see Section \ref{sect:results}.}
    \label{fig:MAE_total_results_no_striding}
\end{figure}

\subsubsection{With less total memory, 1RDMs retain qualitative accuracy} 
In Figure \ref{fig:MAE_total_results_no_striding}, we visualize the consequences of using less total memory for all molecular systems.  We set $k=1$, $\Delta t = 0.008268$ a.u., $n_{\text{steps}} = 20000$.  For $\heh$ in STO-3G and 6-31G, we set $\ell = 160$ (total memory of $1.3$ a.u.)  For $\htwo$ in STO-3G and 6-31G, we set $\ell = 72$ (total memory of $0.6$ a.u.) and $\ell=850$ (total memory of $7$ a.u.), respectively.    As compared with Figure \ref{fig:MAE_total_results_with_striding}, we use less total memory, incurring larger errors $\text{MAE}(t)$ at each time $t$, as shown in Figure \ref{fig:5a}.  Nonetheless, in Figure \ref{fig:5b}, we see close agreement between the upper-left components of the 1RDMs $Q_{\text{true}}(t)$ (black) and $Q_{\text{model}}(t)$ (red) at each $t$. Other components of the 1RDMs feature similar agreement.  We hypothesize that the qualitative match between $Q_{\text{true}}(t)$ and $Q_{\text{model}}(t)$ seen here is sufficient to calculate quantities of chemical/physical interest.

\begin{figure}
    \centering
    \includegraphics[width=1.0\textwidth,clip,trim=0 25 0 10]{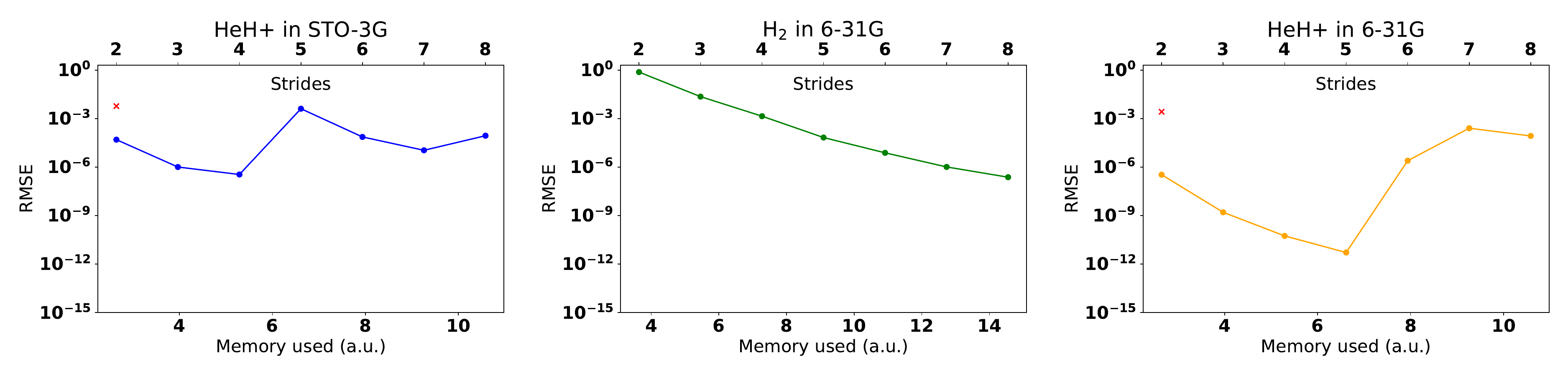}
    \caption{For the three molecular systems with higher error in Figure \ref{fig:MAE_total_results_no_striding}, we keep $\ell$ fixed and consider $k \in \{2,3,4,5,6,7,8\}$.  For each value of $k$, we run our scheme (\ref{eqn:Qpropsymm}) and compute the RMSE (\ref{eqn:RMSE}).  We plot the RMSE both against total memory $k \ell \Delta t$ (lower horizontal axis) and stride $k$ (upper horizontal axis).  In the left and right panels, we plot the baseline $k=1$ RMSE using a single red dot.  In the center panel, the baseline RMSE with $k=1$ is well above the vertical scale of the plot and hence not shown.  We see that increasing stride can reduce the RMSE by multiple orders of magnitude for each molecular system. }

    \label{fig:striding_results}
\end{figure}



\subsubsection{Using stride $k \geq 2$ leads to reduced propagation error} 
The dimensions of $M(t)$ and $\mathbf{q}_{\ell}(t)$ in (\ref{eqn:bigsys}) depend on $\ell$ but not $k$.  Thus for fixed $\ell$, increasing $k$ enables one to increase total memory \emph{without} a proportional increase in computational cost.

We theorized above that, for each molecular system, there exists a physical time scale of memory $k \ell \Delta t$.  One consequence of this theory is that if we fix $\ell$, fix $\Delta t$ and increase $k$, the RMSE (\ref{eqn:RMSE}) should decrease.  Here we test this idea.  We fix $\Delta t = 0.008268$ a.u.  For $\heh$, we set $\ell = 160$; for $\htwo$, we set $\ell = 220$.  For each $k \in \{2, 3, 4, 5, 6, 7, 8\}$, we repeatedly run propagation scheme (\ref{eqn:Qpropsymm}) and compute the RMSE (\ref{eqn:RMSE}).  Note that we omit $\htwo$ in STO-3G as the RMSE was already $O(10^{-4})$ with $k=1$ as shown in Figure \ref{fig:timestep_effect}.

We plot the resulting RMSE values against both total memory $k \ell \Delta t$ (lower horizontal axis) and stride $k$ (upper horizontal axis) in Figure \ref{fig:striding_results}.  We also plot, in the left and right panels, the RMSE associated with the baseline value $k=1$.  In the center panel, the corresponding RMSE for $k=1$ lies well above the vertical scale of the plot.  We see clear evidence that using strides $k \geq 2$ can reduce the RMSE by orders of magnitude compared with $k=1$.  Given the $\ell$ values quoted above, each corresponding unit increase in $k$ leads to a large increase in total memory $k \ell \Delta t$.  Thus increasing $k$ is an efficient way to boost the total memory used by the scheme (\ref{eqn:Qpropsymm}), without incurring increased computational cost.  Overall, these results support the view that scheme (\ref{eqn:Qpropsymm}) can be made arbitrarily accurate by tuning the total memory $k \ell \Delta t$.

\begin{figure}
    \centering
    \includegraphics[width=1.0\textwidth,clip,trim=0 15 0 10]{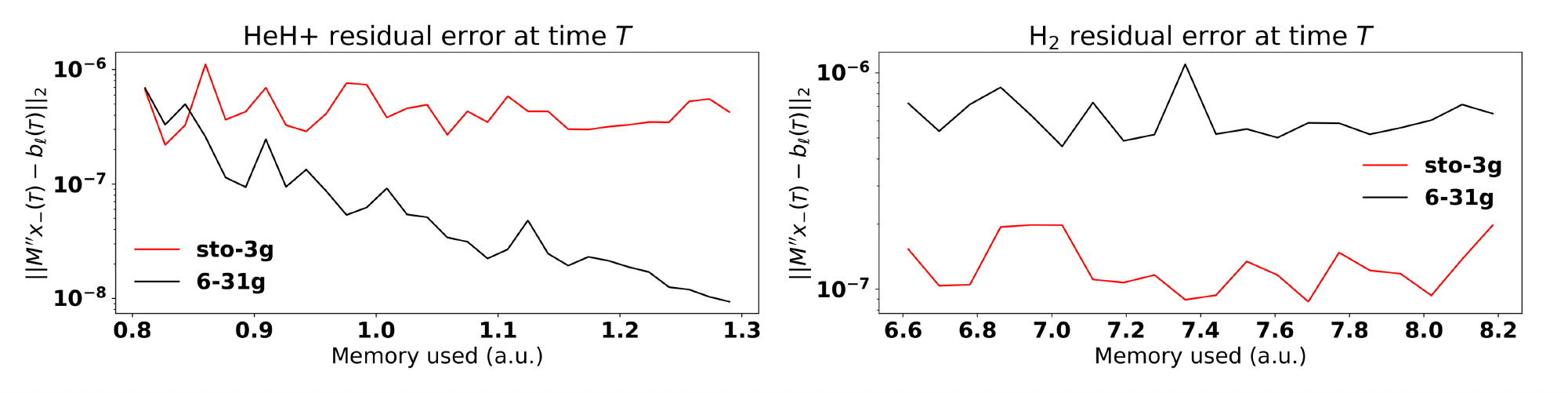}
    \caption{When we use (\ref{eqn:Qpropsymm}) to propagate 1RDMs $Q(t)$ forward in time, we use an approximate pseudoinverse $M''(t)^{+}$ to solve the linear system (\ref{eqn:bigsys4}) at each time step.  For each simulation run corresponding to a different value of $\ell$ and fixed values of $k=1$ and $\Delta t=0.008268$ a.u., we compute the residual error (\ref{eqn:residual}) at the final time $T$.  For each of our four molecular systems, we plot this error versus memory used ($k \ell \Delta t$).  The results show that the approximate pseudoinverse yields tolerable residual errors for all systems.}
    \label{fig:residual_error}
\end{figure}



\subsubsection{Residual errors for the linear system (\ref{eqn:bigsys4}) are reasonable}
Scheme (\ref{eqn:Qpropsymm}) requires solving 
(\ref{eqn:bigsys4}) for $\mathbf{x}_{-}(t) = M''(t)^+ \mathbf{b}_{\ell}(t)$ at each time step.  Recall from the discussions after (\ref{eqn:Qprop}) and (\ref{eqn:Qpropsymm}) that we use an approximate pseudoinverse $M''(t)^{+}$, resulting in an approximate solution $\mathbf{x}_{-}(t)$ at each $t$.  We compute and record the \emph{residual} at time $t$:
\begin{equation}
\label{eqn:residual}
\text{residual}(t) = \|M''(t) \mathbf{x}_{-}(t) - \mathbf{b}_{\ell}(t) \|_2.
\end{equation}
To measure how well our approximate pseudoinverse solves the linear system (\ref{eqn:bigsys4}), we ran a numerical experiment.  With fixed time step $\Delta t = 0.008268$ a.u. and stride $k=1$, and for increasing values of $\ell$, we ran repeated simulations of (\ref{eqn:Qpropsymm}) from $t=0$ to $t=T=165.36$ a.u. by taking $n_{\text{steps}} = 20000$.  

In Figure \ref{fig:residual_error}, for each of our four molecular systems, we plot $\text{residual}(T)$  as a function of total memory $k \ell \Delta t$.  For the values of total memory shown, the residual errors are in the range $O(10^{-6})$ to $O(10^{-8})$.  Based on all our other results , this is small enough for 1RDM propagation to succeed despite the use of an approximate pseudoinverse $M''(t)^{+}$.  However, this corresponds to condition numbers of $M''(t)$ in the range of $10^8$ to $10^{10}$; the plotted residual errors are neither unacceptably large nor desirably small (i.e., closer to machine precision).

Though $M''(t)$ may have full rank, for particular choices of system and simulation parameters, it is clearly ill-conditioned.  This is true despite the work we carried out in Section \ref{sect:methods} to incorporate Hermitian symmetry, constant trace, and identically zero elements into (\ref{eqn:bigsys4})---these steps reduced the number of columns of $M''(t)$.  This motivates future work to incorporate further properties of $P(t)$, such as its definition as a rank-$1$ matrix, into our propagation scheme for 1RDMs. If we can solve for still fewer elements of $P(t)$, then we can eliminate more columns of $M''(t)$, decreasing its condition number.  We describe one such path, leading to a nonlinear model, in Section \ref{sect:conclusion}.

\subsubsection{Results align with a prior study of memory in time-dependent current density functional theory} For the homogeneous electron gas perturbed by an electric field, analytic methods have been used\cite{kurzweil2004time} to derive exchange-correlation functionals with memory for time-dependent current density functional theory (TDCDFT), a theory that is distinct from TDDFT.  The memory kernels for the derived exchange-correlation functionals decay on a physical time scale of $10$-$20$ a.u.\cite{kurzweil2004time}  This is within the same order of magnitude as all results reported here, despite the differences in methods and systems.

\subsubsection{Further work is required to distill the relationship between memory-dependence and basis set}  For each molecular system, as we increase the size of the basis set, we expect to compute more accurate electron dynamics.  In the course of studying propagation method (\ref{eqn:Qpropsymm}), we found a complex interrelationship between basis set, total memory, field frequency, and error.  At some frequencies and choices of total memory, $Q_{\text{model}}(t)$ as computed by (\ref{eqn:Qpropsymm}) agrees more closely with $Q_{\text{true}}(t)$ as the basis set size increases.  For other choices of frequency and total memory, the opposite occurs.  The differences may partially be explained by symmetries of the molecules we have chosen, which affect the Hamiltonian $H(t)$ and by extension our propagation scheme.  For instance, due to the symmetry of the $\htwo$ molecule, there is no permanent dipole moment in either the ground or excited state, and we therefore find that its dipole moment matrix $M_{\text{dip}}$ has a zero diagonal. This has non-trivial implications for the numerics in (\ref{eqn:Qpropsymm}) that rely on the pseudoinverse.  All of this motivates future study, especially as we aim to quantify the memory-dependence of molecules on grids in real space (as in TDDFT), rather than in particular basis sets.





\section{Conclusion}
\label{sect:conclusion}
In this paper, we presented and derived a linear, closed time-delay system to propagate $1$-electron reduced density matrices (1RDMs).  To construct this system, we needed to understand the details of how to compute 1RDMs from full density matrices.  Building on concepts and methods from quantum chemistry, we developed a mathematical formalism that enables full evaluation of the $B$ tensor that relates full configuration interaction densities to 1RDMs---see Proposition \ref{Bten}.  
These results are general and do not depend on the basis set, number of electrons, or subset of Slater determinants that comprise the CI basis functions.

With $B$ in hand, we derived (\ref{eqn:Qprop}), our first propagation scheme for 1RDMs.  The main insight in this derivation was to add enough memory-based constraints---relating the full density at the present time to 1RDMs at past times---to change a linear system from underdetermined to overdetermined.  Next, we accounted for properties of full density matrices: their Hermitian symmetry, constant trace, and zero elements that can be eliminated \emph{a priori}.  This resulted in the improved propagation scheme (\ref{eqn:Qpropsymm}), which we implemented numerically and studied.

In our numerical tests, we focused on two $2$-electron molecules ($\htwo$ and $\heh$), each in two basis sets (STO-3G and 6-31G).  For the resulting four molecular systems, we showed that the propagation scheme (\ref{eqn:Qpropsymm}) produced highly accurate 1RDMs given a sufficient time delay.  For each molecular system, the accuracy of our propagation method improved as we increased total memory-dependence ($k \ell \Delta t$) measured not in time steps but in physical units of time.  For fixed $\Delta t$, when we held the number of delay steps $\ell$ fixed but increased the stride $k$, we saw the error approach machine precision.  We also saw that the error of our propagation scheme is approximately invariant under refinement of $\Delta t$.  This was particularly important given that the true dynamics take place in continuous time.

The present study motivates three areas for future work.  First, having developed mathematical methods to propagate 1RDMs that arise from CI wave functions, we are now well-positioned to address scientific questions.  We aim to establish basis-independent physical time scales of memory-dependence.  This will require disentangling the interdependence between memory-dependence, applied field parameters (including frequency, intensity, number of cycles), and the molecule as represented by its core Hamiltonian $H_0$ and dipole moment matrix $M_{\text{dip}}$.  We also hope to understand the memory-dependence of 1RDMs for larger systems.  As $N$ increases and the number of basis functions increases, we expect that the required memory will also increase due to the larger number of possible electronic states involved in the dynamics. For such systems, the method used here will not be feasible due to the need to construct the pseudoinverse.  Also, the TDCASSCF method with an all electron, full orbital active space will become intractable. In this case, other, more approximate electronic structure methods such as TDCIS (time-dependent configuration interaction singles) can be used to compute ground truth 1RDMs \cite{hirata1999configuration,foresman1992toward}.  Compared to TDCASSCF, TDCIS has fewer excited states for electrons to enter.  It will be of interest to study if and how approximations such as TDCIS impact the memory-dependence of 1RDMs.
 
Second, we can improve the propagation scheme (\ref{eqn:Qpropsymm}).  At present, the linear system (\ref{eqn:bigsys4}) involves roughly half of the entries of $P(t)$---the upper-triangular real and imaginary parts, including all but one element on the diagonal.  Suppose we eliminate more columns from $M''(t)$ and thereby learn fewer entries of the rank-$1$ matrix $P(t)= \mathbf{a}(t) \mathbf{a}(t)^\dagger$.  We can then formulate a quadratic programming problem to learn $\mathbf{a}(t)$ from our incompletely determined $P(t)$.  Solving this at each time step, we could propagate $\mathbf{a}(t)$ forward in time, giving another route to propagate 1RDMs.  The resulting method would be nonlinear, unlike the linear method we have derived in the present work.  If enough columns could be eliminated from $M''(t)$, condition numbers would decrease, yielding greater accuracy.

Third, our work proves that there exists a time-dependent affine map that propagates 1RDMs forward in time; concretely, this affine map is the right-hand side of (\ref{eqn:Qpropsymm}) not including $\mathbf{b}_{\ell}(t)$.  Imagine if we could learn this time-dependent map without having to compute the Hamiltonian $H(t)$.  Such a method would enable propagation of 1RDMs without the computational expense of running a static CI calculation to determine the core Hamiltonian $H_0$.  This would achieve large-scale computational cost savings.  It is an open question whether, for instance, machine learning methods could be used to learn the time-dependent affine map from data.  

In Section \ref{sec:intro}, we described recent efforts to machine learn exchange-correlation potentials $V_{\text{xc}}$ with memory-dependence on $1$-electron reduced densities.  In future work, we plan to use the methods developed in this paper (or improved versions thereof) to study memory-dependence of 1RDMs for a wide array of molecular systems in a broad set of field conditions, with the goal of achieving a comprehensive understanding of memory-dependence.   This will enable a more robust and chemically insightful search of neural network architectures to accurately model and eventually learn $V_{\text{xc}}$ in higher dimensions using, e.g., a three-dimensional version of the adjoint method derived previously \cite{bhat2021dynamic}.  Learning memory-dependent $V_{\text{xc}}$ functionals would would bring TDDFT beyond the realm of the adiabatic approximation, fixing the egregious errors of commonly employed memory-independent potentials and enabling more accurate time-dependent electron dynamics simulations of molecules and materials. 


\begin{acknowledgments}
Research was sponsored by the Office of Naval Research and was accomplished under
Grant Number W911NF-23-1-0153. The views and conclusions contained in this document are those of the authors and should not be interpreted as representing the official policies, either expressed or implied, of the Army Research Office or the U.S. Government. The U.S. Government is authorized to reproduce and distribute reprints for
Government purposes notwithstanding any copyright notation herein. We acknowledge support from the U.S. Department of Energy, Office of Science, Basic Energy Sciences under Award Number DE-SC0020203. H. Bassi acknowledges partial support from NSF DMS-1840265. This research used resources of the National Energy Research Scientific Computing Center (NERSC), a U.S. Department of Energy Office of Science User Facility located at Lawrence Berkeley National Laboratory, operated under Contract No. DE-AC02-05CH11231 using NERSC awards BES-m2530 and ASCR-m4577. We also gratefully acknowledge computational time on the Pinnacles cluster at UC Merced, supported by NSF OAC-2019144.
\end{acknowledgments}

\section*{Data Availability Statement}
The data that support the findings of this study are openly
available in \url{https://github.com/hbassi/1rdm_memory_model}.

\appendix

\section{Proofs of Methodological Results}
\subsection{Proof of Proposition \ref{Bten}}
\label{sect:proofBten}
\begin{proof}
Returning to (\ref{eqn:slater}-\ref{eqn:CIbasis}), let us define $\gamma^{n,q} = (\gamma^{n,q}_1, \ldots, \gamma^{n,q}_N)$ to be the $n$-th permutation of the combination $\mathbf{i}(q)$.  There are $N!$ such permutations.  Let the parity of the $n$-th permutation be $p_n$.  We now slightly abuse notation by referring to  $(\br_1,\sigma_1)$ as $\mathbf{x}_1$ and $(\br'_1,\sigma_1)$ as $\mathbf{x}'_1$; building on this, we use $\mathbf{X}$ to mean 
$(\mathbf{x}_1,\mathbf{x}_2,\ldots,\mathbf{x}_N)$, while $\mathbf{X}'$ means $(\mathbf{x}_1',\mathbf{x}_2,\ldots,\mathbf{x}_N)$.  This matches what we need to evaluate (\ref{eqn:sub1}).
With this notation and (\ref{eqn:spinorbital}), we can expand Slater determinants via (\ref{eqn:slater}) and write
\begin{equation}
\label{eqn:slatersmash}
\psi^\text{SL}_{\mathbf{i}(q)}(\mathbf{X}) \psi^\text{SL}_{\mathbf{i}(q')}(\mathbf{X'})^\ast = \frac{1}{N!} \sum_{n=1}^{N!} \sum_{m=1}^{N!} (-1)^{p_n+p_m} \chi_{\gamma_1^{n,q}}(\mathbf{x}_1) \chi_{\gamma_1^{m,q'}}(\mathbf{x}'_1)^\ast \prod_{j=2}^N \chi_{\gamma_j^{n,q}}(\mathbf{x}_j) \chi_{\gamma_j^{m,q'}}(\mathbf{x}_j)^\ast
\end{equation}
Integrating over $\mathbf{x}_2, \ldots, \mathbf{x}_N$ and using the orthonormality of the spin-orbitals, we find that the integral will vanish unless $\gamma_j^{n,q} = \gamma_j^{m,q'}$ for all $j \geq 2$.
Using the Kronecker $\delta$, we have
\[
\int \psi^\text{SL}_{\mathbf{i}(q)}(\mathbf{X}) \psi^\text{SL}_{\mathbf{i}(q')}(\mathbf{X}')^\ast \, d \mathbf{x}_2 \cdots d \mathbf{x}_N = \frac{1}{N!} \sum_{n=1}^{N!} \sum_{m=1}^{N!} (-1)^{p_n+p_m}  \chi_{\gamma_1^{n,q}}(\mathbf{x}_1) \chi_{\gamma_1^{m,q'}}(\mathbf{x}'_1)^\ast \prod_{j = 2}^{N} \delta_{\gamma_j^{n,q}, \gamma_j^{m,q'}}
\]
Let $\epsilon_{a,b} = 1$ if $a$ and $b$ are either both even or both odd; let $\epsilon_{a,b} = 0$ otherwise.  Then integrating over $\sigma_1$ and using (\ref{eqn:spinorbital}), we obtain
\begin{equation}
\label{eqn:radred}
\bigl( \psi^\text{SL}_{\mathbf{i}(q)} (\psi^\text{SL}_{\mathbf{i}(q')})^\ast \bigr)_1 = \frac{1}{N!} \sum_{n=1}^{N!} \sum_{m=1}^{N!} (-1)^{p_n+p_m}  \phi_{\lceil \gamma_1^{n,q}/2 \rceil} \phi_{\lceil \gamma_1^{m,q'}/2 \rceil}^\ast \epsilon_{\gamma_1^{n,q},\gamma_1^{m,q'}} \prod_{j = 2}^{N} \delta_{\gamma_j^{n,q}, \gamma_j^{m,q'}}
\end{equation}
Essentially, the combinations $\mathbf{i}(q)$ and $\mathbf{i}(q')$ can differ in at most one slot, in which case, in the slot that differs, the numbers must either both be even or both be odd.  If these conditions are not satisfied, the result will be zero.  We now consider the two cases in turn.

\textbf{Case I:} $q=q'$, so   $\mathbf{i}(q) = \mathbf{i}(q')$.  If two permutations of the same set of $N$ integers agree in $N-1$ slots, they must be identical.  For each $n$, there is only one value of $m$ such that $\gamma^{m,q} = \gamma^{n,q}$.  So,
\begin{equation}
\label{eqn:radred2}
\bigl( \psi^\text{SL}_{\mathbf{i}(q)} (\psi^\text{SL}_{\mathbf{i}(q)})^\ast \bigr)_1 = \frac{1}{N!} \sum_{n=1}^{N!} \phi_{\lceil \gamma_1^{n,q}/2 \rceil} \phi_{\lceil \gamma_1^{n,q}/2 \rceil}^\ast 
\end{equation}
Holding $\gamma_1^{n,q}$ fixed at one of its $N$ possible values, there are $(N-1)!$ possible permutations of $\gamma_j^{n,q}$ for $j \geq 2$.  Accounting for this, (\ref{eqn:radred2}) becomes (\ref{eqn:radred3})

In (\ref{eqn:radred3}), it is possible that as $k$ goes from $1$ to $N$, the value $\lceil i_k(q)/2 \rceil$ repeats up to twice.  This is because in a single Slater determinant, it is possible for up to two spin orbitals to share the same spatial orbital.

\textbf{Case II:} $q \neq q'$, so  $\mathbf{i}(q) \neq \mathbf{i}(q')$.  Further assume that $\mathbf{i}(q)$ and $\mathbf{i}(q')$ are not reorderings of one another; this would yield Slater determinants that differ at most by a sign, and hence would not be a good choice to include in the CI expansion (\ref{eqn:CIbasis}).  For a nonzero integral, there must exist unique integers $a \in \mathbf{i}(q)$, $a' \in \mathbf{i}(q')$ such that $a \neq a'$ with $\epsilon_{a,a'} = 1$.  

As above, there are $(N-1)!$ permutations $\gamma^{n,q}$ such that $\gamma^{n,q}_1 = a$.  For each such permutation, there is only one value of $m$ such that $\gamma^{m,q'}_j = \gamma^{n,q}_j$ for all $j \geq 2$.  Let $Z$ be the least number of flips required to permute $\mathbf{i}(q')$ so that it matches $\mathbf{i}(q)$ in all but the first slot; then $(-1)^{p_n + p_m} = (-1)^Z$ for all permutations that yield nonzero integrals.  Using these facts in (\ref{eqn:radred}), we obtain (\ref{eqn:radred4}).

In (\ref{eqn:radred4}), note that $\lceil a/2 \rceil$ must differ from $\lceil a'/2 \rceil$.  To see why, consider that if $\lceil a/2 \rceil = \lceil a'/2 \rceil$, then either $a = a' + 1$ or $a' = a + 1$.  Both of these cases would have resulted in different spin functions of the $\sigma_1$ spin variable; by orthogonality, this yields $0$.
\end{proof}

\subsection{Proof of Proposition \ref{Qprop}}
\label{sect:Qpropproof}
\begin{proof}
Using $\widetilde{B}^T$, we can reformulate (\ref{eqn:rdmdef}) as
$\vvec(Q(t))^T = \vvec(P(t))^T \widetilde{B}^T$.
Transposing yields
\begin{equation}
\label{eqn:matvec}
\widetilde{B} \vvec(P(t)) = \vvec(Q(t)).
\end{equation}
We can identify this with (\ref{eqn:reduced}): $\vvec(Q(t))$ is $\mathbf{y}(t)$, $\vvec(P(t))$ is $\mathbf{z}(t)$, and $\widetilde{B}$ is $R$.  Suppose we are interested in solving for the full TDCI density $\vvec(P(t))$ given the right-hand side, the reduced density $\vvec(Q(t))$. The problem is that $\widetilde{B}$ is of dimension $K^2 \times N_C^2$ where $K < N_C$; the resulting linear system is underdetermined.  As in Section \ref{sect:yprop}, we augment (\ref{eqn:matvec}) by relating $\vvec(P(t))$ to enough past 1RDMs $\vvec(Q(t-j \Delta t))$ to create an overdetermined system.  

By (\ref{eqn:discsuper}) and (\ref{eqn:kronsum}), we have
\begin{equation}
\label{eqn:vvecPprop}
\vvec(P(t)) = \left( \exp(i H(t-\Delta t)^T \Delta t) \otimes \exp(-i H(t-\Delta t) \Delta t) \right) \vvec(P(t-\Delta t)).
\end{equation}
Inverting, iterating, and using properties of the Kronecker product, we obtain
\begin{equation}
\label{eqn:backwardPt}
(C_m(t)^T \otimes A_m(t)) \vvec(P(t)) = \vvec(P(t-m \Delta t)).
\end{equation}
As an aside, because $H$ is Hermitian, $A_m(t)$ is the conjugate transpose of $C_m(t)$---see (\ref{eqn:aj}) and (\ref{eqn:cj})---implying that we need only compute one of the two at each $t$.   Multiplying (\ref{eqn:backwardPt}) through by $\widetilde{B}$, we obtain
\begin{equation}
\label{eqn:backwardQt}
\mathcal{D}_j \vvec(P(t)) = \vvec(Q(t - j k \Delta t)).
\end{equation}
Note that (\ref{eqn:backwardPt}) and (\ref{eqn:backwardQt}) are particular cases of (\ref{eqn:backprop}) and (\ref{eqn:usefulfact}).  The stride $k$ allows one to use every $k$-th previous 1RDM, where $k$ need not equal $1$. Using what we have defined, we form the linear system\begin{equation}
\label{eqn:bigsys}
M(t) \vvec(P(t)) = \mathbf{q}_{\ell}(t).
\end{equation}
This relates the full TDCI density matrix $P(t)$ to  \emph{present and past} $1$-electron reduced density matrices. The block matrix $M(t)$ on the left-hand side---a special case of $M(t)$ defined by (\ref{eqn:Mdef2})---will be of dimension $(\ell+1) K^2 \times N_C^2$. By hypothesis, $M(t)$ has full column rank.  Hence
\begin{equation}
\label{eqn:approxinversion}
\vvec(P(t)) = M(t)^{+} \mathbf{q}_{\ell}(t).
\end{equation}
To obtain (\ref{eqn:Qprop}), we return to (\ref{eqn:discsuper}), multiply both sides by $\widetilde{B}$, and then use (\ref{eqn:approxinversion}).
\end{proof}

\subsection{Proofs of Results from Section \ref{sect:symm}}
\label{sect:SymmConstrProofs}
\noindent Proof of Lemma \ref{HermLem}:
\begin{proof}
All $S^j$ matrices are orthogonal in the Frobenius inner product, implying linear independence.  Given any $Z \in \mathbb{H}$,  there exists $\br \in \mathbb{R}^{N_C^2}$ such that $Z = \sum_{j} r_j S^j$.  To find this vector, first enumerate the diagonal entries of $Z$, then the real parts of the strictly upper-triangular entries, and finally the imaginary parts of the strictly upper-triangular entries.
\end{proof}
\noindent Proof of Proposition \ref{Hermitian}:
\begin{proof}
As each $S^j$ is an $N_C \times N_C$ matrix, the collection $S$ is a tensor of shape $N_C^2 \times N_C \times N_C$.  Reshape $S$ into a matrix $\widetilde{S}$ such that for all $Z \in \mathbb{H}$ with $Z = \sum_{j} r_j S^j$,
$\vvec(Z) = \widetilde{S} \br$.
Explicitly, $S^j_{k,l} = \widetilde{S}_{(k-1)N_C + l, j}$.  Using the notation $\widetilde{S}_{:,j}$ to denote the $j$-th column of $\widetilde{S}$, we can equivalently say that $\widetilde{S}_{:,j} = \vvec(S^j)$.  We see that $\widetilde{S}$ gives us a real representation of $\mathbb{H}$.  Using this representation in  (\ref{eqn:bigsys}), we obtain (\ref{eqn:bigsys2}).  Once we solve for $\mathbf{x}(t)$, we can recover the full density matrix via $\vvec(P(t)) = \widetilde{S} \mathbf{x}(t)$.   This $P(t)$ will be perfectly Hermitian.
\end{proof}
\noindent Proof of Proposition \ref{traceconstr}:
\begin{proof}
By virtue of how we have defined the $\{S^j\}$ basis, the first $N_C$ entries of $\mathbf{x}(t)$ correspond to the real, diagonal entries of $P(t)$.  Thus $\trace(P(t)) = 1$ implies
\begin{equation}
\label{eqn:traceconstr}
x_{N_C}(t) = 1 - \sum_{j=1}^{N_C-1} x_{j}(t).
\end{equation}
Using (\ref{eqn:traceconstr}) in (\ref{eqn:bigsys2}), we obtain
\begin{equation}
\label{eqn:columnwise}
\sum_{j=1}^{N_C-1} \left( [M(t) \widetilde{S}]_{:, j} - [M(t) \widetilde{S}]_{:, N_C} \right) x_j(t) + \sum_{j=N_C+1}^{N_C^2} [M(t) \widetilde{S}]_{:, j} x_j(t) = \mathbf{q}_{\ell}(t) - [M(t) \widetilde{S}]_{:, N_C}.
\end{equation}
To form $M'(t)$, we start with $M(t)$ and subtract its $N_C$-th column $[M(t) \widetilde{S}]_{:, N_C}$ from each of its first $N_C-1$ columns.  From the result, delete the $N_C$-th column and keep the remaining columns unchanged.  With this definition of $M'(t)$, (\ref{eqn:columnwise}) is (\ref{eqn:bigsys3}).  After solving for $\mathbf{x}_{-N_C}(t)$, we can reconstruct $\mathbf{x}(t)$ using (\ref{eqn:traceconstr}); observe that $\mathbf{x}(t)$ is an affine function of $\mathbf{x}_{-N_C}(t)$.
\end{proof}
\noindent Proof of Proposition \ref{identzero}:
\begin{proof}
Start with $M'(t)$ as constructed in Proposition \ref{traceconstr}.  For each of the $D$ entries of $P(t)$ that are known to be zero, eliminate the corresponding column of $M'(t)$ and the corresponding entry of $\mathbf{x}_{-N}(t)$, resulting in matrices $M''(t)$ and vectors $\mathbf{x}_{-}(t)$.  Here the correspondence will be with respect to the ordering of our basis elements $\{S^j\}$.
\end{proof}
\noindent Proof of Proposition \ref{Qpropsymm}:
\begin{proof}
The proof is identical to that of Proposition \ref{Qprop} except that we replace (\ref{eqn:approxinversion}) with the expression labeled $\vvec(P(t))$ in (\ref{eqn:Qpropsymm}).  To justify this expression, first note that recovering $\mathbf{x}_{-N}(t)$ from $\mathbf{x}_{-}(t)$ requires a simple linear transformation.  As described at the end of the proof of Proposition \ref{traceconstr}, computing $\mathbf{x}(t)$ from $\mathbf{x}_{-}(t)$ amounts to applying an affine transformation.  Composing these maps, we obtain the affine map $\mathcal{A}$ that satisfies $\mathbf{x}(t) = \mathcal{A} \mathbf{x}_{-}(t)$.  As $M''(t)$ has full column rank, we can compute $\mathbf{x}_{-}(t) = M''(t)^{+} \mathbf{b}_{\ell}(t)$.  Our real representation of $\mathbb{H}$ yields $\vvec(P(t)) = \widetilde{S} \mathbf{x}(t)$.   In total, we obtain $\vvec(P(t)) = \widetilde{S} \mathcal{A}  M''(t)^{+} \mathbf{b}_{\ell}(t)$.
\end{proof}

\subsection{Proof of Proposition \ref{Schur}}
\label{sect:Schurproof}
\begin{proof} To study the linear independence of the rows of $M(t)$, we consider 
\begin{equation}
\label{eqn:original}
x \widetilde{B} +  y \mathcal{D}_1 = 0.
\end{equation}
for $x,y \in \mathbb{C}^{K^2}$.  Overall, (\ref{eqn:original}) is an equality of row vectors in $\mathbb{C}^{N_C^2}$.  Let us restrict attention to only those columns with indices in $G$: $x \widetilde{B}^G + y \mathcal{D}_1^G = 0$.
By hypothesis, the $K^2 \times K^2$ matrix $\widetilde{B}^G$ has full rank and is thus invertible.  Using this, we solve for $x = -y \mathcal{D}_1^G (\widetilde{B}^G)^{-1}$. Putting this back into (\ref{eqn:original}), we obtain
$y \left( \mathcal{D}_1 - \mathcal{D}_1^G (\widetilde{B}^G)^{-1} \widetilde{B} \right) = 0$. As it is clear that subsetting on the columns in $G$ will yield the zero matrix, we instead subset on the columns \emph{not in} $G$:
\[
y \left( \mathcal{D}_1^{-G} - \mathcal{D}_1^G (\widetilde{B}^G)^{-1} \widetilde{B}^{-G} \right) = 0.
\]
The $K^2 \times (N_C^2 - K^2)$ matrix multiplying $y$ is  the Schur complement of $\widetilde{B}^G$ of $\widehat{M}$.  As this Schur complement has full rank, we conclude that $y=0$, forcing $x=0$.
\end{proof}

\section{Constant Trace for 1RDMs}
\label{appendix:trace}
Suppose one computes 1RDMs $Q(t)$ directly from full density matrices $P(t)$ using the $B$ tensor as in (\ref{eqn:rdmdef}).  The following theorem shows that the resulting 1RDMs must have constant trace.

\begin{theorem}[Constant Trace]
\label{constanttrace}
Regardless of which Slater determinants are used to form the CI basis functions, the 1RDM $Q(t)$ has constant trace equal to $N$, the number of electrons in the system.
\end{theorem}
\begin{proof}
    From (\ref{eqn:rdmdef}), we obtain
\begin{equation}
\label{eqn:traceQdef}
\trace(Q(t)) = \sum_{k,\ell=1}^{N_C} P_{k,\ell}(t) \sum_{b=1}^{K} B_{k,\ell,b,b}.
\end{equation}
From (\ref{eqn:Btendef}) and (\ref{eqn:radred10}),
\begin{equation*}
\sum_{b=1}^{K} B_{k,\ell,b,b} = N \sum_{q,q'} \sum_{b=1}^{K} \int \overline{\phi_b(\br)}  C_{k q} \overline{C_{\ell q'}} \left( \psi^{\text{SL}}_{\mathbf{i}(q)} (\psi^{\text{SL}}_{\mathbf{i}(q')})^\ast \right)_1 (\br) \phi_b(\br) \, d \br.
\end{equation*}
Examining (\ref{eqn:radred4}), we see that this integral will vanish in Case II, whenever $q \neq q'$.  Thus we need only consider $q = q'$ (Case I), upon which (\ref{eqn:radred3}) yields
\begin{align}
\sum_{b=1}^{K} B_{k,\ell,b,b} &= \sum_{q=1}^{N^C} \sum_{b=1}^{K} \int \overline{\phi_b(\br)}  C_{k q} \overline{C_{\ell q}}  \sum_{m=1}^{N} \phi_{\lceil i_m(q)/2 \rceil} (\br)  \phi_{\lceil i_m(q)/2 \rceil}^\ast  (\br) \phi_b(\br) \, d \br \nonumber \\
 \label{eqn:Btrace}
 &= \sum_{q=1}^{N^C} \sum_{b=1}^{K} \sum_{m=1}^{N} \delta_{ \lceil i_m(q)/2 \rceil, b} C_{k q} \overline{C_{\ell q}} = \sum_{q=1}^{N^C} N C_{k q} \overline{C_{\ell q}}
\end{align}
To obtain the final equality, note that for each fixed $m$, there is precisely one $b$ such that $\lceil i_m(q)/2 \rceil = b$.  Summing over $m$, we obtain $N$.

Next, we claim that Slater determinants themselves are orthonormal: to see this, return to (\ref{eqn:slatersmash}) and integrate both sides with respect to $\mathbf{X} = (\mathbf{x}_1, \ldots, \mathbf{x}_N)$.  We will obtain zero unless $\gamma^{n,q} \equiv \gamma^{m,q'}$, which is impossible unless $q=q'$.  In that case, the integral will be $1$, i.e., $\langle \psi^\text{SL}_{\mathbf{i}(q)}, \psi^\text{SL}_{\mathbf{i}(q')}\rangle = \delta_{q,q'}$.  Using this, orthonormality of the CI basis functions (\ref{eqn:CIrep}) yields
\[
\delta_{a b} = \langle \Psi^\mathrm{CI}_a, \Psi^\mathrm{CI}_b \rangle = \sum_{q,q'=1}^{N_C} \overline{C_{a q}} C_{b q'} \langle \psi^\text{SL}_{\mathbf{i}(q)}, \psi^\text{SL}_{\mathbf{i}(q')}\rangle = \sum_{q=1}^{N_C} \overline{C_{a q}} C_{b q},
\]
i.e., $C$ is unitary.  Using this, (\ref{eqn:Btrace}),  (\ref{eqn:traceQdef}), and $\trace(P(t)) \equiv 1$, we find
\begin{equation}
\label{eqn:traceQ}
\trace(Q(t)) = \sum_{k,\ell=1}^{N_C} P_{k,\ell}(t) N \delta_{k,\ell} = N.
\end{equation}
\end{proof} 

\section{Propagating the 1RDM when there are only $1$-electron terms in the Hamiltonian}
\label{sect:oneelectron}
If we retain only $1$-electron terms in our Hamiltonian, then we neglect (entirely) all electron-electron interactions, including Coulomb, exchange, and correlation terms.   In this case, it should be possible to propagate the 1RDM exactly with no memory.  While this makes sense physically, we would like to establish that for such a Hamiltonian, the 1RDM propagation scheme from the manuscript works correctly with no memory.  Specifically, consider the $\ell=0$ (memoryless) case of (\ref{eqn:Qprop}) from the manuscript:
\begin{equation}
\label{eqn:memorylessprop}
\operatorname{vec}(Q(t+1)) = \widetilde{B} \bigl( \exp(i H(t)^T \Delta t) \otimes \exp(-i H(t) \Delta t) \bigr) \widetilde{B}^+ \operatorname{vec}(Q(t)).
\end{equation}
Here $H(t)$ stands for the \emph{full} TDCI Hamiltonian.  In this appendix, we analyze (\ref{eqn:memorylessprop}) fully in the case where $H(t)$ only contains $1$-electron terms, and show that it reduces to an exact propagation equation that involves only the \emph{1-electron Hamiltonian}.

\subsection{Setup}
We will first work in the specific setting where we have $N=2$ electrons, $K=2$ atomic orbitals, and $N_C = 4$ total CI determinants.  These parameters correspond to $\heh$ and $\htwo$ in STO-3G.  Later, in Subsection \ref{sect:appgen}, we will explain how the calculation generalizes to $\heh$ and $\htwo$ in 6-31G.

In what follows, full CI density matrices $P(t)$ are of size $4 \times 4$, while 1RDMs $Q(t)$ are of size $2 \times 2$.  We will use capital $H$ to denote $4 \times 4$ Hamiltonian matrices, and lowercase $h$ to denote $2 \times 2$ Hamiltonian matrices.

Before analyzing our scheme (\ref{eqn:Qprop}), we set up the problem.  We carry this out in detail to illustrate our methods concretely. For a particular index map $\mathbf{i}(q)$, we have the CI wave function
\begin{equation*}
\Psi^{\text{CI}}_a  = \frac{1}{\sqrt{2}} \sum_{q=1}^4 C_{a q} \left[ \chi_{i_1(q)}(\mathbf{r}_1,\sigma_1)\chi_{i_2(q)}(\mathbf{r}_2,\sigma_2) - \chi_{i_1(q)}(\mathbf{r}_2,\sigma_2)\chi_{i_2(q)}(\mathbf{r}_1,\sigma_1) \right]
\end{equation*}
Recall that for $K$ atomic orbitals, we have $K$ molecular orbitals $\{ \phi_j \}_{j=1}^K$, from which we form $2K$ spin-orbitals (\ref{eqn:spinorbital}).  The index map $\mathbf{i} : \{ 1, 2, 3, 4 \} \to \{1, 2, 3, 4\} \times \{1, 2, 3, 4\}$ tells us the precise identity of each Slater determinant.  For our minimal basis $2$-electron systems, we have spin-orbitals
\begin{equation*}
\chi_{1}(\mathbf{x}) = \phi_1(\mathbf{r}) \alpha(\sigma), \ 
\chi_{2}(\mathbf{x}) = \phi_1(\mathbf{r}) \beta(\sigma), \
\chi_{3}(\mathbf{x}) = \phi_2(\mathbf{r}) \alpha(\sigma), \ \text{ and }
\chi_{4}(\mathbf{x}) = \phi_2(\mathbf{r}) \beta(\sigma);
\end{equation*}
and Slater determinants given by the index mapping
\[
\mathbf{i}(1) = (1,2), \quad \mathbf{i}(3) = (3,2), \quad \mathbf{i}(2) = (1,4), \quad  \mathbf{i}(4) = (3,4).
\]
For future reference, we reconcile this with more standard notation:
\begin{itemize}
\item The $\mathbf{i}(1) = (1,2)$ determinant is the Hartree-Fock determinant $| \psi^{\text{RHF}} \rangle$, where both the $\alpha$ (spin up) and $\beta$ (spin down) electrons are in MO (molecular orbital) 1.
\item The $\mathbf{i}(2) = (3,2)$ determinant corresponds to promoting the $\alpha$ electron from MO 1 to MO 2.  This is the $| \psi_{1 \alpha}^{2 \alpha} \rangle$ determinant.
\item The $\mathbf{i}(3) = (1,4)$ determinant corresponds to promoting the $\beta$ electron from MO 1 to MO 2.  This is the $| \psi_{1 \beta}^{2 \beta} \rangle$ determinant.
\item The $\mathbf{i}(4) = (3,4)$ determinant corresponds to promoting both electrons from MO 1 to MO 2.  This is the $| \psi_{11}^{22} \rangle$ determinant.
\end{itemize}
Suppose we write the one-electron parts of the Hamiltonian in the form $\widehat{H} = \widehat{h}^{1} + \widehat{h}^{2}$, where $\widehat{h}^{j}$ encompasses the kinetic and electron-nuclear operators that affect electron $j$.  Borrowing the notation from (\ref{eqn:molham}), we have
\[
\widehat{h}^{j} = -\frac{1}{2} \nabla_j^2 - \sum_{A=1}^{N'} \frac{Z_A}{ \| \mathbf{r}_j - \mathbf{R}_A \| }.
\]
We introduce the notation
\[
h_{k, \ell} = \int \overline{\chi_{\ell}(\mathbf{r}_j,\sigma_j)} \widehat{h}^{j} \chi_{k}(\mathbf{r}_j,\sigma_j) \, d \mathbf{r}_j.
\]
These integrals do not actually depend on $j$, which is why we purposely do not include a superscript $j$ on the left-hand side.  If we apply $\widehat{h}^{1}$ to a CI wave function, we obtain
\[
\widehat{h}^{1} \Psi^{\text{CI}}_{a} = \frac{1}{\sqrt{2}} \sum_{q=1}^4 C_{a q} \left[ \chi_{i_2(q)}(\mathbf{r}_2,\sigma_2) \widehat{h}^{1} \chi_{i_1(q)}(\mathbf{r}_1,\sigma_1) - \chi_{i_1(q)}(\mathbf{r}_2,\sigma_2) \widehat{h}^{1} \chi_{i_2(q)}(\mathbf{r}_1,\sigma_1) \right].
\]
Hence
\begin{multline*}
\langle \Psi^{\text{CI}}_b, \widehat{h}^{1} \Psi^{\text{CI}}_{a} \rangle = \frac{1}{2} \sum_{q,q'} C_{a q} \overline{C_{b q'}}
\biggl[ \delta_{i_2(q),i_2(q')} \int \overline{\chi_{i_1(q')}(\mathbf{r}_1,\sigma_1)} \widehat{h}^1 \chi_{i_1(q)}(\mathbf{r}_1,\sigma_1) \, d \mathbf{r}_1 \\
- \delta_{i_1(q'),i_2(q)} \int \overline{\chi_{i_2(q')}(\mathbf{r}_1,\sigma_1)} \widehat{h}^1 \chi_{i_1(q)}(\mathbf{r}_1,\sigma_1) \, d \mathbf{r}_1 - \delta_{i_1(q),i_2(q')} \int \overline{\chi_{i_1(q')}(\mathbf{r}_1,\sigma_1)} \widehat{h}^1 \chi_{i_2(q)}(\mathbf{r}_1,\sigma_1) \, d \mathbf{r}_1 \\
+ \delta_{i_1(q),i_1(q')} \int \overline{\chi_{i_2(q')}(\mathbf{r}_1,\sigma_1)} \widehat{h}^1 \chi_{i_2(q)}(\mathbf{r}_1,\sigma_1) \, d \mathbf{r}_1 \biggr].
\end{multline*}
We treat the quantity in square brackets as the $(q,q')$-th entry of a $4 \times 4$ matrix $H_1$.  For our particular index mapping $\mathbf{i}(q)$, this matrix is
\[
H_1 = \begin{bmatrix} 
2 h_{1,1} & h_{1,2} & h_{1,2} & 0 \\
 h_{2,1} & h_{1,1}+h_{2,2} & 0 & h_{1,2} \\
 h_{2,1} & 0 & h_{1,1}+h_{2,2} & h_{1,2} \\
 0 & h_{2,1} & h_{2,1} & 2 h_{2,2} \end{bmatrix}.
 \]
This is consistent with the Slater-Condon rules.  Using the Kronecker product $\otimes$, we can write
\[
H_1 = h \otimes I + I \otimes h,
\]
where $I$ is the $2 \times 2$ identity matrix.  Let  $C^\dagger$ be the conjugate transpose of $C$.  Then $\overline{C}_{b q'} = (C^\dagger)_{q' b}$. With this and the definition of $h_1$, we have
\[
\langle \Psi^{\text{CI}}_b, \widehat{H}^{1} \Psi^{\text{CI}}_{a} \rangle = \frac{1}{2} \left( C H_1 C^\dagger \right)_{ab}.
\]
It is straightforward to show that this is also the matrix of the $\widehat{h}^2$ operator in the CI basis.  Hence if $\widehat{H} = \widehat{h}^{1} + \widehat{h}^{2}$,
\[
\langle \Psi^{\text{CI}}_b, \widehat{H} \Psi^{\text{CI}}_{a} \rangle = \left( C H_1 C^\dagger \right)_{ab} = \left( C (h \otimes I + I \otimes h) C^\dagger \right)_{ab}.
\]
As $h \otimes I + I \otimes h$ is Hermitian, there exists unitary $V$ and real diagonal $\Lambda$ such that $h \otimes I + I \otimes h = V \Lambda V^\dagger$.  We set $C = V^\dagger$, so that $C$ is unitary and
\[
C (h \otimes I + I \otimes h) C^\dagger = \Lambda.
\]
As $C$ is unitary, and as Slater determinants are themselves orthonormal, definition (\ref{eqn:CIbasis}) implies that the $\Psi^\mathrm{CI}_a$ functions are also orthonormal, as one can easily check.\\

\noindent Let $\widehat{H}(t)$ be the time-dependent Hamiltonian
$\widehat{H}(t) = \widehat{H} + f(t) \widehat{\mu}_z$,
where $\widehat{\mu}_z$ is the dipole moment operator in the $z$ direction.  This is sufficient because, in our work, we only apply a field in the $z$ direction.  Note that $\widehat{\mu}_z$ is itself a sum of $1$-electron operators; each is essentially the $z$ component of $\widehat{\mathbf{r}}_j$, the position operator for electron $j$.  Therefore, we can reapply the analysis above.  Let $\mu$ be the $2 \times 2$ dipole moment matrix (in the MO basis) that corresponds to a $1$-electron dipole moment operator in the $z$ direction.  Then the matrix of the $\widehat{\mu}_z$ operator in the CI basis is
\[
\langle \Psi^\text{CI}_b \widehat{\mu}_z \Psi^\text{CI}_a \rangle = (C (\mu \otimes I + I \otimes \mu) C^\dagger)_{a b}.
\]
Forming a time-dependent wave function as in (\ref{eqn:CIrep}) and substituting it into the TDSE, we obtain
\[
i \sum_{n} \frac{d}{dt} a_n(t) \Psi^{\text{CI}}_n(\mathbf{X}) = \sum_{n} a_n(t) (\widehat{H} + f(t) \widehat{\mu}_z) \Psi^{\text{CI}}_n(\mathbf{X}).
\]
We take the inner product of both sides with $\Psi^\text{CI}_m(\mathbf{X})$.  On the left-hand side, we use orthonormality.  On the right-hand side, we use the property that the operator $\widehat{H}$ is diagonalized by the $\Psi^\text{CI}_n$ functions, together with the matrix of the $\widehat{\mu}_z$ operator mentioned above.  
With the notation
\[
h(t) = h + f(t) \mu,
\]
where, as we defined above, $h$ and $\mu$ are $2 \times 2$ matrices in the MO basis, we obtain
\[
i \frac{d}{dt} a_n(t) = \sum_{n} a_n(t) (C (h(t) \otimes I + I \otimes h(t) ) C^\dagger)_{n m}.
\]
The full $4 \times 4$ TDCI Hamiltonian is, by the above results,
\begin{equation}
\label{eqn:fullham}
H(t) = C (h(t) \otimes I + I \otimes h(t)) C^\dagger.
\end{equation}
In practice, we find that $h$ is real symmetric.  This happens because the atomic orbitals, molecular orbitals, and operators in the Hamiltonian (kinetic, electron-nuclear, and dipole moment) are all real.  If $h$ is real symmetric, the matrix $C$ that diagonalizes $h \otimes I + I \otimes h$ can be chosen to be real orthogonal (rather than complex unitary as above).  In this case, the equation of motion for $a_n(t)$ simplifies to (\ref{eqn:matrixschro}) and the density matrix $P(t) = \mathbf{a}(t) \mathbf{a}(t)^\dagger$ satisfies the Liouville-von Neumann equation (\ref{eqn:lvn0}) with time-dependent Hamiltonian $H(t)$.

Let $\mathbf{g}_j(t)$ be the solution of the $2 \times 2$ matrix Schr\"odinger equation $i (d\mathbf{g}_j/dt) = h(t) \mathbf{g}_j(t)$ with initial condition $\mathbf{g}_j(0) = \mathbf{e}_j$, the $j$-th basis vector of $\mathbb{C}^2$.  This implies $\mathbf{g}_j^\dagger(t) \mathbf{g}_k(t) = \delta_{jk}$ for all $t$.  Then, for a general initial condition $\mathbf{a}(0)$, we find that
\[
\mathbf{a}(t) = \sum_{j,k=1}^{2} C (\mathbf{g}_j(t) \otimes \mathbf{g}_k(t)) C^T a_{2(j-1)+k}(0),
\]
 solves the $4 \times 4$ matrix Schr\"odinger equation (\ref{eqn:matrixschro}).
  Using $\mathbf{a}(t)$, we use Definition \ref{redquan} to form the  $2 \times 2$ 1RDM $Q(t)$.  Then, with some algebra, one can derive
\begin{equation}
\label{eqn:appendixQLVN}
i \frac{dQ}{dt} = [h(t), Q(t)].
\end{equation}
The upshot is that when we switch off the $2$-electron term in the Hamiltonian, we obtain a closed Liouville-von Neumann equation of motion for the 1RDM $Q(t)$.  Only the $2 \times 2$ matrix Hamiltonian $h(t)$ is required here.  In what follows, we return to the general case where $h$ is complex Hermitian and $C$ is complex unitary.


\subsection{Concrete evaluation of required matrices and tensors} Our goal is to prove that (\ref{eqn:memorylessprop}), the memoryless ($\ell=0$) case of the scheme (\ref{eqn:Qprop}), matches a time-discretization of (\ref{eqn:appendixQLVN}):
\begin{equation}
\label{eqn:Qallegedprop}
\operatorname{vec}(Q(t + \Delta t)) = (\exp(i h(t)^T \Delta t) \otimes \exp(-i h(t) \Delta t)) \operatorname{vec}(Q(t)).
\end{equation}
We start by evaluating the $B$ tensor directly via (\ref{eqn:Btendef}), which we rewrite as
\begin{gather*}
B_{k,\ell,b,c} = \sum_{q,q'} C_{k q} \mathscr{B}_{q,q',b,c} C^\dagger_{q' \ell} \\
\mathscr{B}_{q,q',b,c} = N \int \overline{\phi_b(\mathbf{r})} \phi_c(\mathbf{r}) \left( \psi^\text{SL}_{\mathbf{i}(q)} (\psi^\text{SL}_{\mathbf{i}(q')})^\ast \right)_1  \, d \mathbf{r}.
\end{gather*}
Using the particular index mapping $\mathbf{i}(q)$ defined above, we find that
\begin{equation}
\label{eqn:coreSLresult}
\left( \psi^\text{SL}_{\mathbf{i}(q)} (\psi^\text{SL}_{\mathbf{i}(q')})^\ast \right)_1 = 
\begin{bmatrix}
\phi_1 \phi_1^\ast & (1/2) \phi_1 \phi_2^\ast & (1/2) \phi_1 \phi_2^\ast & 0 \\
(1/2) \phi_2 \phi_1^\ast & (1/2) (\phi_1 \phi_1^\ast + \phi_2 \phi_2^\ast) & 0 & (1/2) (\phi_1 \phi_2^\ast) \\
(1/2) \phi_2 \phi_1^\ast & 0 & (1/2) (\phi_1 \phi_1^\ast + \phi_2 \phi_2^\ast) & (1/2) \phi_1 \phi_2^\ast \\
0 & (1/2) \phi_2 \phi_1^\ast & (1/2) \phi_2 \phi_1^\ast & \phi_2 \phi_2^\ast
\end{bmatrix}_{q,q'}
\end{equation}
The $4 \times 4 \times 2 \times 2$ tensor $\mathscr{B}$ can be evaluated numerically using (\ref{eqn:coreSLresult}), which upon reshaping into a $16 \times 4$ matrix $\widetilde{\mathscr{B}}^T$ yields both
\[
\widetilde{\mathscr{B}}^T = \begin{bmatrix}
 2 & 0 & 0 & 0 \\
 0 & 0 & 1 & 0 \\
 0 & 0 & 1 & 0 \\
 0 & 0 & 0 & 0 \\
 0 & 1 & 0 & 0 \\
 1 & 0 & 0 & 1 \\
 0 & 0 & 0 & 0 \\
 0 & 0 & 1 & 0 \\
 0 & 1 & 0 & 0 \\
 0 & 0 & 0 & 0 \\
 1 & 0 & 0 & 1 \\
 0 & 0 & 1 & 0 \\
 0 & 0 & 0 & 0 \\
 0 & 1 & 0 & 0 \\
 0 & 1 & 0 & 0 \\
 0 & 0 & 0 & 2 
 \end{bmatrix} \quad \text{and} \quad \widetilde{\mathscr{B}}^+ = \begin{bmatrix}
3/8 & 0 & 0 & -1/8 \\
 0 & 0 & 1/4 & 0 \\
 0 & 0 & 1/4 & 0 \\
 0 & 0 & 0 & 0 \\
 0 & 1/4 & 0 & 0 \\
1/8 & 0 & 0 & 1/8 \\
 0 & 0 & 0 & 0 \\
 0 & 0 & 1/4 & 0 \\
 0 & 1/4 & 0 & 0 \\
 0 & 0 & 0 & 0 \\
1/8 & 0 & 0 & 1/8 \\
 0 & 0 & 1/4 & 0 \\
 0 & 0 & 0 & 0 \\
 0 & 1/4 & 0 & 0 \\
 0 & 1/4 & 0 & 0 \\
 -1/8 & 0 & 0 & 3/8
 \end{bmatrix}.
\]
By direct computation, $\widetilde{\mathscr{B}} \widetilde{\mathscr{B}}^+ = I$, the $4 \times 4$ identity matrix, which is expected since the rows of $\widetilde{\mathscr{B}}$ (equivalently, the columns of $\widetilde{\mathscr{B}}^T$) are linearly independent.  Via direct calculation, one can show that for any $2 \times 2$ matrix $W$,
\begin{equation}
\label{eqn:bplus}
\widetilde{\mathscr{B}}^{+} \operatorname{vec}(4 W + 4 I) = \operatorname{vec}( W \otimes I + I \otimes W ).
\end{equation}
Using $\widetilde{\mathscr{B}} \widetilde{\mathscr{B}}^+ = I$, we see that
\begin{equation}
\label{eqn:baction}
\mathscr{B} \operatorname{vec}( W \otimes I + I \otimes W ) =  \operatorname{vec}(4 W + 4 I).
\end{equation}
Reshaping $\widetilde{\mathscr{B}}^+$ into a $4 \times 4 \times 2 \times 2$ tensor $\mathscr{B}^+$, we define the following tensor of the same shape:
\[
B^{+}_{k \ell b c} = \sum_{q,q'} C^{\dagger}_{q k} \mathscr{B}^+_{q,q',b,c} C_{\ell q'}.
\]
We reshape $B^{+}_{k \ell b c}$ into a $16 \times 4$ matrix $\widetilde{B}^+$.  We claim that $\widetilde{B}^+$ is the Moore-Penrose pseudoinverse of $\widetilde{B}$.  To check this claim, we keep in mind $\sum_{k} C^\dagger_{tk} C_{k q} = \delta_{tq}$ and $\sum_{\ell} C^\dagger_{q'\ell} C_{\ell t'} = \delta_{q' t'}$.  With $\%$ signifying integer modulus, we compute
\begin{align*}
\sum_{r} \widetilde{B}_{jr} \widetilde{B}^{+}_{rm} &= \sum_{r} B_{\lceil r/4 \rceil, (r-1)\%4 +1, \lceil j/2\rceil, (j-1)\%2 + 1} B^{+}_{\lceil r/4 \rceil, (r-1)\%4 +1, \lceil m/2\rceil, (m-1)\%2 + 1} \\
 &= \sum_{k,\ell} B_{k,\ell, \lceil j/2\rceil, (j-1)\%2 + 1} B^{+}_{k,\ell, \lceil m/2\rceil, (m-1)\%2 + 1} \\
 &= \sum_{k,\ell,q,q',t,t'} C_{k q} \mathscr{B}_{q,q',\lceil j/2\rceil, (j-1)\%2 + 1} C^{\dagger}_{q'\ell} C^\dagger_{t k} \mathscr{B}^{+}_{t,t',\lceil m/2\rceil, (m-1)\%2 + 1} C_{\ell t'} \\
 &= \sum_{q,q'} \mathscr{B}_{q,q',\lceil j/2\rceil, (j-1)\%2 + 1} \mathscr{B}^{+}_{q,q',\lceil m/2\rceil, (m-1)\%2 + 1} = \sum_{r} \widetilde{\mathscr{B}}_{j r} \widetilde{\mathscr{B}}^{+}_{r m} = \delta_{jm}.
\end{align*}
This itself implies three of the four properties one must check for $\widetilde{B}^+$ to be the pseudoinverse of $\widetilde{B}$: $\widetilde{B} \widetilde{B}^+ \widetilde{B} = \widetilde{B}$, $\widetilde{B}^+ \widetilde{B} \widetilde{B}^+ = \widetilde{B}^+$, and $\widetilde{B} \widetilde{B}^+$ is Hermitian.  For the remaining property, we must show that the $16 \times 16$ matrix $\widetilde{B}^+ \widetilde{B}$ is Hermitian.  Here we simply reverse the order of the multiplication above to find that
\begin{align*}
\sum_{r} \widetilde{B}^{+}_{jr} \widetilde{B}_{rm} 
 &= \sum_{k,\ell} B^{+}_{k,\ell, \lceil j/2\rceil, (j-1)\%2 + 1} B_{k,\ell, \lceil m/2\rceil, (m-1)\%2 + 1} \\
 &= \sum_{k,\ell,q,q',t,t'} C^\dagger_{q k} \mathscr{B}^{+}_{q,q',\lceil j/2\rceil, (j-1)\%2 + 1} C_{\ell q'} C_{k t} \mathscr{B}_{t,t',\lceil m/2\rceil, (m-1)\%2 + 1} C^{\dagger}_{t'\ell} \\
 &= \sum_{q,q'} \mathscr{B}^{+}_{q,q',\lceil j/2\rceil, (j-1)\%2 + 1} \mathscr{B}_{q,q',\lceil m/2\rceil, (m-1)\%2 + 1}  = \sum_{r} \widetilde{\mathscr{B}}^{+}_{j r} \widetilde{\mathscr{B}}_{r m} = (\widetilde{\mathscr{B}}^{+} \widetilde{\mathscr{B}})_{j m}.
\end{align*}
The matrix $\widetilde{\mathscr{B}}^{+} \widetilde{\mathscr{B}}$ is Hermitian; hence $\widetilde{B}^+$ is the Moore-Penrose pseudoinverse of $\widetilde{B}$.

\subsection{Evaluating the right-hand side of (\ref{eqn:memorylessprop})}
\label{sect:rhsqprop}
Consider the product of $\widetilde{B}$ with a Kronecker product of two matrices $X$ and $Y$:
\begin{align*}
(\widetilde{B} (X \otimes Y))_{js} &= \sum_{r} \widetilde{B}_{j r} (X \otimes Y)_{r s} \\
 &= \sum_{r} B_{\lceil r/4 \rceil, 
  (r-1) \% 4 + 1,\lceil j/2 \rceil , 
  (j-1) \% 2 + 1} X_{\lceil r/4 \rceil, \lceil s/4 \rceil} Y_{(r-1) \% 4 + 1, (s-1) \% 4 + 1} \\
  &= \sum_{k,\ell} B_{k,\ell,\lceil j/2 \rceil , 
  (j-1) \% 2 + 1} X_{k, \lceil s/4 \rceil} Y_{\ell, (s-1) \% 4 + 1} \\
  &= \sum_{k,\ell,q,q'} C^T_{qk} \mathscr{B}_{q,q',\lceil j/2 \rceil,  (j-1) \% 2 + 1} C^\dagger_{q' \ell} X_{k, \lceil s/4 \rceil} Y_{\ell, (s-1) \% 4 + 1}
\end{align*}
Recall (\ref{eqn:fullham}) and the identity $\exp(V W V^{-1}) = V \exp(W) V^{-1}$.  Using this, the unitarity of $C$, and 
\begin{equation}
\label{eqn:H1t}
H_1(t) = h(t) \otimes I + I \otimes h(t),
\end{equation}
we substitute
\begin{subequations}
\label{eqn:XY}
\begin{align}
X &= \exp(i H(t)^T \Delta t) = \overline{C} \exp(i H_1(t)^T \Delta t) C^T \\
Y &= \exp(-i H(t) \Delta t) = C \exp(-i H_1(t) \Delta t) C^\dagger 
\end{align}
\end{subequations}
and obtain
\begin{equation*}
(\widetilde{B} (X \otimes Y))_{js} = \sum_{q,q'} \mathscr{B}_{q,q',\lceil j/2 \rceil,  (j-1) \% 2 + 1} (\exp(i H_1(t)^T \Delta t) C^T)_{q,\lceil s/4 \rceil}  
 (\exp(-i H_1(t)  \Delta t)C^\dagger)_{q', (s-1) \% 4 + 1}.
\end{equation*}
Next we multiply on the right by $\widetilde{B}^+$; via (\ref{eqn:XY}), we see that $\widetilde{B} (X \otimes Y) \widetilde{B}^{+}$ is the matrix that multiplies $\operatorname{vec}(Q(t))$ on the right-hand side of (\ref{eqn:memorylessprop}).  Then
\begin{align*}
&(\widetilde{B} (X \otimes Y) \widetilde{B}^{+})_{jm} \\
&= \sum_{s} \sum_{q,q'} \mathscr{B}_{q,q',\lceil j/2 \rceil,  (j-1) \% 2 + 1} (\exp(i H_1(t)^T \Delta t)C^T)_{q,\lceil s/4 \rceil}  (\exp(-i H_1(t) \Delta t)C^\dagger )_{q', (s-1) \% 4 + 1} \widetilde{B}^{+}_{s m} \\
 &= \sum_{t,t',q,q'} \mathscr{B}_{q,q',\lceil j/2 \rceil,  (j-1) \% 2 + 1} (\exp(i H_1(t)^T \Delta t) C^T)_{q,t}  (\exp(-i H_1(t)  \Delta t)C^\dagger)_{q',t'} B^{+}_{t,t', \lceil m/2\rceil, (m-1)\%2 + 1} \\
 &=  \sum_{r,r',t,t',q,q'}   \mathscr{B}_{q,q',\lceil j/2 \rceil,  (j-1) \% 2 + 1} (\exp(i H_1(t)^T \Delta t)C^T)_{q,t}  (\exp(-i H_1(t) \Delta t)C^\dagger)_{q',t'} \overline{C}_{tr} \\
 &\qquad \qquad \quad \mathscr{B}^{+}_{r,r', \lceil m/2\rceil, (m-1)\%2 + 1} C_{t' r'}\\
 &= \sum_{r,r',q,q'} \mathscr{B}_{q,q',\lceil j/2 \rceil,  (j-1) \% 2 + 1} \exp(i H_1(t)^T \Delta t)_{q,r}  \exp(-i H_1(t) \Delta t)_{q',r'} \mathscr{B}^{+}_{r,r', \lceil m/2\rceil, (m-1)\%2 + 1}.
\end{align*}
We conclude that $\widetilde{B} (X \otimes Y) \widetilde{B}^{+} = 
\widetilde{\mathscr{B}} \Bigl(\exp(i H_1(t)^T \Delta t) \otimes \exp(-i H_1(t) \Delta t) \Bigr) \widetilde{\mathscr{B}}^{+}$.
Observe that all $C$ matrices have cancelled out.
We are now in a position to evaluate the right-hand side of (\ref{eqn:memorylessprop}).  Set
\begin{equation}
\label{eqn:Qtilde}
\widetilde{Q}(t) = Q(t)/4 - I.
\end{equation}
With $\ell=0$, we have $\mathbf{q}_{0}(t) = \operatorname{vec}(Q(t)) = \operatorname{vec}(4 \widetilde{Q}(t) + 4I)$.  Applying (\ref{eqn:bplus}), we have
\begin{align*}
\widetilde{B} &(X \otimes Y) \widetilde{B}^{+} \mathbf{q}_{0}(t) = \widetilde{\mathscr{B}} \Bigl(\exp(i H_1(t)^T \Delta t) \otimes \exp(-i H_1(t) \Delta t) \Bigr) \operatorname{vec}( \widetilde{Q}(t) \otimes I + I \otimes \widetilde{Q}(t) ) \\
 &= \widetilde{\mathscr{B}} \operatorname{vec} \Bigl( \exp(-i (h(t) \otimes I + I \otimes h(t)) \Delta t) ( \widetilde{Q}(t) \times I + I \otimes \widetilde{Q}(t) ) \exp(i (h(t) \otimes I + I \otimes h(t)) \Delta t) \Bigr).
\end{align*}
Now we apply $\exp(-i (h(t) \otimes I + I \otimes h(t)) \Delta t) = \exp(-i h(t) \Delta t) \otimes \exp(-i h(t) \Delta t)$,
and similarly for the $+i$ term.  With this, we have
\begin{align*}
\widetilde{B} & (X \otimes Y) \widetilde{B}^{+} \mathbf{q}_{0}(t) = \widetilde{\mathscr{B}} \operatorname{vec} \Bigl( (\exp(-i h(t) \Delta t) \otimes \exp(-i h(t) \Delta t))( \widetilde{Q}(t) \otimes I + I \otimes \widetilde{Q}(t) )\\
&  \qquad \qquad \qquad \qquad \qquad \qquad \qquad \qquad \qquad \qquad (\exp(i h(t) \Delta t) \otimes \exp(i h(t) \Delta t)) \Bigr) \\
&= \widetilde{\mathscr{B}} \operatorname{vec} \biggl( \Bigl[ 
\exp(-i h(t) \Delta t) \widetilde{Q}(t) \otimes \exp(-i h(t) \Delta t) + \exp(-i h(t) \Delta t) \otimes \exp(-i h(t) \Delta t) \widetilde{Q}(t) \Bigr] \\
&  \qquad \qquad \qquad \qquad \qquad \qquad \qquad \qquad \qquad \qquad (\exp(i h(t) \Delta t) \otimes \exp(i h(t) \Delta t)) \biggr) \\
&= \widetilde{\mathscr{B}} \operatorname{vec} \biggl(
\exp(-i h(t) \Delta t) \widetilde{Q}(t) \exp(i h(t) \Delta t) \otimes I + I \otimes \exp(-i h(t) \Delta t) \widetilde{Q}(t) \exp(i h(t) \Delta t) \biggr) \\
&= \widetilde{\mathscr{B}} \widetilde{\mathscr{B}}^+ \operatorname{vec} \biggl( 4 \exp(-i h(t) \Delta t) \widetilde{Q}(t) \exp(i h(t) \Delta t) + 4 I \biggr) \\
&= \operatorname{vec} \biggl( \exp(-i h(t) \Delta t) (4 \widetilde{Q}(t) + 4I) \exp(i h(t) \Delta t) \biggr) = \operatorname{vec} \biggl( \exp(-i h(t) \Delta t) Q(t) \exp(i h(t) \Delta t) \biggr).
\end{align*}
To obtain the fourth equality, we used (\ref{eqn:bplus}) in reverse.  To obtain the fifth equality, we used $\widetilde{\mathscr{B}} \widetilde{\mathscr{B}}^+ = I$.  The last equality follows from (\ref{eqn:Qtilde}).  Using the above result, we see that the propagation result (\ref{eqn:memorylessprop}) reduces to
\begin{equation}
\label{eqn:Qpropderived}
Q(t + \Delta t) = \exp(-i h(t) \Delta t) Q(t) \exp(i h(t) \Delta t).
\end{equation}
This exactly matches (\ref{eqn:Qallegedprop}).  Taking the $\Delta t \to 0$ limit, (\ref{eqn:Qpropderived}) becomes the Liouville-von Neumann equation (\ref{eqn:appendixQLVN}) that governs the true continuous-time dynamics of the 1RDM $Q(t)$.

\subsection{Generalizing the calculation to 6-31G}
\label{sect:appgen}
In 6-31G, our systems $\heh$ and $\htwo$ have $N=2$ electrons, $K=4$ atomic orbitals, and $N_C = 16$ total CI determinants.  We now have $16 \times 16$ full CI density matrices $P(t)$ and associated Hamiltonians $H(t)$, and we have $4 \times 4$ 1RDMs $Q(t)$ and associated Hamiltonians $h(t)$.

All of the algebraic structure detailed above still holds.  In particular, the Hamiltonians $H(t)$ and $h(t)$ are still related via (\ref{eqn:fullham}).  The marginalized outer product of Slater determinants (\ref{eqn:coreSLresult}) that we found above can be rephrased as the $K=2$ case of
\begin{equation}
\begin{split}
\label{eqn:slmatgen}
\left( \psi^\text{SL}_{\mathbf{i}(q)} (\psi^\text{SL}_{\mathbf{i}(q')})^\ast \right)_1 = \frac{1}{2} \left( \boldsymbol{\Phi} \otimes I + I \otimes \boldsymbol{\Phi} \right)  \\
\boldsymbol{\Phi} = \begin{bmatrix} \phi_1 \\ \phi_2 \\ \vdots \\ \phi_K \end{bmatrix} \begin{bmatrix} \phi_1^\ast & \phi_2^\ast & \cdots & \phi_K^\ast.
\end{bmatrix}
\end{split}
\end{equation}
In 6-31G, with $K=4$ and the index mapping $\mathbf{i}(q)$ that encodes the $16$ Slater determinants used in our calculation, (\ref{eqn:slmatgen}) still holds.  With this, we can numerically evaluate the matrices $\widetilde{\mathscr{B}}$ and $\widetilde{\mathscr{B}}^{+}$.  We find that $\widetilde{\mathscr{B}} \widetilde{\mathscr{B}}^{+} = I$, now the $16 \times 16$ identity.  Via direct calculation, we find that for any $4 \times 4$ matrix $W$, we have
\begin{equation}
\begin{split}
\label{eqn:newBBplus}
\widetilde{\mathscr{B}}^{+} \vvec(8 W + 4 I) = \vvec(W \otimes I + I \otimes W)  \\
\widetilde{\mathscr{B}} \vvec(W \otimes I + I \otimes W) =  \vvec(8 W + 4 I).
\end{split}
\end{equation}
These facts enable reuse of the derivation in Section \ref{sect:rhsqprop}, with (\ref{eqn:Qtilde}) redefined as $\widetilde{Q}(t) = Q(t)/8 - I/2$.  All remaining calculations hold with only dimension-related alterations, e.g., changing $4$ to $16$ and $2$ to $4$ in ceiling and modulus expressions.

For both STO-3G and 6-31G $2$-electron systems, the results of this Appendix prove that for any unitary $C$, if we use only the $1$-electron terms in our Hamiltonian, the scheme (\ref{eqn:Qprop}) propagates the 1RDM $Q(t)$ correctly with zero steps of memory required.

\section{Generalization of the derivation to continuous time}
\label{sect:continuoustime}
Suppose that instead of (\ref{eqn:motivlinsys}) we begin with the time-dependent linear differential equation
\begin{equation}
\label{eqn:contmls}
\frac{d}{dt} \mathbf{z}(t) = A(t) \mathbf{z}(t)
\end{equation}
with initial condition $\mathbf{z}(0)$.  As before $A(t)$ is an $n \times n$ complex matrix.  We assume that the system admits a fundamental solution matrix $\Phi(t)$ that satisfies
\[
\frac{d}{dt} \Phi(t) = A(t) \Phi(t)
\]
with initial condition $\Phi(0) = I$.  Then the solution of (\ref{eqn:contmls}) is $\mathbf{z}(t) = \Phi(t) \mathbf{z}(0)$.
By fundamental\cite{Chicone}, what is meant is that $\Phi(t)$ is invertible, so that we can define
\[
\Phi(t, s) = \Phi(t) \Phi(s)^{-1}.
\]
Then $\mathbf{z}(t) = \Phi(t) \mathbf{z}(0) = \Phi(t) \Phi(t-s)^{-1} \mathbf{z}(t-s) = \Phi(t,t-s) \mathbf{z}(t-s)$.  Based on this, we can write
\[
\mathbf{z}(t-s) =  \Phi(t,t-s)^{-1} \mathbf{z}(t) = (\Phi(t) \Phi(t-s)^{-1})^{-1} \mathbf{z}(t) = \Phi(t-s,t) \mathbf{z}(t).
\]
Multiplying both sides by $R$ and using (\ref{eqn:reduced}) yields
\[
\mathbf{y}(t-s) = R  \Phi(t-s,t) \mathbf{z}(t).
\]
Using this equation at $s = 0, \Delta s, \ldots, \ell \Delta s$, we can write
\[
\mathbf{Y}_{\ell}(t) = \begin{bmatrix}
\mathbf{y}(t)\\
\mathbf{y}(t-\Delta s)\\
\mathbf{y}(t-2\Delta s)\\
\vdots\\
\mathbf{y}(t-\ell\Delta s)
\end{bmatrix} = \underbrace{\begin{bmatrix}
R \\
R \Phi(t-\Delta s,t) \\
R \Phi(t-2 \Delta s,t) \\
\vdots\\
R \Phi(t-\ell \Delta s,t)
\end{bmatrix}}_{M(t)} \mathbf{z}(t).
\]
Then by differentiating (\ref{eqn:reduced}) with respect to $t$, we derive the following equation for $\mathbf{y}(t) \in \mathbb{C}^m$:
\[
\frac{d}{dt} \mathbf{y}(t) = R A(t) \mathbf{z}(t) = R A(t) M(t)^{+} \begin{bmatrix}
\mathbf{y}(t)\\
\mathbf{y}(t-\Delta s)\\
\mathbf{y}(t-2\Delta s)\\
\vdots\\
\mathbf{y}(t-\ell\Delta s)
\end{bmatrix}
\]
As before, we have assumed that $M(t)$ admits the pseudoinverse $M(t)^{+}$.  We defer further analysis of this system, especially its properties in the limit where $\Delta s \to 0$ while $\ell \Delta s$ remains constant, to future work.
\bibliography{mainREVISED}  

\begin{thebibliography}{84}%
\makeatletter
\providecommand \@ifxundefined [1]{%
 \@ifx{#1\undefined}
}%
\providecommand \@ifnum [1]{%
 \ifnum #1\expandafter \@firstoftwo
 \else \expandafter \@secondoftwo
 \fi
}%
\providecommand \@ifx [1]{%
 \ifx #1\expandafter \@firstoftwo
 \else \expandafter \@secondoftwo
 \fi
}%
\providecommand \natexlab [1]{#1}%
\providecommand \enquote  [1]{``#1''}%
\providecommand \bibnamefont  [1]{#1}%
\providecommand \bibfnamefont [1]{#1}%
\providecommand \citenamefont [1]{#1}%
\providecommand \href@noop [0]{\@secondoftwo}%
\providecommand \href [0]{\begingroup \@sanitize@url \@href}%
\providecommand \@href[1]{\@@startlink{#1}\@@href}%
\providecommand \@@href[1]{\endgroup#1\@@endlink}%
\providecommand \@sanitize@url [0]{\catcode `\\12\catcode `\$12\catcode
  `\&12\catcode `\#12\catcode `\^12\catcode `\_12\catcode `\%12\relax}%
\providecommand \@@startlink[1]{}%
\providecommand \@@endlink[0]{}%
\providecommand \url  [0]{\begingroup\@sanitize@url \@url }%
\providecommand \@url [1]{\endgroup\@href {#1}{\urlprefix }}%
\providecommand \urlprefix  [0]{URL }%
\providecommand \Eprint [0]{\href }%
\providecommand \doibase [0]{http://dx.doi.org/}%
\providecommand \selectlanguage [0]{\@gobble}%
\providecommand \bibinfo  [0]{\@secondoftwo}%
\providecommand \bibfield  [0]{\@secondoftwo}%
\providecommand \translation [1]{[#1]}%
\providecommand \BibitemOpen [0]{}%
\providecommand \bibitemStop [0]{}%
\providecommand \bibitemNoStop [0]{.\EOS\space}%
\providecommand \EOS [0]{\spacefactor3000\relax}%
\providecommand \BibitemShut  [1]{\csname bibitem#1\endcsname}%
\let\auto@bib@innerbib\@empty
\bibitem [{\citenamefont {Jortner}, \citenamefont {Bixon},\ and\ \citenamefont
  {Michel-Beyerle}(1990)}]{Jortner1990}%
  \BibitemOpen
  \bibfield  {author} {\bibinfo {author} {\bibfnamefont {J.}~\bibnamefont
  {Jortner}}, \bibinfo {author} {\bibfnamefont {M.}~\bibnamefont {Bixon}}, \
  and\ \bibinfo {author} {\bibfnamefont {M.~E.}\ \bibnamefont
  {Michel-Beyerle}},\ }\bibfield  {title} {\enquote {\bibinfo {title} {Electron
  transfer dynamics in photosynthesis},}\ }in\ \href {\doibase
  10.1007/978-94-009-0511-5_2} {\emph {\bibinfo {booktitle} {Current Research
  in Photosynthesis: Proceedings of the VIIIth International Conference on
  Photosynthesis Stockholm, Sweden, August 6-11, 1989}}},\ \bibinfo {editor}
  {edited by\ \bibinfo {editor} {\bibfnamefont {M.}~\bibnamefont
  {Baltscheffsky}}}\ (\bibinfo  {publisher} {Springer Netherlands},\ \bibinfo
  {address} {Dordrecht},\ \bibinfo {year} {1990})\ pp.\ \bibinfo {pages}
  {11--18}\BibitemShut {NoStop}%
\bibitem [{\citenamefont {Kling}\ and\ \citenamefont
  {Vrakking}(2008)}]{Kling2008}%
  \BibitemOpen
  \bibfield  {author} {\bibinfo {author} {\bibfnamefont {M.~F.}\ \bibnamefont
  {Kling}}\ and\ \bibinfo {author} {\bibfnamefont {M.~J.}\ \bibnamefont
  {Vrakking}},\ }\bibfield  {title} {\enquote {\bibinfo {title} {Attosecond
  electron dynamics},}\ }\href {\doibase
  https://doi.org/10.1146/annurev.physchem.59.032607.093532} {\bibfield
  {journal} {\bibinfo  {journal} {Annual Review of Physical Chemistry}\
  }\textbf {\bibinfo {volume} {59}},\ \bibinfo {pages} {463--492} (\bibinfo
  {year} {2008})}\BibitemShut {NoStop}%
\bibitem [{\citenamefont {Provorse}\ and\ \citenamefont
  {Isborn}(2016)}]{Provorse2016}%
  \BibitemOpen
  \bibfield  {author} {\bibinfo {author} {\bibfnamefont {M.~R.}\ \bibnamefont
  {Provorse}}\ and\ \bibinfo {author} {\bibfnamefont {C.~M.}\ \bibnamefont
  {Isborn}},\ }\bibfield  {title} {\enquote {\bibinfo {title} {Electron
  dynamics with real-time time-dependent density functional theory},}\ }\href
  {\doibase https://doi.org/10.1002/qua.25096} {\bibfield  {journal} {\bibinfo
  {journal} {International Journal of Quantum Chemistry}\ }\textbf {\bibinfo
  {volume} {116}},\ \bibinfo {pages} {739--749} (\bibinfo {year}
  {2016})}\BibitemShut {NoStop}%
\bibitem [{\citenamefont {Li}\ \emph {et~al.}(2020)\citenamefont {Li},
  \citenamefont {Govind}, \citenamefont {Isborn}, \citenamefont {DePrince},\
  and\ \citenamefont {Lopata}}]{Li2020}%
  \BibitemOpen
  \bibfield  {author} {\bibinfo {author} {\bibfnamefont {X.}~\bibnamefont
  {Li}}, \bibinfo {author} {\bibfnamefont {N.}~\bibnamefont {Govind}}, \bibinfo
  {author} {\bibfnamefont {C.}~\bibnamefont {Isborn}}, \bibinfo {author}
  {\bibfnamefont {A.~E.~I.}\ \bibnamefont {DePrince}}, \ and\ \bibinfo {author}
  {\bibfnamefont {K.}~\bibnamefont {Lopata}},\ }\bibfield  {title} {\enquote
  {\bibinfo {title} {Real-time time-dependent electronic structure theory},}\
  }\href {\doibase 10.1021/acs.chemrev.0c00223} {\bibfield  {journal} {\bibinfo
   {journal} {Chemical Reviews}\ }\textbf {\bibinfo {volume} {120}},\ \bibinfo
  {pages} {9951--9993} (\bibinfo {year} {2020})}\BibitemShut {NoStop}%
\bibitem [{\citenamefont {Shavitt}(1977)}]{Shavitt1977}%
  \BibitemOpen
  \bibfield  {author} {\bibinfo {author} {\bibfnamefont {I.}~\bibnamefont
  {Shavitt}},\ }\bibfield  {title} {\enquote {\bibinfo {title} {The method of
  configuration interaction},}\ }in\ \href {\doibase
  10.1007/978-1-4757-0887-5_6} {\emph {\bibinfo {booktitle} {Methods of
  Electronic Structure Theory}}},\ \bibinfo {editor} {edited by\ \bibinfo
  {editor} {\bibfnamefont {H.~F.}\ \bibnamefont {Schaefer}}}\ (\bibinfo
  {publisher} {Springer US},\ \bibinfo {address} {Boston, MA},\ \bibinfo {year}
  {1977})\ pp.\ \bibinfo {pages} {189--275}\BibitemShut {NoStop}%
\bibitem [{\citenamefont {McWeeny}(1989)}]{mcweeny1989methods}%
  \BibitemOpen
  \bibfield  {author} {\bibinfo {author} {\bibfnamefont {R.}~\bibnamefont
  {McWeeny}},\ }\href@noop {} {\emph {\bibinfo {title} {Methods of Molecular
  Quantum Mechanics}}},\ \bibinfo {edition} {2nd}\ ed.,\ Chemistry, Physical
  and Theoretical\ (\bibinfo  {publisher} {Academic Press},\ \bibinfo {year}
  {1989})\BibitemShut {NoStop}%
\bibitem [{\citenamefont {Szabo}\ and\ \citenamefont
  {Ostlund}(2012)}]{szabo2012modern}%
  \BibitemOpen
  \bibfield  {author} {\bibinfo {author} {\bibfnamefont {A.}~\bibnamefont
  {Szabo}}\ and\ \bibinfo {author} {\bibfnamefont {N.~S.}\ \bibnamefont
  {Ostlund}},\ }\href@noop {} {\emph {\bibinfo {title} {Modern Quantum
  Chemistry: Introduction to Advanced Electronic Structure Theory}}}\ (\bibinfo
   {publisher} {Dover Publications},\ \bibinfo {year} {2012})\ \bibinfo {note}
  {``Unabridged, unaltered republication of the First Edition, Revised,
  originally published by McGraw-Hill, New York in 1989''}\BibitemShut
  {NoStop}%
\bibitem [{\citenamefont {Shavitt}(1998)}]{Shavitt1998}%
  \BibitemOpen
  \bibfield  {author} {\bibinfo {author} {\bibfnamefont {I.}~\bibnamefont
  {Shavitt}},\ }\bibfield  {title} {\enquote {\bibinfo {title} {The history and
  evolution of configuration interaction},}\ }\href {\doibase
  10.1080/002689798168303} {\bibfield  {journal} {\bibinfo  {journal}
  {Molecular Physics}\ }\textbf {\bibinfo {volume} {94}},\ \bibinfo {pages}
  {3--17} (\bibinfo {year} {1998})}\BibitemShut {NoStop}%
\bibitem [{\citenamefont {Gao}\ \emph {et~al.}(2024)\citenamefont {Gao},
  \citenamefont {Imamura}, \citenamefont {Kasagi},\ and\ \citenamefont
  {Yoshida}}]{Gao2024}%
  \BibitemOpen
  \bibfield  {author} {\bibinfo {author} {\bibfnamefont {H.}~\bibnamefont
  {Gao}}, \bibinfo {author} {\bibfnamefont {S.}~\bibnamefont {Imamura}},
  \bibinfo {author} {\bibfnamefont {A.}~\bibnamefont {Kasagi}}, \ and\ \bibinfo
  {author} {\bibfnamefont {E.}~\bibnamefont {Yoshida}},\ }\bibfield  {title}
  {\enquote {\bibinfo {title} {Distributed implementation of full configuration
  interaction for one trillion determinants},}\ }\href {\doibase
  10.1021/acs.jctc.3c01190} {\bibfield  {journal} {\bibinfo  {journal} {Journal
  of Chemical Theory and Computation}\ }\textbf {\bibinfo {volume} {20}},\
  \bibinfo {pages} {1185--1192} (\bibinfo {year} {2024})}\BibitemShut {NoStop}%
\bibitem [{\citenamefont {L{\"o}wdin}(1995)}]{lowdin1995historical}%
  \BibitemOpen
  \bibfield  {author} {\bibinfo {author} {\bibfnamefont {P.-O.}\ \bibnamefont
  {L{\"o}wdin}},\ }\bibfield  {title} {\enquote {\bibinfo {title} {The
  historical development of the electron correlation problem},}\ }\href
  {\doibase 10.1002/qua.560550203} {\bibfield  {journal} {\bibinfo  {journal}
  {International Journal of Quantum Chemistry}\ }\textbf {\bibinfo {volume}
  {55}},\ \bibinfo {pages} {77--102} (\bibinfo {year} {1995})}\BibitemShut
  {NoStop}%
\bibitem [{\citenamefont {Ullrich}(2012)}]{Ullrich2012}%
  \BibitemOpen
  \bibfield  {author} {\bibinfo {author} {\bibfnamefont {C.}~\bibnamefont
  {Ullrich}},\ }\href {\doibase 10.1093/acprof:oso/9780199563029.001.0001}
  {\emph {\bibinfo {title} {Time-{D}ependent {D}ensity-{F}unctional {T}heory:
  {C}oncepts and {A}pplications}}},\ Oxford Graduate Texts\ (\bibinfo
  {publisher} {Oxford University Press},\ \bibinfo {address} {Oxford},\
  \bibinfo {year} {2012})\BibitemShut {NoStop}%
\bibitem [{\citenamefont {Runge}\ and\ \citenamefont
  {Gross}(1984)}]{runge1984density}%
  \BibitemOpen
  \bibfield  {author} {\bibinfo {author} {\bibfnamefont {E.}~\bibnamefont
  {Runge}}\ and\ \bibinfo {author} {\bibfnamefont {E.~K.~U.}\ \bibnamefont
  {Gross}},\ }\bibfield  {title} {\enquote {\bibinfo {title}
  {Density-functional theory for time-dependent systems},}\ }\href {\doibase
  10.1103/PhysRevLett.52.997} {\bibfield  {journal} {\bibinfo  {journal}
  {Physical Review Letters}\ }\textbf {\bibinfo {volume} {52}},\ \bibinfo
  {pages} {997} (\bibinfo {year} {1984})}\BibitemShut {NoStop}%
\bibitem [{\citenamefont {Gross}\ and\ \citenamefont
  {Maitra}(2012)}]{Gross2012}%
  \BibitemOpen
  \bibfield  {author} {\bibinfo {author} {\bibfnamefont {E.~K.~U.}\
  \bibnamefont {Gross}}\ and\ \bibinfo {author} {\bibfnamefont {N.~T.}\
  \bibnamefont {Maitra}},\ }\bibfield  {title} {\enquote {\bibinfo {title}
  {{Introduction to TDDFT}},}\ }in\ \href {\doibase
  10.1007/978-3-642-23518-4\_4} {\emph {\bibinfo {booktitle} {{Fundamentals of
  Time-Dependent Density Functional Theory}}}},\ \bibinfo {editor} {edited by\
  \bibinfo {editor} {\bibfnamefont {M.~A.}\ \bibnamefont {Marques}}, \bibinfo
  {editor} {\bibfnamefont {N.~T.}\ \bibnamefont {Maitra}}, \bibinfo {editor}
  {\bibfnamefont {F.~M.}\ \bibnamefont {Nogueira}}, \bibinfo {editor}
  {\bibfnamefont {E.}~\bibnamefont {Gross}}, \ and\ \bibinfo {editor}
  {\bibfnamefont {A.}~\bibnamefont {Rubio}}}\ (\bibinfo  {publisher} {Springer
  Berlin Heidelberg},\ \bibinfo {address} {Berlin, Heidelberg},\ \bibinfo
  {year} {2012})\ pp.\ \bibinfo {pages} {53--99}\BibitemShut {NoStop}%
\bibitem [{\citenamefont {Elliott}\ \emph {et~al.}(2012)\citenamefont
  {Elliott}, \citenamefont {Fuks}, \citenamefont {Rubio},\ and\ \citenamefont
  {Maitra}}]{elliott2012universal}%
  \BibitemOpen
  \bibfield  {author} {\bibinfo {author} {\bibfnamefont {P.}~\bibnamefont
  {Elliott}}, \bibinfo {author} {\bibfnamefont {J.~I.}\ \bibnamefont {Fuks}},
  \bibinfo {author} {\bibfnamefont {A.}~\bibnamefont {Rubio}}, \ and\ \bibinfo
  {author} {\bibfnamefont {N.~T.}\ \bibnamefont {Maitra}},\ }\bibfield  {title}
  {\enquote {\bibinfo {title} {Universal dynamical steps in the exact
  time-dependent exchange-correlation potential},}\ }\href {\doibase
  10.1103/PhysRevLett.109.266404} {\bibfield  {journal} {\bibinfo  {journal}
  {Physical Review Letters}\ }\textbf {\bibinfo {volume} {109}},\ \bibinfo
  {pages} {266404} (\bibinfo {year} {2012})}\BibitemShut {NoStop}%
\bibitem [{\citenamefont {Rappoport}\ and\ \citenamefont
  {Hutter}(2012)}]{Rappoport2012}%
  \BibitemOpen
  \bibfield  {author} {\bibinfo {author} {\bibfnamefont {D.}~\bibnamefont
  {Rappoport}}\ and\ \bibinfo {author} {\bibfnamefont {J.}~\bibnamefont
  {Hutter}},\ }\bibfield  {title} {\enquote {\bibinfo {title} {{Excited-State
  Properties and Dynamics}},}\ }in\ \href {\doibase
  10.1007/978-3-642-23518-4\_16} {\emph {\bibinfo {booktitle} {{Fundamentals of
  Time-Dependent Density Functional Theory}}}},\ \bibinfo {editor} {edited by\
  \bibinfo {editor} {\bibfnamefont {M.~A.}\ \bibnamefont {Marques}}, \bibinfo
  {editor} {\bibfnamefont {N.~T.}\ \bibnamefont {Maitra}}, \bibinfo {editor}
  {\bibfnamefont {F.~M.}\ \bibnamefont {Nogueira}}, \bibinfo {editor}
  {\bibfnamefont {E.}~\bibnamefont {Gross}}, \ and\ \bibinfo {editor}
  {\bibfnamefont {A.}~\bibnamefont {Rubio}}}\ (\bibinfo  {publisher} {Springer
  Berlin Heidelberg},\ \bibinfo {address} {Berlin, Heidelberg},\ \bibinfo
  {year} {2012})\ pp.\ \bibinfo {pages} {317--336}\BibitemShut {NoStop}%
\bibitem [{\citenamefont {Habenicht}\ \emph {et~al.}(2014)\citenamefont
  {Habenicht}, \citenamefont {Tani}, \citenamefont {Provorse},\ and\
  \citenamefont {Isborn}}]{Habenicht2014}%
  \BibitemOpen
  \bibfield  {author} {\bibinfo {author} {\bibfnamefont {B.~F.}\ \bibnamefont
  {Habenicht}}, \bibinfo {author} {\bibfnamefont {N.~P.}\ \bibnamefont {Tani}},
  \bibinfo {author} {\bibfnamefont {M.~R.}\ \bibnamefont {Provorse}}, \ and\
  \bibinfo {author} {\bibfnamefont {C.~M.}\ \bibnamefont {Isborn}},\ }\bibfield
   {title} {\enquote {\bibinfo {title} {{Two-electron Rabi oscillations in
  real-time time-dependent density-functional theory}},}\ }\href {\doibase
  10.1063/1.4900514} {\bibfield  {journal} {\bibinfo  {journal} {The Journal of
  Chemical Physics}\ }\textbf {\bibinfo {volume} {141}},\ \bibinfo {pages}
  {184112} (\bibinfo {year} {2014})}\BibitemShut {NoStop}%
\bibitem [{\citenamefont {Provorse}, \citenamefont {Habenicht},\ and\
  \citenamefont {Isborn}(2015)}]{Provorse2015}%
  \BibitemOpen
  \bibfield  {author} {\bibinfo {author} {\bibfnamefont {M.~R.}\ \bibnamefont
  {Provorse}}, \bibinfo {author} {\bibfnamefont {B.~F.}\ \bibnamefont
  {Habenicht}}, \ and\ \bibinfo {author} {\bibfnamefont {C.~M.}\ \bibnamefont
  {Isborn}},\ }\bibfield  {title} {\enquote {\bibinfo {title} {Peak-shifting in
  real-time time-dependent density functional theory},}\ }\href {\doibase
  10.1021/acs.jctc.5b00559} {\bibfield  {journal} {\bibinfo  {journal} {Journal
  of Chemical Theory and Computation}\ }\textbf {\bibinfo {volume} {11}},\
  \bibinfo {pages} {4791--4802} (\bibinfo {year} {2015})}\BibitemShut {NoStop}%
\bibitem [{\citenamefont {Lacombe}\ and\ \citenamefont
  {Maitra}(2020)}]{lacombe2020developing}%
  \BibitemOpen
  \bibfield  {author} {\bibinfo {author} {\bibfnamefont {L.}~\bibnamefont
  {Lacombe}}\ and\ \bibinfo {author} {\bibfnamefont {N.~T.}\ \bibnamefont
  {Maitra}},\ }\bibfield  {title} {\enquote {\bibinfo {title} {Developing new
  and understanding old approximations in {TDDFT}},}\ }\href {\doibase
  10.1039/D0FD00049C} {\bibfield  {journal} {\bibinfo  {journal} {Faraday
  Discussions}\ }\textbf {\bibinfo {volume} {224}},\ \bibinfo {pages}
  {382--401} (\bibinfo {year} {2020})}\BibitemShut {NoStop}%
\bibitem [{\citenamefont {Ranka}\ and\ \citenamefont
  {Isborn}(2023)}]{Ranka2023}%
  \BibitemOpen
  \bibfield  {author} {\bibinfo {author} {\bibfnamefont {K.}~\bibnamefont
  {Ranka}}\ and\ \bibinfo {author} {\bibfnamefont {C.~M.}\ \bibnamefont
  {Isborn}},\ }\bibfield  {title} {\enquote {\bibinfo {title} {{Size-dependent
  errors in real-time electron density propagation}},}\ }\href {\doibase
  10.1063/5.0142515} {\bibfield  {journal} {\bibinfo  {journal} {The Journal of
  Chemical Physics}\ }\textbf {\bibinfo {volume} {158}},\ \bibinfo {pages}
  {174102} (\bibinfo {year} {2023})}\BibitemShut {NoStop}%
\bibitem [{\citenamefont {Suzuki}, \citenamefont {Nagai},\ and\ \citenamefont
  {Haruyama}(2020)}]{Suzuki2020}%
  \BibitemOpen
  \bibfield  {author} {\bibinfo {author} {\bibfnamefont {Y.}~\bibnamefont
  {Suzuki}}, \bibinfo {author} {\bibfnamefont {R.}~\bibnamefont {Nagai}}, \
  and\ \bibinfo {author} {\bibfnamefont {J.}~\bibnamefont {Haruyama}},\
  }\bibfield  {title} {\enquote {\bibinfo {title} {Machine learning
  exchange-correlation potential in time-dependent density-functional
  theory},}\ }\href {\doibase 10.1103/PhysRevA.101.050501} {\bibfield
  {journal} {\bibinfo  {journal} {Phys. Rev. A}\ }\textbf {\bibinfo {volume}
  {101}},\ \bibinfo {pages} {050501} (\bibinfo {year} {2020})}\BibitemShut
  {NoStop}%
\bibitem [{\citenamefont {Bhat}\ \emph {et~al.}(2022)\citenamefont {Bhat},
  \citenamefont {Collins}, \citenamefont {Gupta},\ and\ \citenamefont
  {Isborn}}]{bhat2021dynamic}%
  \BibitemOpen
  \bibfield  {author} {\bibinfo {author} {\bibfnamefont {H.~S.}\ \bibnamefont
  {Bhat}}, \bibinfo {author} {\bibfnamefont {K.}~\bibnamefont {Collins}},
  \bibinfo {author} {\bibfnamefont {P.}~\bibnamefont {Gupta}}, \ and\ \bibinfo
  {author} {\bibfnamefont {C.~M.}\ \bibnamefont {Isborn}},\ }\bibfield  {title}
  {\enquote {\bibinfo {title} {Dynamic learning of correlation potentials for a
  time-dependent {K}ohn-{S}ham system},}\ }in\ \href
  {https://proceedings.mlr.press/v168/bhat22a.html} {\emph {\bibinfo
  {booktitle} {Proceedings of The 4th Annual Learning for Dynamics and Control
  Conference}}},\ \bibinfo {series} {Proceedings of Machine Learning Research},
  Vol.\ \bibinfo {volume} {{168, }},\ \bibinfo {editor} {edited by\ \bibinfo
  {editor} {\bibfnamefont {R.}~\bibnamefont {Firoozi}}, \bibinfo {editor}
  {\bibfnamefont {N.}~\bibnamefont {Mehr}}, \bibinfo {editor} {\bibfnamefont
  {E.}~\bibnamefont {Yel}}, \bibinfo {editor} {\bibfnamefont {R.}~\bibnamefont
  {Antonova}}, \bibinfo {editor} {\bibfnamefont {J.}~\bibnamefont {Bohg}},
  \bibinfo {editor} {\bibfnamefont {M.}~\bibnamefont {Schwager}}, \ and\
  \bibinfo {editor} {\bibfnamefont {M.}~\bibnamefont {Kochenderfer}}}\
  (\bibinfo  {publisher} {PMLR},\ \bibinfo {year} {2022})\ pp.\ \bibinfo
  {pages} {546--558}\BibitemShut {NoStop}%
\bibitem [{\citenamefont {Maitra}, \citenamefont {Burke},\ and\ \citenamefont
  {Woodward}(2002)}]{maitra2002memory}%
  \BibitemOpen
  \bibfield  {author} {\bibinfo {author} {\bibfnamefont {N.~T.}\ \bibnamefont
  {Maitra}}, \bibinfo {author} {\bibfnamefont {K.}~\bibnamefont {Burke}}, \
  and\ \bibinfo {author} {\bibfnamefont {C.}~\bibnamefont {Woodward}},\
  }\bibfield  {title} {\enquote {\bibinfo {title} {Memory in time-dependent
  density functional theory},}\ }\href {\doibase 10.1103/PhysRevLett.89.023002}
  {\bibfield  {journal} {\bibinfo  {journal} {Phys. Rev. Lett.}\ }\textbf
  {\bibinfo {volume} {89}},\ \bibinfo {pages} {023002} (\bibinfo {year}
  {2002})}\BibitemShut {NoStop}%
\bibitem [{\citenamefont {Zwanzig}(1961)}]{Zwanzig1961}%
  \BibitemOpen
  \bibfield  {author} {\bibinfo {author} {\bibfnamefont {R.}~\bibnamefont
  {Zwanzig}},\ }\bibfield  {title} {\enquote {\bibinfo {title} {Memory effects
  in irreversible thermodynamics},}\ }\href {\doibase 10.1103/PhysRev.124.983}
  {\bibfield  {journal} {\bibinfo  {journal} {Phys. Rev.}\ }\textbf {\bibinfo
  {volume} {124}},\ \bibinfo {pages} {983--992} (\bibinfo {year}
  {1961})}\BibitemShut {NoStop}%
\bibitem [{\citenamefont {Mori}(1965)}]{Mori1965}%
  \BibitemOpen
  \bibfield  {author} {\bibinfo {author} {\bibfnamefont {H.}~\bibnamefont
  {Mori}},\ }\bibfield  {title} {\enquote {\bibinfo {title} {{Transport,
  Collective Motion, and Brownian Motion}},}\ }\href {\doibase
  10.1143/PTP.33.423} {\bibfield  {journal} {\bibinfo  {journal} {Progress of
  Theoretical Physics}\ }\textbf {\bibinfo {volume} {33}},\ \bibinfo {pages}
  {423--455} (\bibinfo {year} {1965})}\BibitemShut {NoStop}%
\bibitem [{\citenamefont {Chorin}, \citenamefont {Hald},\ and\ \citenamefont
  {Kupferman}(2000)}]{chorin2000optimal}%
  \BibitemOpen
  \bibfield  {author} {\bibinfo {author} {\bibfnamefont {A.~J.}\ \bibnamefont
  {Chorin}}, \bibinfo {author} {\bibfnamefont {O.~H.}\ \bibnamefont {Hald}}, \
  and\ \bibinfo {author} {\bibfnamefont {R.}~\bibnamefont {Kupferman}},\
  }\bibfield  {title} {\enquote {\bibinfo {title} {Optimal prediction and the
  {M}ori-{Z}wanzig representation of irreversible processes},}\ }\href
  {\doibase 10.1073/pnas.97.7.2968} {\bibfield  {journal} {\bibinfo  {journal}
  {Proceedings of the National Academy of Sciences}\ }\textbf {\bibinfo
  {volume} {97}},\ \bibinfo {pages} {2968--2973} (\bibinfo {year}
  {2000})}\BibitemShut {NoStop}%
\bibitem [{\citenamefont {Darve}, \citenamefont {Solomon},\ and\ \citenamefont
  {Kia}(2009)}]{Darve2009}%
  \BibitemOpen
  \bibfield  {author} {\bibinfo {author} {\bibfnamefont {E.}~\bibnamefont
  {Darve}}, \bibinfo {author} {\bibfnamefont {J.}~\bibnamefont {Solomon}}, \
  and\ \bibinfo {author} {\bibfnamefont {A.}~\bibnamefont {Kia}},\ }\bibfield
  {title} {\enquote {\bibinfo {title} {Computing generalized {L}angevin
  equations and generalized {F}okker–{P}lanck equations},}\ }\href {\doibase
  10.1073/pnas.0902633106} {\bibfield  {journal} {\bibinfo  {journal}
  {Proceedings of the National Academy of Sciences}\ }\textbf {\bibinfo
  {volume} {106}},\ \bibinfo {pages} {10884--10889} (\bibinfo {year}
  {2009})}\BibitemShut {NoStop}%
\bibitem [{\citenamefont {Chorin}\ and\ \citenamefont
  {Hald}(2014)}]{Chorin2014}%
  \BibitemOpen
  \bibfield  {author} {\bibinfo {author} {\bibfnamefont {A.~J.}\ \bibnamefont
  {Chorin}}\ and\ \bibinfo {author} {\bibfnamefont {O.~H.}\ \bibnamefont
  {Hald}},\ }\href {\doibase 10.1007/978-1-4614-6980-3} {\emph {\bibinfo
  {title} {Stochastic Tools in Mathematics and Science}}},\ \bibinfo {edition}
  {3rd}\ ed.\ (\bibinfo  {publisher} {Springer New York},\ \bibinfo {address}
  {New York, NY},\ \bibinfo {year} {2014})\BibitemShut {NoStop}%
\bibitem [{\citenamefont {Takens}(1981)}]{takens2006detecting}%
  \BibitemOpen
  \bibfield  {author} {\bibinfo {author} {\bibfnamefont {F.}~\bibnamefont
  {Takens}},\ }\bibfield  {title} {\enquote {\bibinfo {title} {Detecting
  strange attractors in turbulence},}\ }in\ \href {\doibase 10.1007/BFb0091924}
  {\emph {\bibinfo {booktitle} {Dynamical Systems and Turbulence, Warwick
  1980}}}\ (\bibinfo {organization} {Springer},\ \bibinfo {year} {1981})\ pp.\
  \bibinfo {pages} {366--381}\BibitemShut {NoStop}%
\bibitem [{\citenamefont {Stark}(99 6)}]{Stark1999}%
  \BibitemOpen
  \bibfield  {author} {\bibinfo {author} {\bibfnamefont {J.}~\bibnamefont
  {Stark}},\ }\bibfield  {title} {\enquote {\bibinfo {title} {Delay embeddings
  for forced systems. {I}. {D}eterministic forcing},}\ }\href {\doibase
  10.1007/s003329900072} {\bibfield  {journal} {\bibinfo  {journal} {Journal of
  Nonlinear Science}\ }\textbf {\bibinfo {volume} {9}} (\bibinfo {year}
  {1999-6}),\ 10.1007/s003329900072}\BibitemShut {NoStop}%
\bibitem [{\citenamefont {Broer}\ and\ \citenamefont
  {Takens}(2010)}]{BroerHenk2010DSaC}%
  \BibitemOpen
  \bibfield  {author} {\bibinfo {author} {\bibfnamefont {H.}~\bibnamefont
  {Broer}}\ and\ \bibinfo {author} {\bibfnamefont {F.}~\bibnamefont {Takens}},\
  }\href {\doibase 10.1007/978-1-4419-6870-8} {\emph {\bibinfo {title}
  {Dynamical Systems and Chaos}}},\ \bibinfo {series} {Applied Mathematical
  Sciences}, Vol.\ \bibinfo {volume} {172}\ (\bibinfo  {publisher} {Springer},\
  \bibinfo {address} {New York},\ \bibinfo {year} {2010})\BibitemShut {NoStop}%
\bibitem [{\citenamefont {Mauroy}, \citenamefont {Mezi\'{c}},\ and\
  \citenamefont {Susuki}(2020)}]{MauroyAlexandre2020TKOi}%
  \BibitemOpen
  \bibfield  {author} {\bibinfo {author} {\bibfnamefont {A.}~\bibnamefont
  {Mauroy}}, \bibinfo {author} {\bibfnamefont {I.}~\bibnamefont {Mezi\'{c}}}, \
  and\ \bibinfo {author} {\bibfnamefont {Y.}~\bibnamefont {Susuki}},\ }\href
  {\doibase 10.1007/978-3-030-35713-9} {\emph {\bibinfo {title} {The Koopman
  Operator in Systems and Control: Concepts, Methodologies, and
  Applications}}},\ \bibinfo {edition} {1st}\ ed.,\ \bibinfo {series} {Lecture
  Notes in Control and Information Sciences}, Vol.\ \bibinfo {volume} {484}\
  (\bibinfo  {publisher} {Springer},\ \bibinfo {address} {Netherlands},\
  \bibinfo {year} {2020})\BibitemShut {NoStop}%
\bibitem [{\citenamefont {Otto}\ and\ \citenamefont
  {Rowley}(2021)}]{OttoRowley2021}%
  \BibitemOpen
  \bibfield  {author} {\bibinfo {author} {\bibfnamefont {S.~E.}\ \bibnamefont
  {Otto}}\ and\ \bibinfo {author} {\bibfnamefont {C.~W.}\ \bibnamefont
  {Rowley}},\ }\bibfield  {title} {\enquote {\bibinfo {title} {Koopman
  operators for estimation and control of dynamical systems},}\ }\href
  {https://www.annualreviews.org/content/journals/10.1146/annurev-control-071020-010108}
  {\bibfield  {journal} {\bibinfo  {journal} {Annual Review of Control,
  Robotics, and Autonomous Systems}\ }\textbf {\bibinfo {volume} {4}},\
  \bibinfo {pages} {59--87} (\bibinfo {year} {2021})}\BibitemShut {NoStop}%
\bibitem [{\citenamefont {Brunton}\ \emph {et~al.}(2022)\citenamefont
  {Brunton}, \citenamefont {Budi\v{s}i\'{c}}, \citenamefont {Kaiser},\ and\
  \citenamefont {Kutz}}]{Brunton2022}%
  \BibitemOpen
  \bibfield  {author} {\bibinfo {author} {\bibfnamefont {S.~L.}\ \bibnamefont
  {Brunton}}, \bibinfo {author} {\bibfnamefont {M.}~\bibnamefont
  {Budi\v{s}i\'{c}}}, \bibinfo {author} {\bibfnamefont {E.}~\bibnamefont
  {Kaiser}}, \ and\ \bibinfo {author} {\bibfnamefont {J.~N.}\ \bibnamefont
  {Kutz}},\ }\bibfield  {title} {\enquote {\bibinfo {title} {{Modern Koopman
  Theory for Dynamical Systems}},}\ }\href {\doibase 10.1137/21M1401243}
  {\bibfield  {journal} {\bibinfo  {journal} {SIAM Review}\ }\textbf {\bibinfo
  {volume} {64}},\ \bibinfo {pages} {229--340} (\bibinfo {year}
  {2022})}\BibitemShut {NoStop}%
\bibitem [{\citenamefont {Lin}\ \emph {et~al.}(2021)\citenamefont {Lin},
  \citenamefont {Tian}, \citenamefont {Livescu},\ and\ \citenamefont
  {Anghel}}]{Lin2021}%
  \BibitemOpen
  \bibfield  {author} {\bibinfo {author} {\bibfnamefont {Y.~T.}\ \bibnamefont
  {Lin}}, \bibinfo {author} {\bibfnamefont {Y.}~\bibnamefont {Tian}}, \bibinfo
  {author} {\bibfnamefont {D.}~\bibnamefont {Livescu}}, \ and\ \bibinfo
  {author} {\bibfnamefont {M.}~\bibnamefont {Anghel}},\ }\bibfield  {title}
  {\enquote {\bibinfo {title} {Data-driven learning for the {M}ori-{Z}wanzig
  formalism: {A} generalization of the {K}oopman learning framework},}\ }\href
  {\doibase 10.1137/21M1401759} {\bibfield  {journal} {\bibinfo  {journal}
  {SIAM Journal on Applied Dynamical Systems}\ }\textbf {\bibinfo {volume}
  {20}},\ \bibinfo {pages} {2558--2601} (\bibinfo {year} {2021})}\BibitemShut
  {NoStop}%
\bibitem [{\citenamefont {Mezi\'{c}}\ and\ \citenamefont
  {Banaszuk}(2004)}]{Mezic2004}%
  \BibitemOpen
  \bibfield  {author} {\bibinfo {author} {\bibfnamefont {I.}~\bibnamefont
  {Mezi\'{c}}}\ and\ \bibinfo {author} {\bibfnamefont {A.}~\bibnamefont
  {Banaszuk}},\ }\bibfield  {title} {\enquote {\bibinfo {title} {Comparison of
  systems with complex behavior},}\ }\href {\doibase
  https://doi.org/10.1016/j.physd.2004.06.015} {\bibfield  {journal} {\bibinfo
  {journal} {Physica D: Nonlinear Phenomena}\ }\textbf {\bibinfo {volume}
  {197}},\ \bibinfo {pages} {101--133} (\bibinfo {year} {2004})}\BibitemShut
  {NoStop}%
\bibitem [{\citenamefont {Kutz}\ \emph {et~al.}(2016)\citenamefont {Kutz},
  \citenamefont {Brunton}, \citenamefont {Brunton},\ and\ \citenamefont
  {Proctor}}]{kutz2016dynamic}%
  \BibitemOpen
  \bibfield  {author} {\bibinfo {author} {\bibfnamefont {J.~N.}\ \bibnamefont
  {Kutz}}, \bibinfo {author} {\bibfnamefont {S.~L.}\ \bibnamefont {Brunton}},
  \bibinfo {author} {\bibfnamefont {B.~W.}\ \bibnamefont {Brunton}}, \ and\
  \bibinfo {author} {\bibfnamefont {J.~L.}\ \bibnamefont {Proctor}},\ }\href
  {\doibase 10.1137/1.9781611974508} {\emph {\bibinfo {title} {Dynamic Mode
  Decomposition: Data-Driven Modeling of Complex Systems}}}\ (\bibinfo
  {publisher} {SIAM},\ \bibinfo {year} {2016})\BibitemShut {NoStop}%
\bibitem [{\citenamefont {Kamb}\ \emph {et~al.}(2020)\citenamefont {Kamb},
  \citenamefont {Kaiser}, \citenamefont {Brunton},\ and\ \citenamefont
  {Kutz}}]{Kamb2020}%
  \BibitemOpen
  \bibfield  {author} {\bibinfo {author} {\bibfnamefont {M.}~\bibnamefont
  {Kamb}}, \bibinfo {author} {\bibfnamefont {E.}~\bibnamefont {Kaiser}},
  \bibinfo {author} {\bibfnamefont {S.~L.}\ \bibnamefont {Brunton}}, \ and\
  \bibinfo {author} {\bibfnamefont {J.~N.}\ \bibnamefont {Kutz}},\ }\bibfield
  {title} {\enquote {\bibinfo {title} {Time-delay observables for {K}oopman:
  {T}heory and applications},}\ }\href {\doibase 10.1137/18M1216572} {\bibfield
   {journal} {\bibinfo  {journal} {SIAM Journal on Applied Dynamical Systems}\
  }\textbf {\bibinfo {volume} {19}},\ \bibinfo {pages} {886--917} (\bibinfo
  {year} {2020})}\BibitemShut {NoStop}%
\bibitem [{\citenamefont {Stefanucci}\ and\ \citenamefont
  {Van~Leeuwen}(2013)}]{stefanucci2013nonequilibrium}%
  \BibitemOpen
  \bibfield  {author} {\bibinfo {author} {\bibfnamefont {G.}~\bibnamefont
  {Stefanucci}}\ and\ \bibinfo {author} {\bibfnamefont {R.}~\bibnamefont
  {Van~Leeuwen}},\ }\href@noop {} {\emph {\bibinfo {title} {Nonequilibrium
  Many-Body Theory of Quantum Systems: A Modern Introduction}}}\ (\bibinfo
  {publisher} {Cambridge University Press},\ \bibinfo {year}
  {2013})\BibitemShut {NoStop}%
\bibitem [{\citenamefont {Kadanoff}(2018)}]{kadanoff2018quantum}%
  \BibitemOpen
  \bibfield  {author} {\bibinfo {author} {\bibfnamefont {L.~P.}\ \bibnamefont
  {Kadanoff}},\ }\href@noop {} {\emph {\bibinfo {title} {Quantum Statistical
  Mechanics}}}\ (\bibinfo  {publisher} {CRC Press},\ \bibinfo {year}
  {2018})\BibitemShut {NoStop}%
\bibitem [{\citenamefont {Negele}(1982)}]{negele1982mean}%
  \BibitemOpen
  \bibfield  {author} {\bibinfo {author} {\bibfnamefont {J.~W.}\ \bibnamefont
  {Negele}},\ }\bibfield  {title} {\enquote {\bibinfo {title} {The mean-field
  theory of nuclear structure and dynamics},}\ }\href {\doibase
  10.1103/RevModPhys.54.913} {\bibfield  {journal} {\bibinfo  {journal}
  {Reviews of Modern Physics}\ }\textbf {\bibinfo {volume} {54}},\ \bibinfo
  {pages} {913} (\bibinfo {year} {1982})}\BibitemShut {NoStop}%
\bibitem [{\citenamefont {Onida}, \citenamefont {Reining},\ and\ \citenamefont
  {Rubio}(2002)}]{RevModPhys.74.601}%
  \BibitemOpen
  \bibfield  {author} {\bibinfo {author} {\bibfnamefont {G.}~\bibnamefont
  {Onida}}, \bibinfo {author} {\bibfnamefont {L.}~\bibnamefont {Reining}}, \
  and\ \bibinfo {author} {\bibfnamefont {A.}~\bibnamefont {Rubio}},\ }\bibfield
   {title} {\enquote {\bibinfo {title} {Electronic excitations:
  density-functional versus many-body {G}reen's-function approaches},}\ }\href
  {\doibase 10.1103/RevModPhys.74.601} {\bibfield  {journal} {\bibinfo
  {journal} {Rev. Mod. Phys.}\ }\textbf {\bibinfo {volume} {74}},\ \bibinfo
  {pages} {601--659} (\bibinfo {year} {2002})}\BibitemShut {NoStop}%
\bibitem [{\citenamefont {Aryasetiawan}\ and\ \citenamefont
  {Gunnarsson}(1998)}]{aryasetiawan1998gw}%
  \BibitemOpen
  \bibfield  {author} {\bibinfo {author} {\bibfnamefont {F.}~\bibnamefont
  {Aryasetiawan}}\ and\ \bibinfo {author} {\bibfnamefont {O.}~\bibnamefont
  {Gunnarsson}},\ }\bibfield  {title} {\enquote {\bibinfo {title} {The {GW}
  method},}\ }\href {\doibase 10.1088/0034-4885/61/3/002} {\bibfield  {journal}
  {\bibinfo  {journal} {Reports on Progress in Physics}\ }\textbf {\bibinfo
  {volume} {61}},\ \bibinfo {pages} {237} (\bibinfo {year} {1998})}\BibitemShut
  {NoStop}%
\bibitem [{\citenamefont {Reining}(2018)}]{reining2018gw}%
  \BibitemOpen
  \bibfield  {author} {\bibinfo {author} {\bibfnamefont {L.}~\bibnamefont
  {Reining}},\ }\bibfield  {title} {\enquote {\bibinfo {title} {The {GW}
  approximation: content, successes and limitations},}\ }\href {\doibase
  10.1002/wcms.1344} {\bibfield  {journal} {\bibinfo  {journal} {Wiley
  Interdisciplinary Reviews: Computational Molecular Science}\ }\textbf
  {\bibinfo {volume} {8}},\ \bibinfo {pages} {e1344} (\bibinfo {year}
  {2018})}\BibitemShut {NoStop}%
\bibitem [{\citenamefont {Hedin}(1999)}]{hedin1999correlation}%
  \BibitemOpen
  \bibfield  {author} {\bibinfo {author} {\bibfnamefont {L.}~\bibnamefont
  {Hedin}},\ }\bibfield  {title} {\enquote {\bibinfo {title} {On correlation
  effects in electron spectroscopies and the {GW} approximation},}\ }\href
  {\doibase 10.1088/0953-8984/11/42/201} {\bibfield  {journal} {\bibinfo
  {journal} {Journal of Physics: Condensed Matter}\ }\textbf {\bibinfo {volume}
  {11}},\ \bibinfo {pages} {R489} (\bibinfo {year} {1999})}\BibitemShut
  {NoStop}%
\bibitem [{\citenamefont {Bassi}\ \emph {et~al.}(2023)\citenamefont {Bassi},
  \citenamefont {Zhu}, \citenamefont {Liang}, \citenamefont {Yin},
  \citenamefont {Reeves}, \citenamefont {Vl{\v{c}}ek},\ and\ \citenamefont
  {Yang}}]{Bassi2023}%
  \BibitemOpen
  \bibfield  {author} {\bibinfo {author} {\bibfnamefont {H.}~\bibnamefont
  {Bassi}}, \bibinfo {author} {\bibfnamefont {Y.}~\bibnamefont {Zhu}}, \bibinfo
  {author} {\bibfnamefont {S.}~\bibnamefont {Liang}}, \bibinfo {author}
  {\bibfnamefont {J.}~\bibnamefont {Yin}}, \bibinfo {author} {\bibfnamefont
  {C.~C.}\ \bibnamefont {Reeves}}, \bibinfo {author} {\bibfnamefont
  {V.}~\bibnamefont {Vl{\v{c}}ek}}, \ and\ \bibinfo {author} {\bibfnamefont
  {C.}~\bibnamefont {Yang}},\ }\bibfield  {title} {\enquote {\bibinfo {title}
  {Learning nonlinear integral operators via recurrent neural networks and its
  application in solving integro-differential equations},}\ }\href {\doibase
  10.1016/j.mlwa.2023.100524} {\bibfield  {journal} {\bibinfo  {journal}
  {Machine Learning with Applications}\ ,\ \bibinfo {pages} {100524}} (\bibinfo
  {year} {2023})}\BibitemShut {NoStop}%
\bibitem [{\citenamefont {Zhu}\ \emph {et~al.}(2024)\citenamefont {Zhu},
  \citenamefont {Yin}, \citenamefont {Reeves}, \citenamefont {Yang},\ and\
  \citenamefont {Vlcek}}]{zhu2024predicting}%
  \BibitemOpen
  \bibfield  {author} {\bibinfo {author} {\bibfnamefont {Y.}~\bibnamefont
  {Zhu}}, \bibinfo {author} {\bibfnamefont {J.}~\bibnamefont {Yin}}, \bibinfo
  {author} {\bibfnamefont {C.~C.}\ \bibnamefont {Reeves}}, \bibinfo {author}
  {\bibfnamefont {C.}~\bibnamefont {Yang}}, \ and\ \bibinfo {author}
  {\bibfnamefont {V.}~\bibnamefont {Vlcek}},\ }\href {\doibase
  10.48550/arXiv.2407.09773} {\enquote {\bibinfo {title} {Predicting
  nonequilibrium {G}reen's function dynamics and photoemission spectra via
  nonlinear integral operator learning},}\ } (\bibinfo {year} {2024}),\
  \bibinfo {note} {arXiv:2407.09773}\BibitemShut {NoStop}%
\bibitem [{\citenamefont {Sch\"{a}fer-Bung}\ \emph {et~al.}(2011)\citenamefont
  {Sch\"{a}fer-Bung}, \citenamefont {Hartmann}, \citenamefont {Schmidt},\ and\
  \citenamefont {Sch\"{u}tte}}]{Schmidt2011}%
  \BibitemOpen
  \bibfield  {author} {\bibinfo {author} {\bibfnamefont {B.}~\bibnamefont
  {Sch\"{a}fer-Bung}}, \bibinfo {author} {\bibfnamefont {C.}~\bibnamefont
  {Hartmann}}, \bibinfo {author} {\bibfnamefont {B.}~\bibnamefont {Schmidt}}, \
  and\ \bibinfo {author} {\bibfnamefont {C.}~\bibnamefont {Sch\"{u}tte}},\
  }\bibfield  {title} {\enquote {\bibinfo {title} {Dimension reduction by
  balanced truncation: Application to light-induced control of open quantum
  systems},}\ }\href {\doibase 10.1063/1.3605243} {\bibfield  {journal}
  {\bibinfo  {journal} {The Journal of Chemical Physics}\ }\textbf {\bibinfo
  {volume} {135}},\ \bibinfo {pages} {014112} (\bibinfo {year}
  {2011})}\BibitemShut {NoStop}%
\bibitem [{\citenamefont {Benner}\ \emph {et~al.}(2020)\citenamefont {Benner},
  \citenamefont {Breiten}, \citenamefont {Hartmann},\ and\ \citenamefont
  {Schmidt}}]{Benner2020}%
  \BibitemOpen
  \bibfield  {author} {\bibinfo {author} {\bibfnamefont {P.}~\bibnamefont
  {Benner}}, \bibinfo {author} {\bibfnamefont {T.}~\bibnamefont {Breiten}},
  \bibinfo {author} {\bibfnamefont {C.}~\bibnamefont {Hartmann}}, \ and\
  \bibinfo {author} {\bibfnamefont {B.}~\bibnamefont {Schmidt}},\ }\bibfield
  {title} {\enquote {\bibinfo {title} {Model reduction of controlled
  {F}okker–{P}lanck and {L}iouville–von {N}eumann equations},}\ }\href
  {\doibase 10.3934/jcd.2020001} {\bibfield  {journal} {\bibinfo  {journal}
  {Journal of Computational Dynamics}\ }\textbf {\bibinfo {volume} {7}},\
  \bibinfo {pages} {1--33} (\bibinfo {year} {2020})}\BibitemShut {NoStop}%
\bibitem [{\citenamefont {McLachlan}\ and\ \citenamefont
  {Ball}(1964)}]{mclachlan1964time}%
  \BibitemOpen
  \bibfield  {author} {\bibinfo {author} {\bibfnamefont {A.~D.}\ \bibnamefont
  {McLachlan}}\ and\ \bibinfo {author} {\bibfnamefont {M.~A.}\ \bibnamefont
  {Ball}},\ }\bibfield  {title} {\enquote {\bibinfo {title} {Time-dependent
  {H}artree-{F}ock theory for molecules},}\ }\href {\doibase
  10.1103/RevModPhys.36.844} {\bibfield  {journal} {\bibinfo  {journal}
  {Reviews of Modern Physics}\ }\textbf {\bibinfo {volume} {36}},\ \bibinfo
  {pages} {844} (\bibinfo {year} {1964})}\BibitemShut {NoStop}%
\bibitem [{\citenamefont {Bhat}, \citenamefont {Ranka},\ and\ \citenamefont
  {Isborn}(2020)}]{bhat2020machine}%
  \BibitemOpen
  \bibfield  {author} {\bibinfo {author} {\bibfnamefont {H.~S.}\ \bibnamefont
  {Bhat}}, \bibinfo {author} {\bibfnamefont {K.}~\bibnamefont {Ranka}}, \ and\
  \bibinfo {author} {\bibfnamefont {C.~M.}\ \bibnamefont {Isborn}},\ }\bibfield
   {title} {\enquote {\bibinfo {title} {Machine learning a molecular
  {H}amiltonian for predicting electron dynamics},}\ }\href {\doibase
  10.1007/s40435-020-00699-8} {\bibfield  {journal} {\bibinfo  {journal}
  {International Journal of Dynamics and Control}\ }\textbf {\bibinfo {volume}
  {8}},\ \bibinfo {pages} {1089--1101} (\bibinfo {year} {2020})}\BibitemShut
  {NoStop}%
\bibitem [{\citenamefont {Gupta}\ \emph {et~al.}(2022)\citenamefont {Gupta},
  \citenamefont {Bhat}, \citenamefont {Ranka},\ and\ \citenamefont
  {Isborn}}]{gupta2022statistical}%
  \BibitemOpen
  \bibfield  {author} {\bibinfo {author} {\bibfnamefont {P.}~\bibnamefont
  {Gupta}}, \bibinfo {author} {\bibfnamefont {H.~S.}\ \bibnamefont {Bhat}},
  \bibinfo {author} {\bibfnamefont {K.}~\bibnamefont {Ranka}}, \ and\ \bibinfo
  {author} {\bibfnamefont {C.~M.}\ \bibnamefont {Isborn}},\ }\bibfield  {title}
  {\enquote {\bibinfo {title} {Statistical learning for predicting
  density-matrix-based electron dynamics},}\ }\href {\doibase 10.1002/sta4.439}
  {\bibfield  {journal} {\bibinfo  {journal} {Stat}\ }\textbf {\bibinfo
  {volume} {11}},\ \bibinfo {pages} {e439} (\bibinfo {year}
  {2022})}\BibitemShut {NoStop}%
\bibitem [{\citenamefont {Coleman}(1963)}]{ColemanA.J.1963SoFD}%
  \BibitemOpen
  \bibfield  {author} {\bibinfo {author} {\bibfnamefont {A.~J.}\ \bibnamefont
  {Coleman}},\ }\bibfield  {title} {\enquote {\bibinfo {title} {Structure of
  fermion density matrices},}\ }\href {\doibase 10.1103/RevModPhys.35.668}
  {\bibfield  {journal} {\bibinfo  {journal} {Reviews of Modern Physics}\
  }\textbf {\bibinfo {volume} {35}},\ \bibinfo {pages} {668--686} (\bibinfo
  {year} {1963})}\BibitemShut {NoStop}%
\bibitem [{\citenamefont {Mazziotti}(1998)}]{PhysRevA.57.4219}%
  \BibitemOpen
  \bibfield  {author} {\bibinfo {author} {\bibfnamefont {D.~A.}\ \bibnamefont
  {Mazziotti}},\ }\bibfield  {title} {\enquote {\bibinfo {title} {Contracted
  {S}chr\"odinger equation: {D}etermining quantum energies and two-particle
  density matrices without wave functions},}\ }\href {\doibase
  10.1103/PhysRevA.57.4219} {\bibfield  {journal} {\bibinfo  {journal} {Phys.
  Rev. A}\ }\textbf {\bibinfo {volume} {57}},\ \bibinfo {pages} {4219--4234}
  (\bibinfo {year} {1998})}\BibitemShut {NoStop}%
\bibitem [{\citenamefont {Mentel}(2015)}]{Mentel2015}%
  \BibitemOpen
  \bibfield  {author} {\bibinfo {author} {\bibfnamefont {L.~M.}\ \bibnamefont
  {Mentel}},\ }\emph {\bibinfo {title} {Reduced Density Matrix inspired
  approaches to electronic structure theory}},\ \href@noop {} {\bibinfo {type}
  {{PhD Thesis}}},\ \bibinfo  {school} {Vrije University}, \bibinfo {address}
  {Amsterdam} (\bibinfo {year} {2015}),\ \bibinfo {note} {available at
  \url{https://research.vu.nl/en/publications/7d553654-df2a-45d7-a460-018f3cbbe127}}\BibitemShut
  {NoStop}%
\bibitem [{\citenamefont {Pernal}, \citenamefont {Gritsenko},\ and\
  \citenamefont {Baerends}(2007)}]{PhysRevA.75.012506}%
  \BibitemOpen
  \bibfield  {author} {\bibinfo {author} {\bibfnamefont {K.}~\bibnamefont
  {Pernal}}, \bibinfo {author} {\bibfnamefont {O.}~\bibnamefont {Gritsenko}}, \
  and\ \bibinfo {author} {\bibfnamefont {E.~J.}\ \bibnamefont {Baerends}},\
  }\bibfield  {title} {\enquote {\bibinfo {title} {Time-dependent
  density-matrix-functional theory},}\ }\href {\doibase
  10.1103/PhysRevA.75.012506} {\bibfield  {journal} {\bibinfo  {journal} {Phys.
  Rev. A}\ }\textbf {\bibinfo {volume} {75}},\ \bibinfo {pages} {012506}
  (\bibinfo {year} {2007})}\BibitemShut {NoStop}%
\bibitem [{\citenamefont {Giesbertz}(2010)}]{Giesbertz2010}%
  \BibitemOpen
  \bibfield  {author} {\bibinfo {author} {\bibfnamefont {K.~J.~H.}\
  \bibnamefont {Giesbertz}},\ }\emph {\bibinfo {title} {Time-Dependent One-Body
  Reduced Density Matrix Functional Theory}},\ \href@noop {} {\bibinfo {type}
  {{PhD Thesis}}},\ \bibinfo  {school} {Vrije University}, \bibinfo {address}
  {Amsterdam} (\bibinfo {year} {2010}),\ \bibinfo {note} {available at
  \url{https://research.vu.nl/en/publications/da4297a6-0542-4c98-bc20-ade87d85e4fd}}\BibitemShut
  {NoStop}%
\bibitem [{\citenamefont {Giesbertz}, \citenamefont {Gritsenko},\ and\
  \citenamefont {Baerends}(2010)}]{10.1063/1.3499601}%
  \BibitemOpen
  \bibfield  {author} {\bibinfo {author} {\bibfnamefont {K.~J.~H.}\
  \bibnamefont {Giesbertz}}, \bibinfo {author} {\bibfnamefont {O.~V.}\
  \bibnamefont {Gritsenko}}, \ and\ \bibinfo {author} {\bibfnamefont {E.~J.}\
  \bibnamefont {Baerends}},\ }\bibfield  {title} {\enquote {\bibinfo {title}
  {{The adiabatic approximation in time-dependent density matrix functional
  theory: Response properties from dynamics of phase-including natural
  orbitals}},}\ }\href {\doibase 10.1063/1.3499601} {\bibfield  {journal}
  {\bibinfo  {journal} {The Journal of Chemical Physics}\ }\textbf {\bibinfo
  {volume} {133}},\ \bibinfo {pages} {174119} (\bibinfo {year}
  {2010})}\BibitemShut {NoStop}%
\bibitem [{\citenamefont {Akbari}\ \emph {et~al.}(2012)\citenamefont {Akbari},
  \citenamefont {Hashemi}, \citenamefont {Rubio}, \citenamefont {Nieminen},\
  and\ \citenamefont {van Leeuwen}}]{AkbariA.2012Citt}%
  \BibitemOpen
  \bibfield  {author} {\bibinfo {author} {\bibfnamefont {A.}~\bibnamefont
  {Akbari}}, \bibinfo {author} {\bibfnamefont {M.~J.}\ \bibnamefont {Hashemi}},
  \bibinfo {author} {\bibfnamefont {A.}~\bibnamefont {Rubio}}, \bibinfo
  {author} {\bibfnamefont {R.~M.}\ \bibnamefont {Nieminen}}, \ and\ \bibinfo
  {author} {\bibfnamefont {R.}~\bibnamefont {van Leeuwen}},\ }\bibfield
  {title} {\enquote {\bibinfo {title} {Challenges in truncating the hierarchy
  of time-dependent reduced density matrices equations},}\ }\href {\doibase
  10.1103/PhysRevB.85.235121} {\bibfield  {journal} {\bibinfo  {journal}
  {Physical Review B}\ }\textbf {\bibinfo {volume} {85}} (\bibinfo {year}
  {2012}),\ 10.1103/PhysRevB.85.235121}\BibitemShut {NoStop}%
\bibitem [{\citenamefont {Ferr\'{e}}, \citenamefont {Filatov},\ and\
  \citenamefont {Huix-Rotllant}(2016)}]{FerreNicolas2016RDMF}%
  \BibitemOpen
  \bibfield  {author} {\bibinfo {author} {\bibfnamefont {N.}~\bibnamefont
  {Ferr\'{e}}}, \bibinfo {author} {\bibfnamefont {M.}~\bibnamefont {Filatov}},
  \ and\ \bibinfo {author} {\bibfnamefont {M.}~\bibnamefont {Huix-Rotllant}},\
  }\bibfield  {title} {\enquote {\bibinfo {title} {{Reduced Density Matrix
  Functional Theory (RDMFT) and Linear Response Time-Dependent RDMFT
  (TD-RDMFT)}},}\ }in\ \href {\doibase 10.1007/128_2015_624} {\emph {\bibinfo
  {booktitle} {Density-Functional Methods for Excited States}}},\ \bibinfo
  {series} {Topics in Current Chemistry}, Vol.\ \bibinfo {volume} {368}\
  (\bibinfo  {publisher} {Springer International Publishing AG},\ \bibinfo
  {address} {Switzerland},\ \bibinfo {year} {2016})\ pp.\ \bibinfo {pages}
  {125--183}\BibitemShut {NoStop}%
\bibitem [{\citenamefont {Herbert}\ and\ \citenamefont
  {Harriman}(2002)}]{HerbertJohnM.2002Cotd}%
  \BibitemOpen
  \bibfield  {author} {\bibinfo {author} {\bibfnamefont {J.~M.}\ \bibnamefont
  {Herbert}}\ and\ \bibinfo {author} {\bibfnamefont {J.~E.}\ \bibnamefont
  {Harriman}},\ }\bibfield  {title} {\enquote {\bibinfo {title} {Comparison of
  two-electron densities reconstructed from one-electron density matrices},}\
  }\href@noop {} {\bibfield  {journal} {\bibinfo  {journal} {International
  Journal of Quantum Chemistry}\ }\textbf {\bibinfo {volume} {90}},\ \bibinfo
  {pages} {355--369} (\bibinfo {year} {2002})}\BibitemShut {NoStop}%
\bibitem [{\citenamefont {Giesbertz}, \citenamefont {Baerends},\ and\
  \citenamefont {Gritsenko}(2008)}]{PhysRevLett.101.033004}%
  \BibitemOpen
  \bibfield  {author} {\bibinfo {author} {\bibfnamefont {K.~J.~H.}\
  \bibnamefont {Giesbertz}}, \bibinfo {author} {\bibfnamefont {E.~J.}\
  \bibnamefont {Baerends}}, \ and\ \bibinfo {author} {\bibfnamefont {O.~V.}\
  \bibnamefont {Gritsenko}},\ }\bibfield  {title} {\enquote {\bibinfo {title}
  {Charge transfer, double and bond-breaking excitations with time-dependent
  density matrix functional theory},}\ }\href {\doibase
  10.1103/PhysRevLett.101.033004} {\bibfield  {journal} {\bibinfo  {journal}
  {Phys. Rev. Lett.}\ }\textbf {\bibinfo {volume} {101}},\ \bibinfo {pages}
  {033004} (\bibinfo {year} {2008})}\BibitemShut {NoStop}%
\bibitem [{\citenamefont {Jeffcoat}\ and\ \citenamefont
  {DePrince}(2014)}]{Jeffcoat2014}%
  \BibitemOpen
  \bibfield  {author} {\bibinfo {author} {\bibfnamefont {D.~B.}\ \bibnamefont
  {Jeffcoat}}\ and\ \bibinfo {author} {\bibfnamefont {I.}~\bibnamefont
  {DePrince}, \bibfnamefont {A.~Eugene}},\ }\bibfield  {title} {\enquote
  {\bibinfo {title} {{N-representability-driven reconstruction of the
  two-electron reduced-density matrix for a real-time time-dependent electronic
  structure method}},}\ }\href {\doibase 10.1063/1.4902757} {\bibfield
  {journal} {\bibinfo  {journal} {The Journal of Chemical Physics}\ }\textbf
  {\bibinfo {volume} {141}},\ \bibinfo {pages} {214104} (\bibinfo {year}
  {2014})}\BibitemShut {NoStop}%
\bibitem [{\citenamefont {Lackner}\ \emph {et~al.}(2015)\citenamefont
  {Lackner}, \citenamefont {B\ifmmode~\check{r}\else \v{r}\fi{}ezinov\'a},
  \citenamefont {Sato}, \citenamefont {Ishikawa},\ and\ \citenamefont
  {Burgd\"orfer}}]{PhysRevA.91.023412}%
  \BibitemOpen
  \bibfield  {author} {\bibinfo {author} {\bibfnamefont {F.}~\bibnamefont
  {Lackner}}, \bibinfo {author} {\bibfnamefont {I.}~\bibnamefont
  {B\ifmmode~\check{r}\else \v{r}\fi{}ezinov\'a}}, \bibinfo {author}
  {\bibfnamefont {T.}~\bibnamefont {Sato}}, \bibinfo {author} {\bibfnamefont
  {K.~L.}\ \bibnamefont {Ishikawa}}, \ and\ \bibinfo {author} {\bibfnamefont
  {J.}~\bibnamefont {Burgd\"orfer}},\ }\bibfield  {title} {\enquote {\bibinfo
  {title} {Propagating two-particle reduced density matrices without wave
  functions},}\ }\href {\doibase 10.1103/PhysRevA.91.023412} {\bibfield
  {journal} {\bibinfo  {journal} {Phys. Rev. A}\ }\textbf {\bibinfo {volume}
  {91}},\ \bibinfo {pages} {023412} (\bibinfo {year} {2015})}\BibitemShut
  {NoStop}%
\bibitem [{\citenamefont {Lackner}\ \emph {et~al.}(2017)\citenamefont
  {Lackner}, \citenamefont {B\ifmmode~\check{r}\else \v{r}\fi{}ezinov\'a},
  \citenamefont {Sato}, \citenamefont {Ishikawa},\ and\ \citenamefont
  {Burgd\"orfer}}]{PhysRevA.95.033414}%
  \BibitemOpen
  \bibfield  {author} {\bibinfo {author} {\bibfnamefont {F.}~\bibnamefont
  {Lackner}}, \bibinfo {author} {\bibfnamefont {I.}~\bibnamefont
  {B\ifmmode~\check{r}\else \v{r}\fi{}ezinov\'a}}, \bibinfo {author}
  {\bibfnamefont {T.}~\bibnamefont {Sato}}, \bibinfo {author} {\bibfnamefont
  {K.~L.}\ \bibnamefont {Ishikawa}}, \ and\ \bibinfo {author} {\bibfnamefont
  {J.}~\bibnamefont {Burgd\"orfer}},\ }\bibfield  {title} {\enquote {\bibinfo
  {title} {High-harmonic spectra from time-dependent two-particle
  reduced-density-matrix theory},}\ }\href {\doibase
  10.1103/PhysRevA.95.033414} {\bibfield  {journal} {\bibinfo  {journal} {Phys.
  Rev. A}\ }\textbf {\bibinfo {volume} {95}},\ \bibinfo {pages} {033414}
  (\bibinfo {year} {2017})}\BibitemShut {NoStop}%
\bibitem [{\citenamefont {Fosso-Tande}\ \emph {et~al.}(2016)\citenamefont
  {Fosso-Tande}, \citenamefont {Nguyen}, \citenamefont {Gidofalvi},\ and\
  \citenamefont {DePrince}}]{Fosso-TandeJacob2016LVTR}%
  \BibitemOpen
  \bibfield  {author} {\bibinfo {author} {\bibfnamefont {J.}~\bibnamefont
  {Fosso-Tande}}, \bibinfo {author} {\bibfnamefont {T.-S.}\ \bibnamefont
  {Nguyen}}, \bibinfo {author} {\bibfnamefont {G.}~\bibnamefont {Gidofalvi}}, \
  and\ \bibinfo {author} {\bibfnamefont {A.~E.}\ \bibnamefont {DePrince}},\
  }\bibfield  {title} {\enquote {\bibinfo {title} {Large-scale variational
  two-electron reduced-density-matrix-driven complete active space
  self-consistent field methods},}\ }\href {\doibase 10.1021/acs.jctc.6b00190}
  {\bibfield  {journal} {\bibinfo  {journal} {Journal of Chemical Theory and
  Computation}\ }\textbf {\bibinfo {volume} {12}},\ \bibinfo {pages}
  {2260--2271} (\bibinfo {year} {2016})}\BibitemShut {NoStop}%
\bibitem [{\citenamefont {Mazziotti}(2023)}]{PhysRevLett.130.153001}%
  \BibitemOpen
  \bibfield  {author} {\bibinfo {author} {\bibfnamefont {D.~A.}\ \bibnamefont
  {Mazziotti}},\ }\bibfield  {title} {\enquote {\bibinfo {title} {Quantum
  many-body theory from a solution of the {$N$}-representability problem},}\
  }\href {\doibase 10.1103/PhysRevLett.130.153001} {\bibfield  {journal}
  {\bibinfo  {journal} {Phys. Rev. Lett.}\ }\textbf {\bibinfo {volume} {130}},\
  \bibinfo {pages} {153001} (\bibinfo {year} {2023})}\BibitemShut {NoStop}%
\bibitem [{\citenamefont {Pan}\ and\ \citenamefont
  {Duraisamy}(2020)}]{pan2020structure}%
  \BibitemOpen
  \bibfield  {author} {\bibinfo {author} {\bibfnamefont {S.}~\bibnamefont
  {Pan}}\ and\ \bibinfo {author} {\bibfnamefont {K.}~\bibnamefont
  {Duraisamy}},\ }\bibfield  {title} {\enquote {\bibinfo {title} {On the
  structure of time-delay embedding in linear models of non-linear dynamical
  systems},}\ }\href {\doibase 10.1063/5.0010886} {\bibfield  {journal}
  {\bibinfo  {journal} {Chaos: An Interdisciplinary Journal of Nonlinear
  Science}\ }\textbf {\bibinfo {volume} {30}} (\bibinfo {year} {2020}),\
  10.1063/5.0010886}\BibitemShut {NoStop}%
\bibitem [{\citenamefont {Hehre}, \citenamefont {Ditchfield},\ and\
  \citenamefont {Pople}(1972)}]{Hehre1972}%
  \BibitemOpen
  \bibfield  {author} {\bibinfo {author} {\bibfnamefont {W.~J.}\ \bibnamefont
  {Hehre}}, \bibinfo {author} {\bibfnamefont {R.}~\bibnamefont {Ditchfield}}, \
  and\ \bibinfo {author} {\bibfnamefont {J.~A.}\ \bibnamefont {Pople}},\
  }\bibfield  {title} {\enquote {\bibinfo {title} {Self-consistent molecular
  orbital methods. {XII.} {F}urther extensions of {G}aussian-type basis sets
  for use in molecular orbital studies of organic molecules},}\ }\href
  {\doibase 10.1063/1.1677527} {\bibfield  {journal} {\bibinfo  {journal} {The
  Journal of Chemical Physics}\ }\textbf {\bibinfo {volume} {56}},\ \bibinfo
  {pages} {2257--2261} (\bibinfo {year} {1972})}\BibitemShut {NoStop}%
\bibitem [{\citenamefont {Roos}, \citenamefont {Taylor},\ and\ \citenamefont
  {Sigbahn}(1980)}]{Roos1980}%
  \BibitemOpen
  \bibfield  {author} {\bibinfo {author} {\bibfnamefont {B.~O.}\ \bibnamefont
  {Roos}}, \bibinfo {author} {\bibfnamefont {P.~R.}\ \bibnamefont {Taylor}}, \
  and\ \bibinfo {author} {\bibfnamefont {P.~E.}\ \bibnamefont {Sigbahn}},\
  }\bibfield  {title} {\enquote {\bibinfo {title} {A complete active space
  {SCF} method ({CASSCF}) using a density matrix formulated super-{CI}
  approach},}\ }\href {\doibase 10.1016/0301-0104(80)80045-0} {\bibfield
  {journal} {\bibinfo  {journal} {Chemical Physics}\ }\textbf {\bibinfo
  {volume} {48}},\ \bibinfo {pages} {157--173} (\bibinfo {year}
  {1980})}\BibitemShut {NoStop}%
\bibitem [{\citenamefont {Sato}\ and\ \citenamefont
  {Ishikawa}(2013)}]{PhysRevA.88.023402}%
  \BibitemOpen
  \bibfield  {author} {\bibinfo {author} {\bibfnamefont {T.}~\bibnamefont
  {Sato}}\ and\ \bibinfo {author} {\bibfnamefont {K.~L.}\ \bibnamefont
  {Ishikawa}},\ }\bibfield  {title} {\enquote {\bibinfo {title} {Time-dependent
  complete-active-space self-consistent-field method for multielectron dynamics
  in intense laser fields},}\ }\href {\doibase 10.1103/PhysRevA.88.023402}
  {\bibfield  {journal} {\bibinfo  {journal} {Phys. Rev. A}\ }\textbf {\bibinfo
  {volume} {88}},\ \bibinfo {pages} {023402} (\bibinfo {year}
  {2013})}\BibitemShut {NoStop}%
\bibitem [{\citenamefont {Olsen}\ \emph {et~al.}(1988)\citenamefont {Olsen},
  \citenamefont {Roos}, \citenamefont {J{\o}rgensen},\ and\ \citenamefont
  {Jensen}}]{olsen1988determinant}%
  \BibitemOpen
  \bibfield  {author} {\bibinfo {author} {\bibfnamefont {J.}~\bibnamefont
  {Olsen}}, \bibinfo {author} {\bibfnamefont {B.~O.}\ \bibnamefont {Roos}},
  \bibinfo {author} {\bibfnamefont {P.}~\bibnamefont {J{\o}rgensen}}, \ and\
  \bibinfo {author} {\bibfnamefont {H.~J.~A.}\ \bibnamefont {Jensen}},\
  }\bibfield  {title} {\enquote {\bibinfo {title} {Determinant based
  configuration interaction algorithms for complete and restricted
  configuration interaction spaces},}\ }\href {\doibase 10.1063/1.455063}
  {\bibfield  {journal} {\bibinfo  {journal} {The Journal of Chemical Physics}\
  }\textbf {\bibinfo {volume} {89}},\ \bibinfo {pages} {2185--2192} (\bibinfo
  {year} {1988})}\BibitemShut {NoStop}%
\bibitem [{\citenamefont {Peng}, \citenamefont {Fales},\ and\ \citenamefont
  {Levine}(2018)}]{peng2018simulating}%
  \BibitemOpen
  \bibfield  {author} {\bibinfo {author} {\bibfnamefont {W.-T.}\ \bibnamefont
  {Peng}}, \bibinfo {author} {\bibfnamefont {B.~S.}\ \bibnamefont {Fales}}, \
  and\ \bibinfo {author} {\bibfnamefont {B.~G.}\ \bibnamefont {Levine}},\
  }\bibfield  {title} {\enquote {\bibinfo {title} {Simulating electron dynamics
  of complex molecules with time-dependent complete active space configuration
  interaction},}\ }\href {\doibase 10.1021/acs.jctc.8b00381} {\bibfield
  {journal} {\bibinfo  {journal} {Journal of Chemical Theory and Computation}\
  }\textbf {\bibinfo {volume} {14}},\ \bibinfo {pages} {4129--4138} (\bibinfo
  {year} {2018})}\BibitemShut {NoStop}%
\bibitem [{\citenamefont {Zwanzig}(1964)}]{zwanzig1964identity}%
  \BibitemOpen
  \bibfield  {author} {\bibinfo {author} {\bibfnamefont {R.}~\bibnamefont
  {Zwanzig}},\ }\bibfield  {title} {\enquote {\bibinfo {title} {On the identity
  of three generalized master equations},}\ }\href {\doibase
  10.1016/0031-8914(64)90102-8} {\bibfield  {journal} {\bibinfo  {journal}
  {Physica}\ }\textbf {\bibinfo {volume} {30}},\ \bibinfo {pages} {1109--1123}
  (\bibinfo {year} {1964})}\BibitemShut {NoStop}%
\bibitem [{\citenamefont {Ohtsuki}\ and\ \citenamefont
  {Fujimura}(1989)}]{ohtsuki1989bath}%
  \BibitemOpen
  \bibfield  {author} {\bibinfo {author} {\bibfnamefont {Y.}~\bibnamefont
  {Ohtsuki}}\ and\ \bibinfo {author} {\bibfnamefont {Y.}~\bibnamefont
  {Fujimura}},\ }\bibfield  {title} {\enquote {\bibinfo {title} {Bath-induced
  vibronic coherence transfer effects on femtosecond time-resolved resonant
  light scattering spectra from molecules},}\ }\href {\doibase
  10.1063/1.456822} {\bibfield  {journal} {\bibinfo  {journal} {The Journal of
  Chemical Physics}\ }\textbf {\bibinfo {volume} {91}},\ \bibinfo {pages}
  {3903--3915} (\bibinfo {year} {1989})}\BibitemShut {NoStop}%
\bibitem [{\citenamefont {May}\ and\ \citenamefont
  {K{\"{u}}hn}(2023)}]{May2023}%
  \BibitemOpen
  \bibfield  {author} {\bibinfo {author} {\bibfnamefont {V.}~\bibnamefont
  {May}}\ and\ \bibinfo {author} {\bibfnamefont {O.}~\bibnamefont
  {K{\"{u}}hn}},\ }\href {\doibase 10.1002/9783527633791} {\emph {\bibinfo
  {title} {Charge and Energy Transfer Dynamics in Molecular Systems}}},\
  \bibinfo {edition} {4th}\ ed.\ (\bibinfo  {publisher} {Wiley-VCH},\ \bibinfo
  {address} {Weinheim, Germany},\ \bibinfo {year} {2023})\BibitemShut {NoStop}%
\bibitem [{\citenamefont {Zhang}(2005)}]{Zhang2005}%
  \BibitemOpen
  \bibfield  {author} {\bibinfo {author} {\bibfnamefont {F.}~\bibnamefont
  {Zhang}},\ }\href {\doibase 10.1007/b105056} {\emph {\bibinfo {title} {The
  Schur Complement and Its Applications}}},\ \bibinfo {edition} {{F}irst}\ ed.\
  (\bibinfo  {publisher} {Springer US},\ \bibinfo {address} {New York, NY},\
  \bibinfo {year} {2005})\BibitemShut {NoStop}%
\bibitem [{Note1()}]{Note1}%
  \BibitemOpen
  \bibinfo {note} {\protect \url
  {https://docs.nersc.gov/systems/perlmutter/architecture/}}\BibitemShut
  {NoStop}%
\bibitem [{Note2()}]{Note2}%
  \BibitemOpen
  \bibinfo {note} {\protect \url
  {https://github.com/hbassi/1rdm_memory_model}}\BibitemShut {NoStop}%
\bibitem [{\citenamefont {Bradbury}\ \emph {et~al.}(2018)\citenamefont
  {Bradbury}, \citenamefont {Frostig}, \citenamefont {Hawkins}, \citenamefont
  {Johnson}, \citenamefont {Leary}, \citenamefont {Maclaurin}, \citenamefont
  {Necula}, \citenamefont {Paszke}, \citenamefont {Vander{P}las}, \citenamefont
  {Wanderman-{M}ilne},\ and\ \citenamefont {Zhang}}]{jax2018github}%
  \BibitemOpen
  \bibfield  {author} {\bibinfo {author} {\bibfnamefont {J.}~\bibnamefont
  {Bradbury}}, \bibinfo {author} {\bibfnamefont {R.}~\bibnamefont {Frostig}},
  \bibinfo {author} {\bibfnamefont {P.}~\bibnamefont {Hawkins}}, \bibinfo
  {author} {\bibfnamefont {M.~J.}\ \bibnamefont {Johnson}}, \bibinfo {author}
  {\bibfnamefont {C.}~\bibnamefont {Leary}}, \bibinfo {author} {\bibfnamefont
  {D.}~\bibnamefont {Maclaurin}}, \bibinfo {author} {\bibfnamefont
  {G.}~\bibnamefont {Necula}}, \bibinfo {author} {\bibfnamefont
  {A.}~\bibnamefont {Paszke}}, \bibinfo {author} {\bibfnamefont
  {J.}~\bibnamefont {Vander{P}las}}, \bibinfo {author} {\bibfnamefont
  {S.}~\bibnamefont {Wanderman-{M}ilne}}, \ and\ \bibinfo {author}
  {\bibfnamefont {Q.}~\bibnamefont {Zhang}},\ }\href@noop {} {\enquote
  {\bibinfo {title} {{JAX}: composable transformations of {P}ython+{N}um{P}y
  programs},}\ } (\bibinfo {year} {2018}),\ \bibinfo {note}
  {\url{http://github.com/google/jax}}\BibitemShut {NoStop}%
\bibitem [{\citenamefont {Okuta}\ \emph {et~al.}(2017)\citenamefont {Okuta},
  \citenamefont {Unno}, \citenamefont {Nishino}, \citenamefont {Hido},\ and\
  \citenamefont {Loomis}}]{cupy_learningsys2017}%
  \BibitemOpen
  \bibfield  {author} {\bibinfo {author} {\bibfnamefont {R.}~\bibnamefont
  {Okuta}}, \bibinfo {author} {\bibfnamefont {Y.}~\bibnamefont {Unno}},
  \bibinfo {author} {\bibfnamefont {D.}~\bibnamefont {Nishino}}, \bibinfo
  {author} {\bibfnamefont {S.}~\bibnamefont {Hido}}, \ and\ \bibinfo {author}
  {\bibfnamefont {C.}~\bibnamefont {Loomis}},\ }\bibfield  {title} {\enquote
  {\bibinfo {title} {{CuPy: A NumPy-Compatible Library for NVIDIA GPU
  Calculations}},}\ }in\ \href
  {http://learningsys.org/nips17/assets/papers/paper_16.pdf} {\emph {\bibinfo
  {booktitle} {Proceedings of Workshop on Machine Learning Systems
  (LearningSys) in The Thirty-first Annual Conference on Neural Information
  Processing Systems}}}\ (\bibinfo {year} {2017})\BibitemShut {NoStop}%
\bibitem [{\citenamefont {Kurzweil}\ and\ \citenamefont
  {Baer}(2004)}]{kurzweil2004time}%
  \BibitemOpen
  \bibfield  {author} {\bibinfo {author} {\bibfnamefont {Y.}~\bibnamefont
  {Kurzweil}}\ and\ \bibinfo {author} {\bibfnamefont {R.}~\bibnamefont
  {Baer}},\ }\bibfield  {title} {\enquote {\bibinfo {title} {Time-dependent
  exchange-correlation current density functionals with memory},}\ }\href
  {\doibase 10.1063/1.1802793} {\bibfield  {journal} {\bibinfo  {journal} {The
  Journal of Chemical Physics}\ }\textbf {\bibinfo {volume} {121}},\ \bibinfo
  {pages} {8731--8741} (\bibinfo {year} {2004})}\BibitemShut {NoStop}%
\bibitem [{\citenamefont {Hirata}, \citenamefont {Head-Gordon},\ and\
  \citenamefont {Bartlett}(1999)}]{hirata1999configuration}%
  \BibitemOpen
  \bibfield  {author} {\bibinfo {author} {\bibfnamefont {S.}~\bibnamefont
  {Hirata}}, \bibinfo {author} {\bibfnamefont {M.}~\bibnamefont {Head-Gordon}},
  \ and\ \bibinfo {author} {\bibfnamefont {R.~J.}\ \bibnamefont {Bartlett}},\
  }\bibfield  {title} {\enquote {\bibinfo {title} {Configuration interaction
  singles, time-dependent {H}artree-{F}ock, and time-dependent density
  functional theory for the electronic excited states of extended systems},}\
  }\href {\doibase 10.1063/1.480443} {\bibfield  {journal} {\bibinfo  {journal}
  {The Journal of Chemical Physics}\ }\textbf {\bibinfo {volume} {111}},\
  \bibinfo {pages} {10774--10786} (\bibinfo {year} {1999})}\BibitemShut
  {NoStop}%
\bibitem [{\citenamefont {Foresman}\ \emph {et~al.}(1992)\citenamefont
  {Foresman}, \citenamefont {Head-Gordon}, \citenamefont {Pople},\ and\
  \citenamefont {Frisch}}]{foresman1992toward}%
  \BibitemOpen
  \bibfield  {author} {\bibinfo {author} {\bibfnamefont {J.~B.}\ \bibnamefont
  {Foresman}}, \bibinfo {author} {\bibfnamefont {M.}~\bibnamefont
  {Head-Gordon}}, \bibinfo {author} {\bibfnamefont {J.~A.}\ \bibnamefont
  {Pople}}, \ and\ \bibinfo {author} {\bibfnamefont {M.~J.}\ \bibnamefont
  {Frisch}},\ }\bibfield  {title} {\enquote {\bibinfo {title} {Toward a
  systematic molecular orbital theory for excited states},}\ }\href {\doibase
  10.1021/j100180a030} {\bibfield  {journal} {\bibinfo  {journal} {The Journal
  of Physical Chemistry}\ }\textbf {\bibinfo {volume} {96}},\ \bibinfo {pages}
  {135--149} (\bibinfo {year} {1992})}\BibitemShut {NoStop}%
\bibitem [{\citenamefont {Chicone}(2024)}]{Chicone}%
  \BibitemOpen
  \bibfield  {author} {\bibinfo {author} {\bibfnamefont {C.}~\bibnamefont
  {Chicone}},\ }\href@noop {} {\emph {\bibinfo {title} {Ordinary Differential
  Equations with Applications}}},\ Texts in Applied Mathematics\ (\bibinfo
  {publisher} {Springer},\ \bibinfo {address} {New York},\ \bibinfo {year}
  {2024})\BibitemShut {NoStop}%
\end{thebibliography}%

\end{document}